\pdfoutput=1

\documentclass{amsart}
\usepackage[utf8]{inputenc}
\usepackage[T1]{fontenc}

\usepackage{amssymb, mathabx}
\usepackage{braket}
\usepackage{mathrsfs}
\usepackage{ifthen}
\usepackage{here}
\usepackage{todonotes}
\usepackage{tikz}
\usepackage{mathtools}
\usetikzlibrary{patterns,decorations.pathreplacing,calligraphy}
\usepackage{comment} 
\usepackage{mleftright}

\usepackage{amsmath,amsthm,verbatim,amsfonts, graphicx,booktabs,array}

\usepackage[dvipsnames]{xcolor}
\usepackage[pagebackref,hypertexnames=false,colorlinks=true,
            linkcolor=NavyBlue,
            citecolor=NavyBlue,
            urlcolor=NavyBlue]{hyperref} 
\usepackage{cleveref}
 \usepackage[all]{xy}
 \usepackage{amscd}
 \usepackage[alphabetic,backrefs,msc-links]{amsrefs}
 \usepackage{color}
 \usepackage{enumitem}
\newlist{steps}{enumerate}{1}
\setlist[steps, 1]{label = Step \arabic*:}
\usepackage[abs]{overpic}
\usepackage{tikz-cd}
\usepackage{pinlabel}

\usepackage{geometry}
\makeatletter
\DeclareRobustCommand\widecheck[1]{{\mathpalette\@widecheck{#1}}}
\def\@widecheck#1#2{%
   \setbox\z@\hbox{\m@th$#1#2$}%
   \setbox\tw@\hbox{\m@th$#1%
      {%
         \vrule\@width\z@\@height\ht\z@
         \vrule\@height\z@\@width\wd\z@}$}%
   \dp\tw@-\ht\z@
   \@tempdima\ht\z@ \advance\@tempdima2\ht\tw@ \divide\@tempdima\thr@@
   \setbox\tw@\hbox{%
      \raise\@tempdima\hbox{\scalebox{1}[-1]{\lower\@tempdima\box\tw@}}}%
   {\ooalign{\box\tw@ \cr \box\z@}}}
   \let\@wraptoccontribs\wraptoccontribs
\makeatother

\theoremstyle{plain}
\newtheorem*{theorem*}{Theorem}
\newtheorem{thm}{Theorem}[section]
\crefname{thm}{Theorem}{Theorems}
\Crefname{thm}{Theorem}{Theorems}
\newtheorem{prop}[thm]{Proposition}
\crefname{prop}{Proposition}{Propositions}
\Crefname{prop}{Proposition}{Propositions}
\newtheorem{lem}[thm]{Lemma}
\crefname{lem}{Lemma}{Lemmas}
\Crefname{lem}{Lemma}{Lemmas}
\newtheorem{cor}[thm]{Corollary}
\crefname{cor}{Corollary}{Corollaries}
\Crefname{cor}{Corollary}{Corollaries}
\newtheorem{rem}[thm]{Remark}
\crefname{rem}{Remark}{Remarks}
\Crefname{rem}{Remark}{Remarks}

\crefname{claim}{Claim}{Claims}
\Crefname{claim}{Claim}{Claims}

\crefname{property}{Property}{Properties}
\Crefname{property}{Property}{Properties}

\crefname{problem}{Problem}{Problems}
\Crefname{problem}{Problem}{Problems}

\crefname{conjecture}{Conjecture}{Conjecture}
\Crefname{conjecture}{Conjecture}{Conjecture}

\theoremstyle{definition}
\newtheorem{defn}[thm]{Definition}
\crefname{defn}{Definition}{Definitions}
\Crefname{defn}{Definition}{Definitions}

\crefname{notation}{Notation}{Notations}
\Crefname{notation}{Notation}{Notations}

\crefname{convention}{Convention}{Conventions}
\Crefname{convention}{Convention}{Conventions}
\crefname{cond}{Condition}{Conditions}
\Crefname{cond}{Condition}{Conditions}

\crefname{assum}{Assumption}{Assumptions}
\Crefname{assum}{Assumption}{Assumptions}

\Crefname{ques}{Question}{Question}

\theoremstyle{remark}
\crefname{rem}{Remark}{Remarks}
\Crefname{rem}{Remark}{Remarks}
\newtheorem{ex}[thm]{Example}
\crefname{ex}{Example}{Examples}
\Crefname{ex}{Example}{Examples}

\crefname{section}{Section}{Sections}
\Crefname{section}{Section}{Sections}
\crefname{subsection}{Subsection}{Subsections}
\Crefname{subsection}{Subsection}{Subsections}
\crefname{figure}{Figure}{Figures}
\Crefname{figure}{Figure}{Figures}

\AddToHook{env/thm/begin}{\crefalias{thm}{theorem}}
\AddToHook{env/defn/begin}{\crefalias{thm}{definition}}
\AddToHook{env/lem/begin}{\crefalias{thm}{lemma}}
\AddToHook{env/prop/begin}{\crefalias{thm}{proposition}}
\AddToHook{env/rem/begin}{\crefalias{thm}{remark}}
\AddToHook{env/ex/begin}{\crefalias{thm}{example}}

\newcommand{\diracpartial}{{\partial\mkern-9.5mu/}}
\newcommand{\dirac}{{\mathcal{D}\mkern-11mu/}}
\newcommand{\Z}{\mathbb{Z}}

\newcommand{\Q}{\mathbb{Q}}

\newcommand{\sign}{\mathrm{sign}}

\newcommand{\fraks}{\mathfrak{s}}

\newcommand{\C}{\mathbb{C}}
\newcommand{\s}{\mathfrak{s}}

\newcommand{\R}{\mathbb R}

\def\det{\mathrm{det}}

\def\dim{\mathrm{dim}}

\def\id{\mathrm{Id}}

\def\ind{\mathrm{ind}}

\newcommand{\mbar}[1]{{\ooalign{\hfil#1\hfil\crcr\raise.167ex\hbox{--}}}}

\def\wt{\widetilde}

\def\Ker{\mathrm{Ker}\,}
\def\Cok{\mathrm{Coker}\,}

     \RequirePackage{rotating}                   
    \def\HMt{%
       \setbox0=\hbox{$\widehat{\mathit{HM}}$}
       \setbox1=\hbox{$\mathit{HM}$}
       \dimen0=1.1\ht0
       \advance\dimen0 by 1.17\ht1
       \smash{\mskip2mu\raise\dimen0\rlap{%
          \begin{turn}{180}
              {$\widehat{\phantom{\mathit{HM}}}$}
           \end{turn}} \mskip-2mu    
                \mathit{HM}
                    }{\vphantom{\widehat{\mathit{HM}}}}{}}

\title[Equivariant lattice homology for Seifert homology 3-spheres ]{
Equivariant lattice homology for Seifert homology 3-spheres 
}

\author{Sungkyung Kang}
\address{Department of Pure Mathematics and Mathematical Statistics, University of Cambridge, United Kingdom}
\email{sungkyung3838@gmail.com}

\author{JungHwan Park}
\address{Department of Mathematical Sciences and RIM, Seoul National University, Republic of Korea}
\email{jungpark0817@snu.ac.kr}

\author{Masaki Taniguchi} 
\address{Department of Mathematics, Kyoto University, Japan}
\email{taniguchi.masaki.7m@kyoto-u.ac.jp}

\begin{document}

\begin{abstract}
We develop an $S^1\times\mathbb{Z}_p$-equivariant refinement of the lattice homotopy type for Seifert homology $3$-spheres equipped with the Seifert $\mathbb{Z}_p$-action. We construct an equivariant comparison map to the $S^1\times\mathbb{Z}_p$-equivariant Seiberg--Witten Floer spectrum that becomes an $S^1$-equivariant stable homotopy equivalence after forgetting the $\mathbb{Z}_p$-action, yielding a combinatorial description of its Borel equivariant cohomology.
\end{abstract}


\maketitle

\setcounter{tocdepth}{1}

\section{Introduction}

Manolescu's Seiberg--Witten Floer stable homotopy type~\cite{Man03} refines monopole Floer homology~\cite{KM07} by retaining stable homotopy-theoretic information coming from the Seiberg--Witten equations. For negative-definite plumbed $3$-manifolds, Dai--Sasahira--Stoffregen constructed a lattice homotopy type and related it to the Seiberg--Witten Floer spectrum~\cite{dai2023lattice}. Their construction gives a spectrum-level refinement of the combinatorial description of Floer homology arising from N\'emethi's lattice-theoretic approach~\cite{nemethi2005ozsvath,nemethi2008lattice}.

The purpose of this paper is to develop an equivariant version of the theory for Seifert fibered $3$-manifolds. Let $p$ be a prime, and let $\Z_p$ act on a Seifert homology $3$-sphere $Y$ as a subgroup of the Seifert $S^1$-action.  We construct an $S^1\times\Z_p$-equivariant lattice homotopy type together with an equivariant comparison map to the corresponding Seiberg--Witten Floer spectrum that becomes an $S^1$-equivariant stable homotopy equivalence after forgetting the $\Z_p$-action. The construction is based on an equivariant refinement of N\'emethi's computation sequence. The additional equivariant information is encoded by a sequence of $\Z_p$-equivariant $\mathrm{Spin}^c$ structures on a negative-definite plumbing of $Y$ and by their equivariant Dirac indices. To control these indices, we use Donnelly's equivariant Atiyah--Patodi--Singer index theorem~\cite{donnelly1978eta}.

Equivariant Dirac indices refine ordinary integer-valued indices to virtual $\Z_p$-representations, which determine the representation spheres in the lattice model. Equivariant adjunction relations for relative Bauer--Furuta maps provide the compatible homotopies needed to assemble the comparison map. The resulting lattice model also gives an explicit cochain model, with differentials expressed in terms of equivariant Euler classes, making the Borel equivariant cohomology and its associated equivariant Fr\o yshov type invariants accessible to direct computation.

We summarize the main comparison theorem as follows.

\begin{thm}[\Cref{thm: eqv lattice comparison map}]
\label{main thm:computation}
Let $Y$ be a Seifert homology $3$-sphere equipped with the $\Z_p$-subaction of its Seifert $S^1$-action. Then there is a combinatorial algorithm which produces an equivariant $\mathrm{Spin}^c$ computation sequence
\[
\gamma=(\mathfrak{s}_1,\mathfrak{s}_2,\ldots)
\]
on a negative-definite plumbing $W_\Gamma$ of $Y$ and an associated $S^1\times\Z_p$-equivariant lattice homotopy type
\[
\mathcal{H}_{S^1\times\Z_p}(\Gamma,\gamma).
\]
After an overall normalization of the equivariant grading, there is an $S^1\times\Z_p$-equivariant stable map
\[
\mathcal{T}_{S^1\times\Z_p}\colon \mathcal{H}_{S^1\times\Z_p}(\Gamma,\gamma)\longrightarrow SWF_{S^1\times\Z_p}(Y)
\]
which is an $S^1$-equivariant stable homotopy equivalence. After forgetting the $\Z_p$-action, the construction recovers the ordinary $S^1$-equivariant lattice homotopy model.
\end{thm}

Taking Borel equivariant cohomology gives the following corollary.

\begin{cor}\label{cor: combinatorial equivariant Floer cohomology}
Under the assumptions of \Cref{main thm:computation}, the $S^1\times\Z_p$-equivariant Seiberg--Witten Floer cohomology of $Y$ is determined combinatorially by the equivariant $\mathrm{Spin}^c$ computation sequence $\gamma$. More precisely, after an overall grading shift and the coefficient-ring reparametrization induced by a character twist of the chosen equivariant $\mathrm{Spin}^c$ lift, there is an isomorphism
\[
\widetilde{H}^{*}_{S^1\times\Z_p}\bigl(SWF_{S^1\times\Z_p}(Y);\Z_p\bigr)
\cong
\widetilde{H}^{*}_{S^1\times\Z_p}\bigl(\mathcal{H}_{S^1\times\Z_p}(\Gamma,\gamma);\Z_p\bigr)
\]
of modules over $H^*\bigl(B(S^1\times\Z_p);\Z_p\bigr)$. In particular, the left-hand side can be computed explicitly from the combinatorial data of the equivariant computation sequence.
\end{cor}

Baraglia and Hekmati introduced equivariant Seiberg--Witten Floer cohomology and the associated equivariant Fr\o yshov invariants in~\cite{Baraglia-Hekmati:2024-1}, and applied these invariants to cyclic actions on Brieskorn spheres in~\cite{baraglia2024brieskorn}. Our $S^1\times\Z_p$-equivariant Floer spectrum recovers their equivariant Floer cohomology after taking Borel equivariant cohomology. Thus \Cref{cor: combinatorial equivariant Floer cohomology} gives, in particular, a combinatorial algorithm for computing the Baraglia--Hekmati theory and its equivariant Fr\o yshov invariants for Seifert homology $3$-spheres.

The same combinatorial model can also be used for Seifert fibered rational homology spheres. As an application, we prove in \Cref{subsec: large p Froyshov} that, for any Seifert fibered rational homology sphere $Y$ with the $\Z_p$-subaction of the Seifert $S^1$-action and any $\Z_p$-invariant $\mathrm{Spin}^c$ structure $\mathfrak{s}$ on $Y$, one has
\[
\delta^{(p)}_0(Y,\mathfrak{s})
=
\delta(Y,\mathfrak{s})
+
\dim_{\Z_p}HF_{\mathrm{red}}(Y,\mathfrak{s};\Z_p)
\]
for every sufficiently large prime $p$. Thus, for sufficiently large $p$, the difference between the equivariant and ordinary Fr\o yshov invariants detects the rank of the reduced Floer homology.

A further ingredient of the paper is a metric-independent definition of the $S^1\times\Z_p$-equivariant Seiberg--Witten Floer homotopy type for a $\Z_p$-equivariant $\mathrm{Spin}^c$ rational homology $3$-sphere $Y$ with $p\nmid |H_1(Y;\Z)|$; see \Cref{subsec:S1xZp-SWF-homotopy-type}. Our construction is closely related to the equivariant Floer spectra used by Montague~\cite{montague2022seiberg}, but is formulated in the $\mathrm{Spin}^c$ setting.

This framework also naturally gives rise to an orbifold invariant of knots. For a knot $K\subset S^3$, a prime $p$, and an element $[k]\in\Z_p$, we define an $S^1$-equivariant pointed spectrum
\[
SWF_{\mathrm{ofd}}^{(p),[k]}(K)
\]
as an appropriate fixed-point spectrum associated with the $p$-fold cyclic branched cover of $S^3$ along $K$. We show that this is an invariant of the knot $K$. More generally, when a global $p$-fold branched cover of a properly embedded surface carries compatible equivariant $\mathrm{Spin}^c$ data and satisfies the hypotheses of the relative Bauer--Furuta construction, taking the appropriate fixed-point sector gives a corresponding cobordism map; see \Cref{rem:ofd,rem:ofd_bf}. 

\subsection*{Acknowledgements} 
The authors would like to thank Matthew Stoffregen, Imogen Montague, Kristen Hendricks, David Baraglia, Nobuhiro Nakamura, Mike Miller Eismeier, and Haochen Qiu for very helpful discussions. The authors also thank Jianfeng Lin for explaining an alternative proof of a version of \Cref{thm: Froyshov strict inequality} based on a spectral sequence argument. 

The first author is partially supported by the Royal Society University Research Fellowship URF\textbackslash R1\textbackslash 251501. 
The second author is partially supported by the Samsung Science and Technology Foundation (SSTF-BA2102-02) and by National Research Foundation of Korea (NRF) grants funded by the Korean government (MSIT) (RS-2025-00542968 and RS-2026-25506097).
The third author was partially supported
by JSPS KAKENHI Grant Number 22K13921 and RIKEN iTHEMS Program.

\section{Background materials}

\subsection{Notations} \label{subsec: notations}

Throughout the paper, we unify the notation as follows:
\begin{itemize}
    \item Every Seifert fibered rational homology sphere $Y$ is given the boundary orientation induced by the specified normalized negative-definite Seifert plumbing $W_\Gamma$, so that $W_\Gamma$ is a cobordism from $\emptyset$ to $Y$.
    \item All tensor products of dgas, dg-modules, $A_\infty$-modules, bimodules, and $E_\infty$-modules are derived tensor products unless explicitly mentioned otherwise.
    \item For a topological space $X$, we will sometimes replace its singular cochain complex $C^\ast(X)$ by the \emph{normalized} singular cochain complex, namely the subcomplex of cochains vanishing on degenerate singular simplices, equivalently the dual of the normalized singular chain complex.
    \item $\zeta_p = e^{\frac{2\pi i}{p}} \in \C$ denotes the primitive $p$-th root of unity.
    \item The geometric $\Z_p$-action on the 3-manifold $Y$ is denoted by $\tau \colon Y \to  Y$.
    \item For a given Spin or $\mathrm{Spin}^c$ structure $\fraks$, we write $P(\fraks)$ for the corresponding principal bundle. Denote by $S$ the spinor bundle for $\s$ and 
    by $\dirac_A : \Gamma(S^+) \to \Gamma(S^-)$ the plus part of the 4-dimensional Dirac operator for a fixed $\mathrm{Spin}^c$ connection $A$ and $S=S^+ \oplus S^-$. The notation $\diracpartial_B: \Gamma(S)\to \Gamma(S)$ denotes the 3-dimensional Dirac operator for a $\mathrm{Spin}^c$ connection $B$.
    \item $\tilde{\tau}$ denotes a lift of $\tau$ to the principal or spinor bundle. A $\mathrm{Spin}^c$ structure with such a lift is written as $\tilde{\fraks}$.
    \item $\R$ and $\C$ denote the trivial and the standard representations of $S^1$, respectively.
    \item $\widetilde{\C}_{[i]}$ denotes the representation of $\mathrm{Pin}(2) \times \Z_p$, where $\Z_p$ acts by $\frac{i}{p}$-fold rotation and $j$ acts by $-1$.
    \item ${\C}_{[i]}$ denotes the representation of $S^1 \times \Z_p$, where $\Z_p$ acts by $\frac{i}{p}$-fold rotation and $S^1$ acts in the standard way.
    
    \item For a compact Lie group $G$, we write $R(G)$ and $RO(G)$ for the complex and real representation rings of $G$. 
    \item For a $G$-vector space $V$, we denote by $V^+$ the $ G$-sphere obtained as the one point compactification. 
    \item The subsets $RQ(G)_{\ge 0}$, $R(G)_{\geq 0}$, and $RO(G)_{\geq 0}$ denote the classes represented by actual quaternionic, complex, and real $G$-representations, respectively.
    \item The notation for equivariant cohomology is
    \[
    H^\ast(B(S^1 \times \Z_p); \Z_p) \cong 
    \begin{cases}
        \Z_p[U, \theta] & \text{if } p = 2, \\
        \Z_p[U, R, S]/(R^2) & \text{if } p > 2,
    \end{cases}
    \]
    where $U$ and $\theta$ are the degree two and one elements coming from the generators of $H^\ast(BS^1; \Z_p)$ and $H^\ast(B\Z_2; \Z_2)$ respectively, and $R$ and $S$ are degree one and two elements generating $H^\ast(B\Z_p; \Z_p)$.
    \item For a $\mathrm{Spin}^c$ $4$-manifold $(X, \mathfrak{s})$ with boundary $Y$, equipped with a Riemannian metric $g$ which is a product metric $dt^2 + g_Y$ near the boundary and with a $\mathrm{Spin}^c$ connection $A_0$ that is flat near the boundary, we write the $\mathrm{Spin}^c$ Dirac index with Atiyah--Patodi--Singer boundary conditions as 
\[
\ind^{\mathrm{APS}} \dirac_{X, \mathfrak{s}, A_0, g} \in \mathbb{Z}.
\]
The topological part of the index is defined by 
\[
\ind^{t} \dirac_{X, \mathfrak{s}} := \ind^{\mathrm{APS}} \dirac_{X, \mathfrak{s}, A_0, g} - n(Y, A_0|_Y, g, \mathfrak{s}|_Y) \in \mathbb{Q},
\]
which is independent of the choices of Riemannian metrics and connections, where $n(Y, A_0|_Y, g, \mathfrak{s}|_Y)$ denotes Manolescu's correction term introduced in \cite{Man03}. In other words,
\[
\ind^{t} \dirac_{X, \mathfrak{s}} = \frac{1}{8} \big( c_1(\mathfrak{s})^2 - \sigma(X) \big).
\]
\item Suppose $X$ has a smooth $\Z_p$-action, preserving the connected components of $\partial X$. If such $(X,\s)$ lifts to a $\Z_p$-equivariant $\mathrm{Spin}^c$ structure and $A_0$ and $g$ are taken so that $\Z_p$-invariant, we write the $\Z_p$-equivariant $\mathrm{Spin}^c$ Dirac index with Atiyah--Patodi--Singer boundary conditions as 
\[
\ind^{\mathrm{APS}}_{\Z_p} \dirac_{X, \s, A_0, g} \in R(\Z_p). 
\]
For given element $[i] \in\Z_p$, the trace index is written as 
\[
\ind^{\mathrm{APS}}_{[i]} \dirac_{X, \s, A_0, g} := \operatorname{Tr}_{[i]} \left( \ind^{\mathrm{APS}}_{\Z_p} \dirac_{X, \s, A_0, g}\right) \in \C. 
\]
In \Cref{appendix A}, we introduce their topological parts 
\[
\ind^{t}_{\Z_p} \dirac_{X, \s} \in R(\Z_p)\otimes \Q \text{ and } \ind^{t}_{[i]} \dirac_{X, \s} \in \C. 
\]
\end{itemize}

Since we will be doing stable equivariant homotopy theory throughout the paper, we fix, for every prime $p$, the real universe
\[
\mathcal U_p
=
\mathbb R[\Z_p]^\infty
\oplus
\bigoplus_{i=0}^{p-1}\C_{[i]}^\infty .
\]
Here $S^1$ acts trivially on the first summand and with its standard
weight on every complex summand.  In particular, the case $p=2$ is
included in this convention.  Whenever a $\mathrm{Pin}(2)\times\Z_2$
universe is used in the even-spin-lift case, we also include infinitely
many trivial real summands.

Let $t=[\C_{[1]}]$.  Then the complex representation ring of
$\Z_p$ is
\[
R(\Z_p)
\cong
\Z[t]/(t^p-1)
\cong
\Z[\Z_p].
\]
For any element $h\in\Z_p$, we write the corresponding character
evaluation as
\[
\operatorname{Tr}_h\colon R(\Z_p)\longrightarrow \Z[\zeta_p].
\]
In particular, evaluation at a generator of $\Z_p$ is a ring
homomorphism to $\Z[\zeta_p]$, but it is not an isomorphism.
The augmentation maps on the $KO,K$ theories, as well as the representation rings $RO,R$, are written as $\alpha_{\R}, \alpha_\C$ respectively.  
Note that a given element $[V]\in R(\Z_p)$ is recovered from its traces by the formula  
\[
V = \sum_{l=0}^{p-1} \left( \frac{1}{p} \sum_{k \in \Z_p} \operatorname{Tr}_{\zeta_p^k}(V) \cdot \zeta_p^{-kl} \right) \cdot  \left[\C_{[l]}\right].
\]  
The augmentation map is itself a trace, namely  
\[
\alpha_{\mathbb{C}}(V) = \operatorname{Tr}_1(V) = \dim V.
\]

\subsection{Computation sequences and graded roots} \label{subsec: computations sequence}

In this subsection, we will review the materials in \cite{nemethi2005ozsvath}. For simplicity, we fix the following notations.
\begin{itemize}
    \item Let $\Gamma$ be a negative-definite, almost rational plumbing graph. Let $W_\Gamma$ denote the associated plumbed 4-manifold, and assume that $Y=\partial W_\Gamma$ is a rational homology sphere.
    \item Let $D_v$ denote the disk bundle associated to a node $v \in V(\Gamma)$, and let $S_v$ be its zero-section.
    \item Let $V(\Gamma)$ be the set of nodes of $\Gamma$, and for each node $v \in V(\Gamma)$, denote its weight by $w(v)$.
\item Fix a ``base node'' $v_0\in V(\Gamma)$ such that decreasing its weight sufficiently makes $\Gamma$ rational.
    \item Identify $H_2(W_\Gamma; \mathbb{Z})$ with $\mathbb{Z} V(\Gamma)$ and $H^2(W_\Gamma; \mathbb{Z})$ with $\operatorname{Hom}_{\mathbb{Z}}(\mathbb{Z} V(\Gamma), \mathbb{Z})$, and regard $\mathbb{Z} V(\Gamma)$ as a sublattice of $\operatorname{Hom}_{\mathbb{Z}}(\mathbb{Z} V(\Gamma), \mathbb{Z})$ by mapping each node $v \in V(\Gamma)$ to its dual $v^*$ with respect to the intersection form on $W_\Gamma$, i.e., $v^*(w) = v \cdot w$ for all $w \in V(\Gamma)$.
    \item Since the index of $\mathbb{Z} V(\Gamma)$ in $\operatorname{Hom}_{\mathbb{Z}}(\mathbb{Z} V(\Gamma), \mathbb{Z})$ is $|H_1(\partial W_\Gamma; \mathbb{Z})|$, and hence finite, we canonically identify $\operatorname{Hom}_{\mathbb{Z}}(\mathbb{Z} V(\Gamma), \mathbb{Z})$ with a subgroup of $\mathbb{Q} V(\Gamma)$.
    \item For $x = \sum_{v \in V(\Gamma)} \lambda_v v \in \mathbb{Q} V(\Gamma)$ and $v \in V(\Gamma)$, denote the coefficient $\lambda_v$ by $m_v(x)$.
    \item We endow $\mathbb{Q} V(\Gamma)$ with the partial order given by $x \le y$ if and only if $m_v(x) \le m_v(y)$ for all $v \in V(\Gamma)$.
\end{itemize}

We start by observing that, since $H^2(W_\Gamma;\mathbb{Z})$ is free and hence has no 2-torsion, the first Chern class map
\[
c_1 \colon \mathrm{Spin}^c(W_\Gamma) \longrightarrow H^2(W_\Gamma;\mathbb{Z}) = \mathrm{Hom}_\mathbb{Z}(\mathbb{Z} V(\Gamma), \mathbb{Z})
\]
is injective, and its image consists precisely of \emph{characteristic elements} of $\Gamma$, i.e., elements $v \in H^2(W_\Gamma; \mathbb{Z})$ satisfying $v(w) \equiv w \cdot w \pmod{2}$ for all $w \in H_2(W_\Gamma; \mathbb{Z}) = \mathbb{Z} V(\Gamma)$. For any characteristic element $x$ of $\Gamma$, we denote the corresponding $\mathrm{Spin}^c$ structure on $W_\Gamma$ by $\mathrm{sp}(x)$.

Denote the set of characteristic elements by
$\mathrm{Char}(\Gamma)$, and identify it with
$\mathrm{Spin}^c(W_\Gamma)$ via the first Chern class.
Every boundary $\mathrm{Spin}^c$ structure $\mathfrak{s}$ determines a
\emph{characteristic coset}
\[
\widetilde c_1(\mathfrak{s})
\subset
\mathrm{Char}(\Gamma),
\]
namely the set of Chern classes of its extensions to $W_\Gamma$.
The first Chern classes of two such extensions differ by an element of $2\mathbb ZV(\Gamma)$, so
\[
\mathrm{Spin}^c(\partial W_\Gamma)
\cong
\mathrm{Char}(\Gamma)/2\mathbb ZV(\Gamma).
\]

The set $\mathrm{Char}(\Gamma)$ contains the canonical class $K$,
defined by
\[
K(v)=-w(v)-2
\qquad
(v\in V(\Gamma)).
\]
For a boundary structure $\mathfrak{s}$, define
\[
S_{\mathfrak{s}}
=
\left\{
x\in H^2(W_\Gamma;\mathbb Z)
\;\middle|\;
K+2x\in\widetilde c_1(\mathfrak{s}),
\quad
x(v)\leq0\text{ for every }v\in V(\Gamma)
\right\}.
\]
By \cite[Lemma~5.4]{nemethi2005ozsvath}, this set has a unique minimal
element $l'_{\mathfrak{s}}$.  Put
\[
k_{\mathfrak{s}}=K+2l'_{\mathfrak{s}}
\in\widetilde c_1(\mathfrak{s}).
\]
Using this representative, define
\[
\chi_{\mathfrak{s}}(x)
=
-\frac{k_{\mathfrak{s}}(x)+x\cdot x}{2},
\qquad
x\in\mathbb ZV(\Gamma).
\]

It is then straightforward to verify that the topological part of the index of the $\mathrm{Spin}^c$ Dirac operator $\dirac_{W_\Gamma, \mathrm{sp}(k_\mathfrak{s} + 2x)}$ for $(W_\Gamma, \mathrm{sp}(k_\mathfrak{s} + 2x))$, whose boundary is $(\partial W_\Gamma, \mathfrak{s})$, is
\[
\begin{aligned}
\mathrm{ind}^t \, \dirac_{W_\Gamma, \mathrm{sp}(k_\mathfrak{s} + 2x)}
&=
\frac{c_1(\mathrm{sp}(k_\mathfrak{s} + 2x))^2
-\sigma(W_\Gamma)}{8} =
\frac{k_\mathfrak{s}^2+|V(\Gamma)|}{8}
-\chi_\mathfrak{s}(x).
\end{aligned}
\]

Now, for each integer $i \ge 0$, we consider elements of $\mathbb{Z} V(\Gamma)$ whose coefficient at the base node $v_0$ is exactly $i$. Define the following subset:
\[
D_i = \left\{ x \in \mathbb{Z} V(\Gamma) \,\middle\vert\, m_{v_0}(x) = i,\; (x + l'_\mathfrak{s})(v) \le 0 \text{ for all } v \in V(\Gamma) \smallsetminus \{v_0\} \right\} \subset \mathbb{Z} V(\Gamma).
\]Since $\Gamma$ is negative-definite, there exists a unique minimal element in $D_i$ with respect to the partial ordering on $\mathbb{Z} V(\Gamma)$~\cite[Lemma 7.6]{nemethi2005ozsvath}, which we denote by $x_\mathfrak{s}(i)$. Moreover, the elements $x_\mathfrak{s}(i)$ and $x_\mathfrak{s}(i+1)$ can be connected by a \emph{computation sequence}, which is a sequence
\[
x^\mathfrak{s}_{i,0}, x^\mathfrak{s}_{i,1}, \dots, x^\mathfrak{s}_{i,n_i}
\]
in $\mathbb{Z} V(\Gamma)$ defined as follows:
\begin{itemize}
    \item $x^\mathfrak{s}_{i,0} = x_\mathfrak{s}(i) + v_0$.
    \item Suppose that $x^\mathfrak{s}_{i,0}, \dots, x^\mathfrak{s}_{i,k}$ have been defined. If $x^\mathfrak{s}_{i,k} = x_\mathfrak{s}(i+1)$, then the sequence terminates, and we set $n_i = k$.
    \item Otherwise, there exists some $v \in V(\Gamma) \smallsetminus \{v_0\}$ such that $(x^\mathfrak{s}_{i,k} + l'_\mathfrak{s})(v) > 0$~\cite[Lemma 7.7]{nemethi2005ozsvath}. Choose such a vertex $v$ and define $x^\mathfrak{s}_{i,k+1} = x^\mathfrak{s}_{i,k} + v$.
\end{itemize}
By concatenating the cycles $x_\mathfrak{s}(i)$ for $i \ge 0$ with the computation sequences connecting them, we obtain the following infinite sequence in $\mathbb{Z} V(\Gamma)$:
\[
x_\mathfrak{s}(0),\, x^\mathfrak{s}_{0,0},\, x^\mathfrak{s}_{0,1},\, \dots,\, x^\mathfrak{s}_{0,n_0-1},\, x_\mathfrak{s}(1),\, x^\mathfrak{s}_{1,0},\, x^\mathfrak{s}_{1,1},\, \dots,\, x^\mathfrak{s}_{1,n_1-1},\, x_\mathfrak{s}(2),\, \dots
\]
Note that $x^\mathfrak{s}_{i,n_i} = x_\mathfrak{s}(i+1)$. Furthermore, since $\Gamma$ is almost rational, we have
\[
\chi_\mathfrak{s}(x^\mathfrak{s}_{i,0}) = \chi_\mathfrak{s}(x^\mathfrak{s}_{i,1}) = \cdots = \chi_\mathfrak{s}(x^\mathfrak{s}_{i,n_i})
\]
for each $i$~\cite[Lemma 9.1]{nemethi2005ozsvath}. Finally, there exists an integer $N > 0$ such that
\[
\chi_\mathfrak{s}(x_\mathfrak{s}(n+1)) \ge \chi_\mathfrak{s}(x_\mathfrak{s}(n))
\]
for all $n > N$~\cite[Theorem 9.3]{nemethi2005ozsvath}.

\begin{rem} \label{rem: computation sequence of spin c str}
We note that the above sequence can also be regarded as a sequence of $\mathrm{Spin}^c$ structures on $W_\Gamma$, which we describe as follows. For each $x \in \mathbb{Z} V(\Gamma)$, let us write $\mathrm{sp}(k_\mathfrak{s} + 2x)$ as $\mathrm{sp}_\mathfrak{s}(x)$. Then the corresponding sequence of $\mathrm{Spin}^c$ structures becomes
\[
\mathrm{sp}_\mathfrak{s}(x_\mathfrak{s}(0)),\, \mathrm{sp}_\mathfrak{s}(x^\mathfrak{s}_{0,0}),\, \dots,\, \mathrm{sp}_\mathfrak{s}(x^\mathfrak{s}_{0,n_0-1}),\, \mathrm{sp}_\mathfrak{s}(x_\mathfrak{s}(1)),\, \dots
\]
We refer to this as the \emph{$\mathrm{Spin}^c$ computation sequence of $(\Gamma, \mathfrak{s})$}. This sequence has the following properties:
\begin{itemize}
    \item $\mathrm{sp}_\mathfrak{s}(x^\mathfrak{s}_{i,0}) = \mathrm{sp}_\mathfrak{s}(x_\mathfrak{s}(i)) + PD[S_{v_0}]$;
    \item $\mathrm{sp}_\mathfrak{s}(x^\mathfrak{s}_{i,j+1}) = \mathrm{sp}_\mathfrak{s}(x^\mathfrak{s}_{i,j}) + PD[S_v]$ for some node $v \in V(\Gamma) \smallsetminus \{v_0\}$;
    \item Each $\mathrm{Spin}^c$ structure in the sequence restricts to $\mathfrak{s}$ on $\partial W_\Gamma$.
\end{itemize}
In fact, any two successive terms in the sequence differ only in the interior of the disk bundle associated to some node of $\Gamma$.
\end{rem}

\begin{rem}\label{rem: conjugation in terms of cycles}
Let $\overline{\mathfrak{s}}$ denote the conjugate boundary
$\mathrm{Spin}^c$ structure.  Then
$
\widetilde c_1(\overline{\mathfrak{s}})
=
-\widetilde c_1(\mathfrak{s}),$
and $
\frac{k_{\mathfrak{s}}+k_{\overline{\mathfrak{s}}}}{2}
\in\mathbb ZV(\Gamma).$ 
For every $x\in\mathbb ZV(\Gamma)$,
\[
\overline{\mathrm{sp}_{\mathfrak{s}}(x)}
=
\mathrm{sp}_{\overline{\mathfrak{s}}}
\left(
-x-\frac{k_{\mathfrak{s}}+k_{\overline{\mathfrak{s}}}}{2}
\right).
\]
Indeed, the first Chern class on the right is
$-k_{\mathfrak{s}}-2x$.  In the self-conjugate case this reduces to
the familiar formula with $-k_{\mathfrak{s}}-x$.
\end{rem}

To convert the computation sequence into a graded root, we recall the following definitions.

\begin{defn}
A \emph{planar graded root} is an infinite tree $R=(V,E)$ embedded in $\R^2$, with each edge a straight line segment. Write $\chi(v)$ for the $y$-coordinate of a vertex $v$, called its \emph{weight}. We require the following conditions:
\begin{itemize}
\item The weights are integers and are bounded below.
\item Every level $\chi^{-1}(n)$ is finite and contains exactly one vertex for all sufficiently large $n$.
\item The weights at the ends of an edge differ by one, and every vertex has exactly one adjacent vertex of greater weight.
\end{itemize}
Two planar graded roots are equivalent if they are related by a horizontal isotopy and an overall vertical translation.
\end{defn}

We also recall simple angles of planar graded roots.
\begin{defn}
Let $v,v'$ be leaves and let $w$ be a vertex of $R$. We say that \emph{$v$ and $v'$ form an angle at $w$} if the paths $[v,w]$ and $[v',w]$ are strictly increasing in weight, intersect only at $w$, and the first path lies to the left of the second near $w$. The angle is \emph{simple} if no path from a third leaf to $w$ lies between them. We put
\[
\mathrm{Angle}(w)=\{(v,v')\mid v,v'\text{ form a simple angle at }w\}.
\]
The weight of this angle is $\chi(w)$.
\end{defn}

The following observation describes the data needed to recover a planar graded root.
\begin{lem}\label{lem: how to describe a graded root}
A planar graded root is determined, up to equivalence and an overall weight shift, by the weights of its ordered leaves and the weights of the simple angles between consecutive leaves.
\end{lem}

We call an integer sequence $\mathbf n=(n_i)_{i\ge0}$ \emph{eventually increasing} if $n_{i+1}\ge n_i$ for all sufficiently large $i$. Such a sequence determines a planar graded root as follows.
\begin{defn}\label{defn: graded root from abstract sequence}
Retain a finite initial segment through the last strict decrease of $\mathbf n$. Extract the following data:
\begin{itemize}
\item Let $I=\{i_0<\cdots<i_m\}$ contain the right endpoint of each local-minimum plateau in this segment. Include a one-sided endpoint when it is a minimum; for a constant segment, retain one index.
\item For each $1\le k\le m$, choose $j_k$ with
\[
i_{k-1}\le j_k\le i_k,
\qquad n_{j_k}=\max_{i_{k-1}\le j\le i_k}n_j.
\]
\end{itemize}
The \emph{planar graded root $R_{\mathbf n}$} has an ordered leaf of weight $n_{i_k}$ for each $i_k\in I$, and the consecutive leaves $i_{k-1},i_k$ form a simple angle of weight $n_{j_k}$.
\end{defn}

Equivalently, take an upward ray starting at height $n_i$ for each retained index and identify consecutive rays at all heights at least $\max(n_i,n_{i+1})$. This constructs the specified root and shows that extending the retained segment through the nondecreasing tail does not change it.

Recall that the computation sequence gives the eventually increasing sequence
\[
\chi_\mathfrak{s}(x_\mathfrak{s}(0)),
\chi_\mathfrak{s}(x^\mathfrak{s}_{0,0}),\ldots,
\chi_\mathfrak{s}(x^\mathfrak{s}_{0,n_0-1}),
\chi_\mathfrak{s}(x_\mathfrak{s}(1)),\ldots.
\]
The intermediate terms between $x_\mathfrak{s}(i)$ and $x_\mathfrak{s}(i+1)$ have the same weight as the latter. Hence its root is the same as that of the sequence
\[
\mathbf n(\Gamma,\mathfrak{s})=
\bigl(\chi_\mathfrak{s}(x_\mathfrak{s}(i))\bigr)_{i\ge0}.
\]
We denote this root by $R_{\Gamma,\mathfrak{s}}$. It gives the usual lattice description of $HF^+(-\partial W_\Gamma,\mathfrak{s})$; see \cite[Theorem~2.4.6]{nemethi2005ozsvath}. We retain the planar order supplied by the chosen computation sequence.

\subsection{Graded roots from $\Delta$-sequences} \label{subsec: graded root to delta seq}

Let $Y$ be a Seifert fibered rational homology sphere with a fixed
normalized presentation
\[
Y=Y\bigl(e_0;(p_1,q_1),\dots,(p_\nu,q_\nu)\bigr),
\qquad 0<q_l<p_l,
\]
and rational Euler number
\[
e=e_0+\sum_{l=1}^{\nu}\frac{q_l}{p_l}
=
-\frac{|H_1(Y;\Z)|}{p_1\cdots p_\nu}<0.
\]
This is the rational Euler number of the Seifert fibration, not the
orbifold Euler characteristic of the base.  The normalized
presentation determines the star-shaped negative-definite plumbing
$\Gamma$: the central node has weight $e_0$, and the $l$th arm is
given by 
\[
\begin{tikzpicture}[xscale=1.5, yscale=1, baseline={(0,-0.1)}]
    \node at (1, 0.3) {$-k^l_1$};
    \node at (2, 0.3) {$-k^l_2$};
    \node at (3, 0.3) {$-k^l_3$};
    \node at (4.5, 0.3) {$-k^l_{s_l}$};
    
    \node at (1, 0) (C') {$\bullet$};
    \node at (2, 0) (C) {$\bullet$};
    \node at (3, 0) (A1) {$\bullet$};
    \node at (3, -0.2) (A1b) {};
    \node at (4.5, 0) (A5) {$\bullet$};
    \node at (4.5, -0.2) (A5b) {};
    
    \draw (C') -- (C) -- (A1);
    \draw[dotted] (A1) -- (A5);

\end{tikzpicture}
\]
where $k^l_1,\cdots,k^l_{s_l}$ are uniquely determined positive integers satisfying $k^l_1,\dots,k^l_{s_l} \ge 2$ and 
\[
\frac{p_l}{q_l}= [k^l_1,\dots,k^l_{s_l}] = k^l_1 - \cfrac{1}{k^l_2 - \cfrac{1}{\ddots-\cfrac{1}{k^l_{s_l}}}}.
\]
The Seifert relation is
\[
e_0+\sum_{l=1}^{\nu}\frac{q_l}{p_l}
=
-\frac{|H_1(Y;\Z)|}{p_1\cdots p_\nu}.
\]
We denote by $v^l_i$ the node on the $l$th arm whose weight is $-k^l_i$, that is, the $i$th node from the central node.

The resulting plumbing graph $\Gamma$ is negative-definite and almost rational; note that $Y\cong\partial W_\Gamma$. Furthermore, define
\[
SI_{red}(Y)
=
\left\{
(a_0,\dots,a_\nu)\in\Z^{\nu+1}
\;\middle|\;
\begin{array}{l}
a_0\ge0,\quad 0\le a_l<p_l\ (1\le l\le\nu),\\[1mm]
\displaystyle
1+a_0+ie_0+
\sum_{l=1}^{\nu}
\left\lfloor\frac{iq_l+a_l}{p_l}\right\rfloor
\le0
\quad\text{for every integer }i\ge1
\end{array}
\right\}.
\]
This is an effective finite test: it is enough to check
$1\le i\le p_1\cdots p_\nu$, because increasing $i$ by
$p_1\cdots p_\nu$ decreases the displayed expression by
$|H_1(Y;\Z)|$.  Moreover the $i=1$ inequality bounds $a_0$.
Thus there are only finitely many tuples and finitely many inequalities
to check.

There exists a bijective correspondence between $SI_{red}(Y)$ and $\mathrm{Spin}^c(Y)$ \cite[Corollary 11.7]{nemethi2005ozsvath}, which is constructed in the following way. Given $\mathfrak{s}\in \mathrm{Spin}^c(Y)$, the corresponding element $(a^\mathfrak{s}_0,a^\mathfrak{s}_1,\dots,a^\mathfrak{s}_\nu)\in SI_{red}(Y)$ is determined as follows \cite[Proposition 11.5]{nemethi2005ozsvath}. 
\begin{itemize}
    \item We have $a^\mathfrak{s}_0 = -l'_\mathfrak{s}(v_c)$.
    \item For each $l = 1, \dots, \nu$ and any $i = 1, \dots, s_l$, define $n^l_i$ and $d^l_i$ to be the unique coprime positive integers satisfying
    \[
    \frac{n^l_i}{d^l_i} = [k^l_i, \dots, k^l_{s_l}].
    \]
    Then, for each $l = 1, \dots, \nu$, we have
    \[
    a^\mathfrak{s}_l = -l'_\mathfrak{s}(v_{s_l}^l) - \sum_{t=1}^{s_l-1} n^l_{t+1} \, l'_\mathfrak{s}(v^l_t).
    \]
\end{itemize}

\begin{rem} \label{rem: canonical spin c str}
    The zero vector $(0, \dots, 0)$ is always contained in $SI_{\mathrm{red}}(Y)$.  
    If we denote by $\mathfrak{s}^{\mathrm{can}}_Y$ the corresponding $\mathrm{Spin}^c$ structure on $Y$, then it follows from \cite[Proposition 11.6]{nemethi2005ozsvath} that $l'_{\mathfrak{s}^{\mathrm{can}}_Y} = 0$, and thus $k_{\mathfrak{s}^{\mathrm{can}}_Y} = K$.  
    In other words, $\mathfrak{s}^{\mathrm{can}}_Y$ is the restriction to the boundary of the $\mathrm{Spin}^c$ structure on $W_\Gamma$ whose first Chern class is the canonical class $K$.  
    Hence, we call $\mathfrak{s}^{\mathrm{can}}_Y$ the \emph{canonical $\mathrm{Spin}^c$ structure} of $Y$.  
    We note that $\mathfrak{s}^{\mathrm{can}}_Y$ is self-conjugate if and only if $K \in \mathbb{Z} V(\Gamma)$, which is equivalent to $m_v(K) \in \mathbb{Z}$ for all $v \in V(\Gamma)$.
\end{rem}

Choose any $\mathrm{Spin}^c$ structure $\mathfrak{s}$ on $Y$, and write the corresponding element of $SI_{red}(Y)$ as $(a^\mathfrak{s}_0, a^\mathfrak{s}_1, \dots, a^\mathfrak{s}_\nu)$. To compute the planar graded root $R_{\Gamma,\mathfrak{s}}$ (which depends only on $Y$ and $\mathfrak{s}$), it suffices to compute the planar graded root associated to the sequence $(\chi_\mathfrak{s}(x_\mathfrak{s}(i)))_{i \ge 0}$, which additionally depends on the choice of a base node $v_0$ of $\Gamma$. For simplicity, it is natural to take the base node to be the central node, i.e., $v_0 = v_c$.

It is clear that $x_\mathfrak{s}(0)=0$, and thus $\chi_\mathfrak{s}(x_\mathfrak{s}(0))=0$. After that, we consider the \emph{$\Delta$-sequence for $(Y,\mathfrak{s})$}, defined as follows:
\[
\Delta_{Y,\mathfrak{s}}(i) = 1+a^\mathfrak{s}_0 - e_0 i + \sum_{l=1}^\nu \left\lfloor \frac{-iq_l+a^\mathfrak{s}_l}{p_l} \right\rfloor \quad \text{for each}\quad i\ge 0.
\]
Then, for any integer $i\ge 0$, the following equation holds \cite[Section 11.12]{nemethi2005ozsvath}:
\[
\chi_\mathfrak{s}(x_\mathfrak{s}(i+1)) - \chi_\mathfrak{s}(x_\mathfrak{s}(i)) = \Delta_{Y,\mathfrak{s}}(i).
\]
Using
\[
\left\lfloor\frac{b}{p}\right\rfloor
\ge
\frac{b-p+1}{p},
\]
we obtain
\[
\Delta_{Y,\mathfrak{s}}(i)
\ge
1+a_0^{\mathfrak{s}}-\nu
+
\sum_{l=1}^{\nu}\frac{a_l^{\mathfrak{s}}+1}{p_l}
-ei.
\]
Since $e<0$, this is eventually greater than $-1$, and integrality
gives $\Delta_{Y,\mathfrak{s}}(i)\ge0$ for all sufficiently large
$i$.  Thus the sequence is eventually nondecreasing for every
$\mathfrak{s}$, without a self-conjugacy assumption.

The ordered sequence determines $R_{\Gamma,\mathfrak{s}}$ for the
fixed plumbing, and its lattice homology is isomorphic to
$HF^+(-Y,\mathfrak{s})$.

\begin{rem} \label{rem: Delta seq for canonical spin c}
The $\Delta$-sequence for $(Y, \mathfrak{s}^{\mathrm{can}}_Y)$ can be computed in a much simpler way, since, as noted in \Cref{rem: canonical spin c str}, the corresponding element of $SI_{\mathrm{red}}(Y)$ is the zero vector. Indeed, for each integer $i \ge 0$, the formula for $\Delta_{Y,\mathfrak{s}}(i)$ simplifies to:
\[
\Delta_{Y, \mathfrak{s}^{\mathrm{can}}_Y}(i) = 1 - e_0 i - \sum_{l=1}^\nu \left\lceil \frac{i q_l}{p_l} \right\rceil.
\]
Using the inequality $\left\lceil \frac{q}{p} \right\rceil \le \frac{q + p - 1}{p}$ for integers $p, q$ with $p > 0$, we obtain the following lower bound:
\[
\begin{split}
\Delta_{Y, \mathfrak{s}^{\mathrm{can}}_Y}(i)
&\ge 1 - i\left(e_0 + \sum_{l=1}^\nu \frac{q_l}{p_l} \right) - \nu + \sum_{l=1}^\nu \frac{1}{p_l} = 1 - \nu + \frac{|H_1(Y; \mathbb{Z})| \cdot i + \sum_{l=1}^\nu p_1 \cdots \widehat{p}_l \cdots p_\nu}{p_1 \cdots p_\nu}.
\end{split}
\]
Hence, if we define the number
\[
N_Y = \frac{(\nu - 2)p_1 \cdots p_\nu - \sum_{l=1}^\nu p_1 \cdots \widehat{p}_l \cdots p_\nu}{|H_1(Y; \mathbb{Z})|},
\]
which we regard as a rational cutoff in general, then we obtain $\Delta_{Y, \mathfrak{s}^{\mathrm{can}}_Y}(i) > -1$, and hence $\Delta_{Y, \mathfrak{s}^{\mathrm{can}}_Y}(i) \ge 0$, for all integers $i > N_Y$. Consequently, the sequence $(\chi_{\mathfrak{s}^{\mathrm{can}}_Y}(x_{\mathfrak{s}^{\mathrm{can}}_Y}(i)))_{i \ge 0}$ is increasing for $i > N_Y$.
\end{rem}

\section{Equivariant Bauer--Furuta theory for equivariant $\mathrm{Spin}^c$ 4-manifolds}

Although Montague~\cite{montague2022seiberg} developed a general framework for the $\mathrm{Pin}(2)\times \mathbb{Z}_p$-equivariant Floer homotopy type for spin $\mathbb{Z}_p$-actions, it is also necessary to construct the $S^1 \times \mathbb{Z}_p$-equivariant Floer homotopy type in the general case, along with the corresponding $S^1 \times \mathbb{Z}_p$-equivariant Bauer--Furuta invariants, which recover the theory of Baraglia and Hekmati~\cite{Baraglia-Hekmati:2024-1}. We also establish a gluing theorem that will be used in our construction.

\subsection{$S^1\times \Z_p$-equivariant Seiberg--Witten Floer homotopy type} \label{subsec:S1xZp-SWF-homotopy-type}

Let $Y$ be a rational homology 3-sphere equipped with a $\Z_p$-action and a $\Z_p$-invariant $\mathrm{Spin}^c$ structure $\mathfrak{s}$. Here we do {\bf not} assume the $\Z_p$-action is free.
We fix a $\Z_p$-invariant Riemannian metric $g$. 
First, choose a reference $\mathrm{Spin}^c$ connection $B_0$ such that the associated connection on the determinant line bundle is flat. As shown in~\cite[Section 3.2]{Baraglia-Hekmati:2024-1}, for each $\tau \in \Z_p$, we can choose a lift $\tilde{\tau}$ of $\tau$ to the spinor bundle that preserves $B_0$. Here we use the assumptions that the $\Z_p$-action preserves the isomorphism class of $\mathfrak{s}$ and that $b_1(Y) = 0$.
Let $G_{\mathfrak{s}}$ denote the set of unitary automorphisms $u \colon S \to S$ of the spinor bundle $S$ that preserve $B_0$ and lift the $\Z_p$-action on $Y$. Then we have a short exact sequence:
\[
1 \longrightarrow S^1 \longrightarrow G_{\mathfrak{s}} \longrightarrow \Z_p \longrightarrow 1.
\]
This extension is always trivial in our setting, as shown in~\cite[Section 5]{Baraglia-Hekmati:2024-1}. Therefore, we may choose a section and identify
\begin{align*}\label{splitting}
    G_{\mathfrak{s}} \cong S^1 \times \Z_p.
\end{align*}
Note that the set of splittings of this extension, i.e., the set of sections
\[
\mathrm{Split}(Y, \mathfrak{s}) := \left\{ \text{right inverses } \Z_p \to G_{\mathfrak{s}} \text{ of } G_{\mathfrak{s}} \to \Z_p \right\},
\]
admits a naturally defined map
\[
\mathcal{S}_{Y, \mathfrak{s}} \colon \mathrm{Split}(Y, \mathfrak{s}) \longrightarrow \mathrm{Spin}^c_{\Z_p}(Y, \mathfrak{s}),
\]
where $\mathrm{Spin}^c_{\Z_p}(Y, \mathfrak{s})$ denotes the set of \emph{$\Z_p$-equivariant} $\mathrm{Spin}^c$ structures lifting $\mathfrak{s}$. This map is defined as follows: a section $\Z_p \to G_{\mathfrak{s}}$ defines a lift of the $\Z_p$-action on the frame bundle of $Y$ to the $\mathrm{Spin}^c$ bundle corresponding to $\mathfrak{s}$, and hence defines a $\Z_p$-equivariant lift of $\mathfrak{s}$.

In \Cref{lem: splittings are eqv spin c lifts}, we will show that $\mathcal{S}_{Y, \mathfrak{s}}$ is bijective whenever $p$ does not divide $|H_1(Y; \mathbb{Z})|$. From now on (in this subsection), we assume that $|H_1(Y; \mathbb{Z})|$ is not divisible by $p$. Given a $\mathbb{Z}_p$-equivariant $\mathrm{Spin}^c$ structure $\mathfrak{s}$ on $Y$, we denote by $G_{\mathfrak{s}}$ the group defined above, \emph{together with the identification $G_{\mathfrak{s}} \cong S^1 \times \mathbb{Z}_p$ induced by $\mathfrak{s}$}; conversely, if we have chosen a $\Z_p$-equivariant $\mathrm{Spin}^c$ structure $\mathfrak{s}$, whenever the group $S^1 \times \Z_p$ appears, it should be understood as $G_{\mathfrak{s}}$.

We fix a flat connection $B_0$ on $\s$ which is $\Z_p$-invariant. Now we have an action of $G_\s$ on the global Coulomb slice
\[
V=\left( B_0 + i\ker d^* \right) \oplus \Gamma(S) \subset i\Omega^1(Y) \oplus \Gamma(S). 
\]
Again, by finite-dimensional approximation of the Seiberg--Witten equation, we obtain a $G_\s$-equivariant Conley index $I^\mu_\lambda(g)$ for sufficiently large real numbers $\mu, -\lambda$.
For the details of the construction, see \cite{Man03, baraglia2024brieskorn, montague2022seiberg}.
Now a metric-dependent equivariant Floer homotopy type is defined as 
\[
SWF_{S^1\times \Z_p}(Y, \s, \tilde{\tau}, g)=\Sigma^{-V^0_\lambda(g)}I^\mu_\lambda(g).
\]

Let $L$ be the determinant line bundle of $S$, and fix a $\Z_p$-invariant flat connection $B_0$ on $L$. We denote by 
\[
\diracpartial_{B_0} \colon \Gamma(S) \longrightarrow \Gamma(S)
\]
the $\mathrm{Spin}^c$-Dirac operator with respect to $B_0$.
Notice that $I^\mu_\lambda(g)$ can be taken as a finite $S^1 \times \Z_p$-CW complex.
\begin{defn}
Now, for $k \in \Z_p$, we introduce the equivariant correction term
\[
n(Y, \mathfrak{s}, \wt{\tau}, g)_{k} \in \C
\]
defined by
\[
n(Y, \mathfrak{s}, \wt{\tau}, g)_{k}
:=
\begin{cases}
\displaystyle
\overline{\eta}^k_{\diracpartial_{B_0}}(g)
-
\frac{1}{2\pi}
\sum_{i=1}^{s_k}
T_{k,i}\,
t\!\left(Y,K_{k,i},g\right),
& 0<k<p, \\[1.2ex]
\displaystyle
\overline{\eta}^{0}_{\diracpartial_{B_0}}(g)
-\frac18\eta_{\mathrm{sign}}(g)
=
n(Y,\mathfrak{s},g),
& k=0.
\end{cases}
\]
where we use the following notations.
\begin{itemize}
    \item For the flat connection $A_0$ and the $\Z_p$-invariant Riemannian metric on $Y$, we associate a $\Z_p$-equivariant Dirac operator
    \[
    \diracpartial_{B_0} \colon \Gamma(S) \longrightarrow \Gamma(S).
    \]
    The equivariant eta-invariant 
    \[
{\eta}^k_{\diracpartial_{B_0}}(g) \in \C
\]
associated to $\diracpartial_{B_0}$ is the value at $s = 0$ of the meromorphic continuation of the equivariant eta function
    \[
    {\eta}^k_{\diracpartial_{B_0}}(g,s) = \sum_{0\neq \lambda \text{ eigenvalue of } \diracpartial_{B_0}}  
    \frac{\sign(\lambda) \cdot \operatorname{Trace}\big( (\tilde{\tau}^k)^* \colon V_\lambda \longrightarrow V_\lambda \big)}{|\lambda|^s} \in \C, \quad s \in \C,
    \]
    where $V_\lambda$ denotes the eigenspace of $\diracpartial_{B_0}$ for $\lambda$. Note that the finiteness of ${\eta}^k_{\diracpartial_{B_0}}(g,0)$ is verified in   Donnelly~\cite{donnelly1978eta} using an equivariant version of the heat kernel representation of it, together with the small-time asymptotic expansion of the heat kernel, which shows that all potentially divergent terms cancel, leaving a regular value at $s=0$. 
Similarly, for the operator 
\[
D_{\mathrm{sign}} :=  \begin{pmatrix}
    * d & -d \\
    -d^* & 0 
\end{pmatrix} : \Omega^1_Y \oplus \Omega^0_Y  \longrightarrow \Omega^1_Y \oplus \Omega^0_Y, 
\]
we have the non-equivariant eta invariant 
\[
\eta_{\mathrm{sign}}(g) \in \C. 
\]

    \item The reduced equivariant eta-invariant 
    \[
\overline{\eta}^k_{\diracpartial_{B_0}}(g) \in \C 
    \]
    is defined as 
    \[
\overline{\eta}^k_{\diracpartial_{B_0}}(g) = \tfrac{1}{2} \left( {\eta}^k_{\diracpartial_{B_0}}(g) - c^k_{\diracpartial_{B_0}}(g) \right),
    \]
    where $c^k_{\diracpartial_{B_0}}(g)$ denotes
    \[
    \operatorname{Trace}\big( (\tilde{\tau}^k)^* \colon \ker \diracpartial_{B_0} \longrightarrow \ker \diracpartial_{B_0} \big) \in \C.
    \]

    \item Let $L$ be a connected component of the fixed point set, which is a knot in $Y$. We fix orientations on these components.  For the rational Seifert framing of $L$ with the fixed $\Z_p$-invariant Riemannian metric $g$, we obtain a number
    \[
    t(Y, L, g) \in \R
    \]
    called \emph{torsion}, defined as follows:  
    Let $\nabla^{\mathrm{fr}}$ be the $\mathrm{SO}(3)$-connection on the $\mathrm{SO}(3)$-frame bundle $\mathrm{Fr}(Y) \to Y$ induced by the Levi-Civita connection on $(Y, g)$, and let $\theta = (\theta_{ij}) \in \Omega^1(Y; \mathfrak{so}(3))$ be the connection one-form associated to $\nabla^{\mathrm{fr}}$.  
    Given a framed, oriented link $L \subset Y$ equipped with a framing $\alpha$ of $L$, we can trivialize $TY|_L$ by setting, at each point $x \in L$:
    \begin{itemize}
        \item $e_1(x)$ is the unit tangent vector to $L$, oriented consistently.
        \item $e_2(x)$ is the unit vector pointing in the direction of the framing.
        \item $e_3(x) = e_1(x) \times e_2(x)$.
    \end{itemize}
    This trivialization provides a section $\varphi \colon L \to \mathrm{Fr}(Y)$, and we define the torsion of $L$ with respect to $(g, \alpha)$ by
    \[
      t(L, g, \alpha) := - \int_L \varphi^* \theta_{23}
    \]
    Note that for any two framings $\alpha_0, \alpha_1$, we have
    \[
      t(L, g, \alpha_1) - t(L, g, \alpha_0) \in 2\pi \mathbb{Z}.
    \]
    If $Y$ is a rational homology sphere, then any link $L \subset Y$ is rationally null-homologous. We use such a rational canonical framing as $\alpha$. See \cite{yoshida1985eta, montague2022seiberg} for the details. 

    Also, if we change the orientation of $L$, we have 
    \[
    t(-L, g, \alpha)  = t(L, g, \alpha). 
    \]

    \item We decompose the fixed point set $Y^{\tau^k}$ of $\tau^k$
as a union of its connected components
\[
Y^{\tau^k}
=
K_{k,1}\sqcup\cdots\sqcup K_{k,s_k},
\]
equipped with orientations fixed in the previous item.
For each $i=1,\dots,s_k$, write the action of $\tau^k$ on the
oriented normal bundle of $K_{k,i}$ as
\[
z\longmapsto \zeta_p^{k_i}z,
\qquad
k_i\in\{1,\dots,p-1\}.
\]
Choose $m_i\in\Z_p$ so that the induced equivariant
$\mathrm{Spin}^c$ structure on a neighborhood of
$K_{k,i}\times I\subset Y\times I$ is described by
\[
[(x,y,z)]
\longmapsto
\left[
\left(
\zeta_{2p}^{k_i}x,\,
\zeta_{2p}^{-k_i}y,\,
\zeta_{2p}^{k_i+2m_i}z
\right)
\right].
\]
Here
\[
\widetilde T
=
\frac{U(1)\times U(1)\times U(1)}
{\{\pm(1,1,1)\}}
\longrightarrow
SO(2)\times SO(2),
\qquad
[(x,y,z)]\longmapsto(xy,xy^{-1}),
\]
is the local reduction of the principal $\mathrm{Spin}^c$ bundle.
The integer $m_i$ is the equivariance number in the convention of
\Cref{appendix A}.

For this component, set
\[
T_{k,i}
:=
\zeta_p^{m_i}T(k_i),
\qquad
T(u)
:=
-\frac{1+\zeta_p^{-u}}
{2(1-\zeta_p^{-u})^2}.
\]
Thus $T_{k,i}$ is the full coefficient of the normal-curvature
contribution for the action of $\tau^k$ on $K_{k,i}$.
\end{itemize}
\end{defn}

For a disjoint union $(Y, \s, \tilde{\tau}, g)$ of $\Z_p$-equivariant $\mathrm{Spin}^c$ rational homology 3-spheres $\bigsqcup (Y_i, \s_i, \tilde{\tau}_i, g_i)$ with $\Z_p$-invariant Riemannian metrics, we define
\[
n(Y, \s, \tilde{\tau}, g)_{k} := \sum_{i=1}^n n(Y_i, \s_i, \tilde{\tau}_i, g_i)_{k} \in \C .
\]

\begin{rem}
Suppose that the equivariant $\mathrm{Spin}^c$ structure is induced
from an equivariant Spin structure.  To compare our normalization
with Montague's, one substitutes the actual characters of the chosen
Spin lift into the fixed-point formula of \Cref{appendix A}.
With the same lift, orientation, and reduced eta conventions, this
recovers Montague's correction term.  
\end{rem}

\begin{defn}
We define
\[
\mathbf{n}(Y, \mathfrak{s}, \wt{\tau}, g) := \frac{1}{p}\sum_{l=0}^{p-1} \left( \sum_{k=0}^{p-1} n(Y, \mathfrak{s}, \wt{\tau}, g)_k \cdot \zeta_p^{-kl} \right) \otimes [\C_{[l]}] \in R(\Z_p) \otimes \C
\]
and call it the \emph{equivariant correction term}. 
\end{defn}

In order to see several properties of the equivariant correction term, we use the equivariant index of the Dirac operator. Let $(X, \mathfrak{s})$ be a compact connected $\mathrm{Spin}^c$ 4-manifold bounded by a disjoint union of rational homology 3-spheres equipped with the restricted $\mathrm{Spin}^c$ structure $\mathfrak{t} = \mathfrak{s}|_Y$. Suppose $X$ is equipped with a smooth $\Z_p$-action such that the action preserves the isomorphism class of the $\mathrm{Spin}^c$ structure.
Assume that $\mathfrak{s}$ is equipped with a chosen $\Z_p$-equivariant $\mathrm{Spin}^c$ structure $\widetilde{\mathfrak{s}}$ on $X$, whose restriction to the boundary is the prescribed equivariant structure on $\mathfrak{t}$. 
We take a $\Z_p$-invariant Riemannian metric on $X$ which is product near the boundary and a reference $\Z_p$-invariant $\mathrm{Spin}^c$-connection $A_0$.
For $0<k<p$, write the fixed-point set of $\tau^k$ as
\[
X^{\tau^k}_0
=
\{p_{k,1},\dots,p_{k,r_k}\},
\qquad
X^{\tau^k}_2
=
\Sigma_{k,1}\sqcup\cdots\sqcup\Sigma_{k,s_k},
\]
where the fixed surfaces are oriented and may have boundary.

For an isolated fixed point $p_{k,i}$, let
$(u_{k,i},v_{k,i})$ be its tangent weights and let
$m_{k,i}$ be its equivariance number.  Put
\[
R_{k,i}
=
\zeta_p^{m_{k,i}}
R(u_{k,i},v_{k,i}).
\]
For a fixed surface $\Sigma_{k,j}$, let $u_{k,j}$ be its oriented
normal weight and let $m_{k,j}$ be its equivariance number.  Put
\[
S_{k,j}
=
\zeta_p^{m_{k,j}}S(u_{k,j}),
\qquad
T_{k,j}
=
\zeta_p^{m_{k,j}}T(u_{k,j}).
\]
Then \Cref{even_equivariant_index} gives
\[
\begin{split}
\ind_{\tau^k}^{\mathrm{APS}}(\dirac_{A_0})
={}&
\overline{\eta}^{\,k}_{\diracpartial_{B_0}}(g)
+
\sum_{i=1}^{r_k} R_{k,i}
+
\sum_{j=1}^{s_k}
\left(
S_{k,j}
\left\langle c_1(\s),[\Sigma_{k,j}]\right\rangle
+
T_{k,j}
\int_{\Sigma_{k,j}}\widetilde F_N
\right).
\end{split}
\]
Here the $c_1$-pairing is understood using the unique relative
rational lift described in \Cref{rem: index thm for nonclosed fixed points}.
The entire expression is independent of the auxiliary orientations
used to describe the local fixed-point data.

We will need the following topological lemma.

\begin{prop}\label{spinc}
There exists a sufficiently large integer $N > 0$ such that the disjoint union $\bigsqcup_N (Y, \tilde{\mathfrak{s}})$ bounds a $\Z_p$-equivariant $\mathrm{Spin}^c$ filling.
\end{prop}

\begin{proof}
We first remove the fixed circles by equivariant surgery.

Let $K\subset Y$ be a connected component of the fixed-point set of
a nontrivial element of $\Z_p$.  Since the action is
orientation-preserving, an equivariant tubular neighborhood of $K$
is of the form
$
S^1\times D^2$,
where $\Z_p$ acts trivially on the first factor and by rotation on
the second factor.

Attach a $2$--handle
$
D^2\times D^2$
along $S^1\times D^2$, using the same rotation on the second factor.
On the new boundary piece $D^2\times S^1$, the induced $\Z_p$-action
is free.

We claim that the prescribed equivariant $\mathrm{Spin}^c$ structure
extends over this handle.  The attaching region has a unique
underlying $\mathrm{Spin}^c$ structure, while the possible
equivariant lifts form a torsor over
\[
H^2_{\Z_p}(S^1\times D^2;\Z)
\cong
H^2(S^1\times B\Z_p;\Z)
\cong
\Z_p.
\]
A reference equivariant $\mathrm{Spin}^c$ structure on the handle
restricts to a reference structure on the attaching region, and every
character twist on the attaching region extends over the handle by
the same character twist.  Hence the given equivariant
$\mathrm{Spin}^c$ structure extends.

Repeating this surgery for every fixed circle gives a
$\Z_p$-equivariant $\mathrm{Spin}^c$ cobordism from
$(Y,\widetilde{\mathfrak s})$ to a $3$--manifold with free
$\Z_p$-action.  A component on which the geometric action is trivial
can instead be filled by ordinary $\mathrm{Spin}^c$ bordism, together
with the same scalar character on the lift.

A free $\Z_p$-equivariant $\mathrm{Spin}^c$ $3$--manifold determines
an ordinary $\mathrm{Spin}^c$ $3$--manifold equipped with a map to
$B\Z_p$.  We therefore consider
\[
\Omega^{\mathrm{Spin}^c}_3(B\Z_p).
\]
The Atiyah--Hirzebruch spectral sequence has
\[
E^2_{i,j}
=
H_i(B\Z_p;\Omega_j^{\mathrm{Spin}^c})
\Longrightarrow
\Omega^{\mathrm{Spin}^c}_{i+j}(B\Z_p).
\]
Using
\[
\Omega_0^{\mathrm{Spin}^c}\cong\Z,
\qquad
\Omega_1^{\mathrm{Spin}^c}=0,
\qquad
\Omega_2^{\mathrm{Spin}^c}\cong\Z,
\qquad
\Omega_3^{\mathrm{Spin}^c}=0,
\]
the only possibly nonzero terms in total degree three are
\[
H_3(B\Z_p;\Z)\cong\Z_p,
\qquad
H_1(B\Z_p;\Z)\cong\Z_p.
\]
Thus $\Omega^{\mathrm{Spin}^c}_3(B\Z_p)$ is finite; in fact its
exponent divides $p^2$.

Consequently some positive integer $N$ annihilates the bordism class
of the free-action manifold obtained above.  Gluing the corresponding
bordism to the fixed-circle surgeries shows that $
\bigsqcup_N (Y,\widetilde{\mathfrak s})$
bounds a $\Z_p$-equivariant $\mathrm{Spin}^c$ $4$--manifold.
\end{proof}

Now we record some fundamental properties of $\mathbf{n}(Y, \mathfrak{s}, \wt{\tau}, g)$.

\begin{prop}\label{prop:basic_of_n}
The equivariant correction term satisfies the following properties.
\begin{itemize}
    \item[(i)] $\mathbf{n}(Y, \mathfrak{s}, \wt{\tau}, g) \in R(\Z_p) \otimes \Q$;
    \item[(ii)] Under the augmentation map 
    \[
    \alpha_\C \colon R(\Z_p) \otimes \Q \longrightarrow \Q,
    \]
    $\mathbf{n}(Y, \mathfrak{s}, \wt{\tau}, g)$ is sent to Manolescu's original correction term $n(Y, \mathfrak{s}, g)$;
    \item[(iii)] For a 1-parameter family of $\Z_p$-invariant Riemannian metrics $g_s$ from $g_0$ to $g_1$, we have 
  \[
\mathbf{n}(Y, \mathfrak{s}, \wt{\tau}, g_1)
-
\mathbf{n}(Y, \mathfrak{s}, \wt{\tau}, g_0)
=
\operatorname{Sf}( \diracpartial_{B_0}(g_s)) \in R(\Z_p),
\]
    where $\operatorname{Sf}( \diracpartial_{B_0}(g_s))$ denotes the equivariant spectral flow introduced in \cite{lim2024equivariant}.  
\end{itemize}
\end{prop}

\begin{proof}
\noindent{\bf \underline{Proof of (i).}}
By \Cref{spinc}, there are an integer $N>0$ and a compact
$\Z_p$-equivariant $\mathrm{Spin}^c$ $4$--manifold
$(X,\widetilde{\mathfrak{s}}_X)$ with
\[
\partial X=\bigsqcup_N(Y,\widetilde{\mathfrak{s}}).
\]
For $0<k<p$, let
\[
A_k(X)
=
\sum_i R_{k,i}
+
\sum_j
\left(
S_{k,j}
\langle c_1(\mathfrak{s}_X),[\Sigma_{k,j}]\rangle
+
T_{k,j}[\Sigma_{k,j}]^2
\right),
\]
using the relative interpretation of
\Cref{rem: index thm for nonclosed fixed points} for fixed surfaces
with boundary, and put
\[
A_0(X)=\frac{c_1(\mathfrak{s}_X)^2-\sigma(X)}8.
\]
The equivariant APS formula and the torsion correction give
\[
\ind_{\tau^k}^{\mathrm{APS}}(\dirac_{A_0})
=
A_k(X)
+
N\,n(Y,\mathfrak{s},\widetilde\tau,g)_k
\qquad(k\in\Z_p).
\]

Define
\[
\mathbf I^t(X)
=
\frac1p\sum_{l=0}^{p-1}
\left(
\sum_{k=0}^{p-1}A_k(X)\zeta_p^{-kl}
\right)[\C_{[l]}].
\]
For $a\in(\Z/p\Z)^\times$, Galois conjugation
$\sigma_a(\zeta_p)=\zeta_p^a$ carries $A_k(X)$ to $A_{ak}(X)$, while
$A_0(X)\in\Q$.  Hence every Fourier coefficient above is Galois
fixed, so
\[
\mathbf I^t(X)\in R(\Z_p)\otimes\Q.
\]
Fourier inversion for the integral equivariant APS index gives
\[
\ind_{\Z_p}^{\mathrm{APS}}(\dirac_{A_0})
=
\mathbf I^t(X)
+
N\,\mathbf n(Y,\mathfrak{s},\widetilde\tau,g),
\]
and therefore
\[
\mathbf n(Y,\mathfrak{s},\widetilde\tau,g)
=
\frac{\ind_{\Z_p}^{\mathrm{APS}}(\dirac_{A_0})-\mathbf I^t(X)}{N}
\in R(\Z_p)\otimes\Q.
\]

\medskip
\noindent{\bf \underline{Proof of (ii).}}
Using $\alpha_{\C}([\C_{[l]}])=1$ and Fourier orthogonality,
\[
\begin{aligned}
\alpha_{\C}\!\left(
\mathbf n(Y,\mathfrak{s},\widetilde\tau,g)
\right)
&=
\frac1p\sum_{k=0}^{p-1}
n(Y,\mathfrak{s},\widetilde\tau,g)_k
\sum_{l=0}^{p-1}\zeta_p^{-kl} =
n(Y,\mathfrak{s},\widetilde\tau,g)_0
=
n(Y,\mathfrak{s},g).
\end{aligned}
\]

\medskip
\noindent{\bf \underline{Proof of (iii).}}
Let $g_s$, $0\le s\le1$, be the given path.  On
$[0,1]\times Y$ use the complementary endpoint APS conditions
compatible with the spectral-flow convention of \Cref{appendix A}.
For $0<k<p$, the fixed surfaces are the cylinders
$[0,1]\times K_{k,i}$, and the equivariant APS formula gives
\[
\overline\eta^k(g_1)-\overline\eta^k(g_0)
+
\sum_iT_{k,i}
\int_{[0,1]\times K_{k,i}}\widetilde F_N
=
\operatorname{Sf}^k
\bigl(\diracpartial_{B_0}(g_s)\bigr).
\]
With the framing convention used in the torsion term,
\[
t(K_{k,i},g_1)-t(K_{k,i},g_0)
=
-2\pi
\int_{[0,1]\times K_{k,i}}\widetilde F_N.
\]
Substituting into
\[
n_k(g)=
\overline\eta^k(g)
-\frac1{2\pi}\sum_iT_{k,i}t(K_{k,i},g)
\]
gives
\[
n_k(g_1)-n_k(g_0)
=
\operatorname{Sf}^k
\bigl(\diracpartial_{B_0}(g_s)\bigr).
\]
The $k=0$ statement is the ordinary APS index/spectral-flow identity.
Fourier inversion now yields
\[
\mathbf n(Y,\mathfrak{s},\widetilde\tau,g_1)
-
\mathbf n(Y,\mathfrak{s},\widetilde\tau,g_0)
=
\operatorname{Sf}
\bigl(\diracpartial_{B_0}(g_s)\bigr).
\]
\end{proof}

Next, we prepare to define an equivariant stable homotopy category that contains our equivariant Seiberg--Witten Floer homotopy types.

\begin{defn}
A finite based $S^1 \times \Z_p$-equivariant CW complex $X$ is called a
\emph{space of type SWF} if the following conditions hold:
\begin{enumerate}
    \item $S^1$ acts freely on $X \smallsetminus X^{S^1}$;
    \item there exists a finite-dimensional real $\Z_p$-representation $V$
    such that $
    X^{S^1} \simeq V^+$
    $S^1 \times \Z_p$-equivariantly, where $S^1$ acts trivially on $V$.
\end{enumerate}
\end{defn}

Now, we define the category $\mathcal{C}^{sp}_{S^1 \times \Z_p}$ as follows.

\begin{defn}
The objects of $\mathcal{C}^{sp}_{S^1 \times \Z_p}$ are triples $(X, a, b)$, where
\begin{enumerate}
  \item $X$ is a space of type $S^1 \times \Z_p$-SWF,
  \item $a \in RO(\Z_p)$,
  \item $b \in R(\Z_p) \otimes \Q$. 
\end{enumerate}
Given two objects $(X,a,b)$ and $(X',a',b')$, we define the morphism set between them as 
\[
\mathrm{Mor}((X,a,b),(X',a',b')) 
= \left( \coprod_{\substack{\alpha - \alpha' = a' - a \\ \beta - \beta' = b' - b}} 
\left[ X \wedge \alpha^+ \wedge \beta^+, \, X' \wedge (\alpha')^+ \wedge (\beta')^+ \right]^{S^1 \times \Z_p} \right) \big/ \sim
\]
where $\alpha, \alpha'$ run over finite-rank real $\Z_p$-representations, $\beta, \beta'$ run over finite-rank complex $\Z_p$-representations, and for two maps
\[
f \colon X \wedge \alpha_1^+ \wedge \beta_1^+ \longrightarrow X' \wedge (\alpha'_1)^+ \wedge (\beta'_1)^+, 
\qquad 
g \colon X \wedge \alpha_2^+ \wedge \beta_2^+ \longrightarrow X' \wedge (\alpha'_2)^+ \wedge (\beta'_2)^+,
\]
we define $f \sim g$ if and only if there exist finite-rank real $\Z_p$-representations $\alpha''_1, \alpha''_2$, finite-rank complex $\Z_p$-representations $\beta''_1, \beta''_2$, and identifications 
\[
\begin{split}
    \alpha_1 \oplus \alpha''_1 &\cong \alpha_2 \oplus \alpha''_2, \\
    \alpha'_1 \oplus \alpha''_1 &\cong \alpha'_2 \oplus \alpha''_2, \\
    \beta_1 \oplus \beta''_1 &\cong \beta_2 \oplus \beta''_2, \\
    \beta'_1 \oplus \beta''_1 &\cong \beta'_2 \oplus \beta''_2,
\end{split}
\]
such that the maps 
\[
f \wedge \mathrm{id}_{(\alpha''_1)^+} \wedge \mathrm{id}_{(\beta''_1)^+}
\quad \text{and} \quad 
g \wedge \mathrm{id}_{(\alpha''_2)^+} \wedge \mathrm{id}_{(\beta''_2)^+}
\]
are $S^1 \times \Z_p$-equivariantly homotopic.
\end{defn}

We have the smash product operation
\[
\mathcal{C}^{sp}_{S^1 \times \Z_p} \times \mathcal{C}^{sp}_{S^1 \times \Z_p} \longrightarrow \mathcal{C}^{sp}_{S^1 \times \Z_p},
\]
which allows us to set
\[
SWF_{S^1 \times \Z_p}(Y, \s) = \bigwedge_i SWF_{S^1 \times \Z_p}(Y_i, \s_i)
\]
for a disjoint union of $\Z_p$-equivariant rational homology 3-spheres. The operation $-\wedge-$ makes $\mathcal{C}^{sp}_{S^1 \times \Z_p}$ into a symmetric monoidal category. Furthermore, the process of taking $S^1 \times \Z_p$-equivariant singular cochains also makes sense for objects of $\mathcal{C}^{sp}_{S^1 \times \Z_p}$ as follows.

\begin{defn}
The functor 
\[
C^\ast_{S^1 \times \Z_p}(-;\Z_p) \colon \mathcal{C}^{sp}_{S^1 \times \Z_p} \longrightarrow \mathrm{Mod}^{op}_{C^\ast(B(S^1 \times \Z_p);\Z_p)}
\]
is defined by
\[
C^\ast_{S^1 \times \Z_p}((X,a,b);\Z_p) 
= \widetilde{C}^\ast_{S^1 \times \Z_p}(X;\Z_p)[\alpha(a) + 2\alpha(b)],
\]
where $\alpha$ denotes the augmentation maps on $RO(\Z_p)$ and $R(\Z_p) \otimes \Q$.
\end{defn}

Rational grading shifts in this cochain notation are understood formally; morphisms preserve the total grading.

\begin{defn}\label{def:eqSWF}
We define the $S^1 \times \Z_p$-equivariant spectrum class by
\[
SWF_{S^1 \times \Z_p}(Y, \mathfrak{s}, \tilde{\tau})
:=
\left[
\left(
I^\mu_\lambda(g),
V^0_\lambda(g)(\mathbb{R}),
V^0_\lambda(g)(\mathbb{C})
+
\mathbf{n}(Y,\mathfrak{s},\tilde{\tau},g)
\right)
\right]
\]
as an isomorphism class of objects in the category
$\mathcal{C}^{sp}_{S^1 \times \Z_p}$.
We also define
\[
\wt{H}^*_{S^1\times \Z_p}
\bigl(SWF_{S^1\times \Z_p}(Y,\s,\tilde{\tau})\bigr)
:=
\wt{H}^{*+
2\alpha_{\C}\!\left(
\mathbf{n}(Y,\fraks,\tilde{\tau},g)
\right)}_{S^1\times \Z_p}
\left(
SWF_{S^1\times \Z_p}(Y,\s,\tilde{\tau},g)
\right),
\]
where $\alpha_{\C}\colon R(\Z_p)\otimes\Q \to \Q$
denotes the complex augmentation.
\end{defn}

\begin{rem}
The pair $(\s,\tilde\tau)$, where $\s$ is a $\mathrm{Spin}^c$ structure on $Y$ and $\tilde\tau$ is a lift of the $\Z_p$--action on $Y$ to $S$, determines a $\Z_p$--equivariant $\mathrm{Spin}^c$ structure on $Y$. Accordingly, we will often write $SWF_{S^1 \times \Z_p}(Y,\s)$ in place of $SWF_{S^1 \times \Z_p}(Y,\s,\tilde\tau)$, with the understanding that $\s$ denotes a $\Z_p$--equivariant $\mathrm{Spin}^c$ structure on $Y$.
\end{rem}

The equivariant correction term gives the following invariance statement, as in Montague's construction~\cite{montague2022seiberg}.

\begin{prop}
   The spectrum class $SWF_{S^1\times \Z_p}(Y, \mathfrak{s}, \tilde{\tau})$ is independent of the choices of $\Z_p$-invariant Riemannian metrics and finite dimensional approximations. 
\end{prop}
\begin{proof}
Changing the spectral cutoffs suspends the Conley index by the intervening eigenspaces. These are $\Z_p$-representations, and the same summands occur in $V^0_\lambda(g)$, so they cancel in \Cref{def:eqSWF}.

For a path of invariant metrics, subdivide the path so that common spectral cutoffs can be used on each subinterval. Equivariant continuation gives the usual identifications of the Conley indices. After removing $V^0_\lambda(g)$, their change is complex suspension by the equivariant spectral flow of the Dirac operators. By \Cref{prop:basic_of_n}(iii), this is exactly the change of $\mathbf n(Y,\mathfrak{s},\widetilde\tau,g)$, so it cancels in the normalized spectrum. The real linear part is treated as in \cite{montague2022seiberg}; on a rational homology sphere its relevant kernel is unchanged. Thus both choices leave the spectrum class unchanged.
\end{proof}

We also note that there is a duality formula for the equivariant Floer homotopy types. Similar to \cite[Proposition~6.13]{montague2022seiberg}, we have the following duality:

\begin{prop}
Let $(Y, \mathfrak{s}, \tilde{\tau})$ be a $\Z_p$-equivariant $\mathrm{Spin}^c$ rational homology sphere with a $\Z_p$-invariant Riemannian metric $g$, and let $(-Y, \mathfrak{s}, \tilde{\tau})$ denote its orientation reverse. Then
\begin{equation*} \label{eq:prop613}
\mathbf{n}(Y, \mathfrak{s}, \tilde{\tau}, g) + \mathbf{n}(-Y, \mathfrak{s}, \tilde{\tau}, g) = -\ker \diracpartial_{B_0} \in R(\Z_p),
\end{equation*}
where $\diracpartial_{B_0}$ is the 3-dimensional $\mathrm{Spin}^c$ Dirac operator with respect to a $\Z_p$-invariant flat connection $B_0$.
\end{prop}

Again, by the same argument given in \cite[Proposition~6.23]{montague2022seiberg}, we see the following.

\begin{prop}\label{Vduality}
The two spectra $SWF_{S^1 \times \Z_p}(Y, \mathfrak{s}, \tilde{\tau})$ and $SWF_{S^1 \times \Z_p}(-Y, \mathfrak{s}, \tilde{\tau})$ are $S^1 \times \Z_p$-equivariant $[S^0, 0, 0]$-duals. 
\end{prop}

\begin{lem}\label{PSC-vanish}
Let $(Y, \s, \tilde{\tau})$ be a $\Z_p$-equivariant $\mathrm{Spin}^c$ rational homology 3-sphere that admits a $\Z_p$-invariant positive scalar curvature metric $g$. Then
\[
SWF_{S^1 \times \Z_p}(Y, \s, \tilde{\tau}) = [(S^0, 0, \mathbf{n}(Y, \s, \tilde{\tau}, g))].
\]
\end{lem}
\begin{proof}
The Weitzenb\"ock formula excludes irreducible solutions and gives an invertible Dirac operator at the reducible, whose determinant connection is flat. Since $b_1(Y)=0$, this is the only critical point modulo gauge. Its Conley index is the sphere of the negative eigenspace of the linearization. Cancelling that eigenspace in \Cref{def:eqSWF} leaves the displayed correction-term suspension. The construction is equivariant because the metric and the reference connection are invariant.
\end{proof}

We observe that, by following the construction of $\mathcal{C}^{sp}_{S^1 \times \Z_p}$, we can define the space-level $S^1 \times \Z_p$-equivariant local equivalence group as follows.

\begin{defn}
Given two objects $(X,a,b)$ and $(X',a',b')$ of $\mathcal{C}^{sp}_{S^1\times \Z_p}$, a morphism $f$ between them, represented as an $S^1 \times \Z_p$-equivariant map
\[
f \colon X \wedge \alpha^+ \wedge \beta^+ \longrightarrow X' \wedge (\alpha')^+ \wedge (\beta')^+,
\]
is a \emph{local map} if $f^{S^1}$ is (non-equivariantly) a homotopy equivalence. We say that $(X,a,b)$ and $(X',a',b')$ are \emph{locally equivalent} if there exist local maps between them in both directions.
\end{defn}

\begin{defn}
We define
\[
\mathfrak{C}^{sp}_{S^1 \times \Z_p} 
= \frac{\left\{ \text{isomorphism classes of objects of } \mathcal{C}^{sp}_{S^1 \times \Z_p} \right\}}
       {\text{local equivalence}},
\]
where the group operation is given by $-\wedge-$. Equivariant Spanier--Whitehead duality gives the inverse, so $\mathfrak{C}^{sp}_{S^1 \times \Z_p}$ is an abelian group.
\end{defn}

\subsubsection{Recovering Baraglia--Hekmati's theory}\label{section:relation_to_BH}
We now compare our spectrum with the invariants of Baraglia--Hekmati~\cite{Baraglia-Hekmati:2024-1}.

Let $Y$ be a rational homology 3-sphere equipped with a free $\Z_p$-action $\tau \colon Y \to Y$ and a $\Z_p$-equivariant $\mathrm{Spin}^c$ structure $\tilde{\mathfrak{s}}$. Then we have the spectrum $SWF_{S^1\times \Z_p}(Y, \mathfrak{s}, \tilde{\tau}) \in \mathfrak{C}^{sp}_{S^1 \times \Z_p}$.

\begin{lem}\label{lem:recoveringBH}
If we take cohomology, we recover Baraglia--Hekmati's $S^1 \times \Z_p$-equivariant Floer cohomology:
\[
\wt{H}^{*+2\alpha_{\C}(\mathbf n(Y,\fraks,\tilde{\tau},g))}_{S^1\times\Z_p}
\left(SWF_{S^1\times\Z_p}(Y,\s,\tilde{\tau},g)\right)
\cong
H^*\left(C^*_{S^1\times\Z_p}(SWF_{S^1\times\Z_p}(Y,\s))\right).
\]
as modules over the ring $H^*_{S^1\times \Z_p} := H^*_{S^1\times \Z_p}(\ast; \Z_p)$, where $SWF_{S^1\times \Z_p}(Y, \s, \tilde{\tau}, g)$
is the metric-dependent Seiberg--Witten Floer homotopy type introduced in \cite{Baraglia-Hekmati:2024-1}.  

Note that
\[
H^*_{S^1\times \Z_p} =
\begin{cases}
  \Z_2[U, \theta] & \text{if $p=2$,} \\
  \Z_p[U, R, S] / (R^2) & \text{if $p>2$,}
\end{cases}
\]
where $\deg(U) = \deg(S) = 2$ and $\deg(\theta) = \deg(R) = 1$ and $SWF_{S^1\times \Z_p}({Y}, \s, \tilde{\tau}, g)$ denotes the metric-dependent Floer homotopy type.
\end{lem}

\begin{proof}
We use the identity
\[
\alpha_{\C}\left(\mathbf n(Y,\fraks,\tilde{\tau},g)\right)
=
n(Y,\fraks,g)
\]
from \Cref{prop:basic_of_n}(ii), together with the $\Z_p$-equivariant Thom isomorphism theorem with $\Z_p$ coefficients.
\end{proof}

\begin{rem}
The isomorphism class of the module 
\[
\wt{H}^*_{S^1\times \Z_p}\left( SWF(Y, \s, \tilde{\tau}) \right)
\]
depends on the choice of equivariant $\mathrm{Spin}^c$ structure, as pointed out in \cite{baraglia2024brieskorn}.
\end{rem}

Let us also review the equivariant Fr\o yshov invariants, developed by Baraglia and Hekmati in \cite{baraglia2024brieskorn}.  
Let $(X,a,b)$ be an element of $\mathcal{C}^{sp}_{S^1 \times \Z_p}$.  
The inclusion of the fixed points $\iota \colon X^{S^1} \to X$ induces a map
\[
\iota^* \colon U^{-1} H^*_{S^1\times \Z_p}\left(X;\Z_p\right) \longrightarrow U^{-1} \wt{H}^*_{S^1\times \Z_p}\left(X^{S^1}\right) 
   \cong U^{-1} H^*_{S^1 \times \Z_p}.
\]
We now recall the sequence of invariants $\{\delta_{G,j}(X,a,b)\}$:

\begin{itemize}
    \item[(i)] If $p=2$, we define
    \[
    \delta_{G,j}(X,a,b)
    :=
    \frac12
    \left(
    \min
    \left\{
    i\in\Q
    \;\middle|\;
    \begin{array}{c}
    \exists\,
    x\in
    \widetilde H^{\,i+\alpha_{\R}(a)+2\alpha_{\C}(b)}
    _{S^1\times\Z_2}(X;\Z_2),\\[1mm]
    \iota^\ast x
    \equiv
    U^k\theta^j
    \pmod{\theta^{j+1}}
    \text{ for some }k\ge0
    \end{array}
    \right\}
    -j
    \right).
    \]

    \item[(ii)] If $p>2$, we define
    \[
    \delta_{G,j}(X,a,b)
    :=
    \frac12
    \left(
    \min
    \left\{
    i\in\Q
    \;\middle|\;
    \begin{array}{c}
    \exists\,
    x\in
    \widetilde H^{\,i+\alpha_{\R}(a)+2\alpha_{\C}(b)}
    _{S^1\times\Z_p}(X;\Z_p),\\[1mm]
    \iota^\ast x
    \equiv
    U^kS^j
    \pmod{(S^{j+1},RS^{j+1})}
    \text{ for some }k\ge0
    \end{array}
    \right\}
    -2j
    \right).
    \]
\end{itemize}
We then set
\[
\delta_{G,j}\left(Y,\mathfrak{s}\right) := \delta_{G,j}\left(SWF(Y, \s, \tilde{\tau})\right) \in \Q.
\]
By \Cref{lem:recoveringBH}, this agrees with the equivariant Fr\o yshov invariants originally introduced in \cite{Baraglia-Hekmati:2024-1}.

\begin{rem}\label{rem:ofd}
Let $K$ be a knot in $S^3$. For each prime $p$, by \Cref{def:eqSWF} we obtain a metric-independent 
$S^1 \times \Z_p$--equivariant Seiberg--Witten Floer homotopy type of the $p$--fold branched cover $\Sigma_p(K)$ with its canonical
deck-invariant spin structure $\s_0$, equipped with a
$\Z_p$--lift:
\[
    SWF_{S^1\times \Z_p}(\Sigma_p(K),\s_0).
\]
For each $k=0,\dots,p-1$, let $G_k\subset S^1\times\Z_p$ be the
graph subgroup defined above.  We may then consider the geometric
$G_k$--fixed-point spectrum
\begin{align}\label{fix_pt_invariant}
SWF_{S^1\times\Z_p}(\Sigma_p(K),\s_0)^{G_k}.
\end{align}
which defines a knot invariant. At first sight, \eqref{fix_pt_invariant} may appear to depend on the choice of equivariant $\mathrm{Spin}^c$ structures. However, it can be canonically regarded as a knot invariant as follows.

\begin{itemize}
    \item When $p=2$, there are precisely two equivariant $\mathrm{Spin}^c$ structures ${\tau}_1$ and ${\tau}_2$ on $\s_0$ covering the deck transformation $\tau:\Sigma_2(K)\to \Sigma_2(K)$. These are given by
    \[
        \tau_1 = i \,\wt{\tau}, \qquad \tau_2 = -i \,\wt{\tau},
    \]
    where $\wt{\tau}$ is an order-four lift of $\tau$ to $\s_0$ commuting with the principal $\mathrm{Spin}(4)$--action, as observed in \cite{iida2024monopoles}. One then has $j \tau_1 = \tau_2 j$ on the configuration space of $\Sigma_2(K)$. Consequently,
    \[
        SWF_{S^1\times \Z_2}(\Sigma_2(K),\s_0)^{\tau_1}
        \qquad\text{ and }\qquad 
        SWF_{S^1\times \Z_2}(\Sigma_2(K),\s_0)^{\tau_2}
    \]
are homeomorphic after conjugating the residual $S^1$--action.
More precisely, the map induced by $j$ identifies the two spectra
equivariantly with respect to the automorphism $
S^1\longrightarrow S^1,\qquad z\longmapsto\overline{z}$.
    
    \item When $p$ is odd, it was shown in \cite[Proposition~2.2]{montague2022seiberg} that the $\Z_p$--lift $\wt{\tau}$ of the deck transformation to the principal $\mathrm{Spin}(4)$--bundle is uniquely determined (whereas $-\wt{\tau}$ gives an odd lift). Passing through the natural map $\mathrm{Spin}(4)\to \mathrm{Spin}^c(4)$ then yields a canonical $\Z_p$--equivariant $\mathrm{Spin}^c$ structure on $\s_0$.
\end{itemize}

Since this construction is compatible with orbifold gauge theory, we define the \emph{orbifold Seiberg--Witten Floer homotopy type} of $K$ by
\[
SWF^{(p),[k]}_{\mathrm{ofd}}(K)
:=
SWF_{S^1\times\Z_p}(\Sigma_p(K),\s_0)^{G_k},
\]
as an $S^1$-equivariant stable pointed homotopy type. For $p=2$, this invariant is understood up to conjugation of the residual $S^1$-action.
When $p=2$, this invariant is expected to be related to Jiakai Li's monopole Floer homology of webs \cite{li2023monopole} in the case with no real locus, which may be viewed as a version of monopole Floer homology for $3$-orbifolds with cone angle~$\pi$ obtained from knots.
From \cite[Theorem 1.16]{iida2024monopoles}, we see that $SWF^{(2), [1]}_{\mathrm{ofd}}(K)$ is a $\Z_2$--homology sphere for any $K \subset S^3$. 
\end{rem}

\section{Equivariant Bauer--Furuta invariants}

In this section, we shall discuss several properties of $S^1 \times \Z_p$-equivariant Bauer--Furuta invariants, which will be used to construct an equivariant map from the lattice homotopy type to the equivariant Seiberg--Witten Floer homotopy type.

Let $(X, \mathfrak{s})$ be a compact connected $\mathrm{Spin}^c$ 4-manifold bounded by a rational homology 3-sphere equipped with the restricted $\mathrm{Spin}^c$ structure $\mathfrak{t} = \mathfrak{s}|_Y$. Suppose $X$ is equipped with a smooth $\Z_p$-action such that the action preserves the isomorphism class of the $\mathrm{Spin}^c$ structure, and $b_1(X)=0$.  
Assume that $\mathfrak{s}$ is equipped with a chosen $\Z_p$-equivariant $\mathrm{Spin}^c$ structure $\widetilde{\mathfrak{s}}$ on $X$ extending the prescribed equivariant structure on $\mathfrak{t}$.  
We take a $\Z_p$-invariant Riemannian metric on $X$ which is product near the boundary and a reference $\Z_p$-invariant $\mathrm{Spin}^c$ connection $A_0$.  
For the action of $\tau^k$, suppose the fixed point set is described as the union of 0-dimensional components and 2-dimensional components:  
\[
X_0^{\tau^k} = \{ p_{k,1}, \ldots, p_{k,m} \}, \qquad 
X^{\tau^k}_{2} = \Sigma_{k,1} \sqcup \cdots \sqcup \Sigma_{k,n}. 
\]
In order to state the results, we introduce two topological invariants.
First,
\[
H^+_{\Z_p}(X)=H^+(X;\R)\in RO(\Z_p)
\]
denotes the real representation induced by the action on a maximal
positive subspace.

For $0<k<p$, define
\[
I_k(X,\mathfrak{s})
=
\sum_iR_{k,i}
+
\sum_j
\left(
S_{k,j}\langle c_1(\mathfrak{s}),[\Sigma_{k,j}]\rangle
+
T_{k,j}[\Sigma_{k,j}]^2
\right),
\]
using the relative characteristic numbers from
\Cref{rem: index thm for nonclosed fixed points} when a fixed surface
has boundary, and put
\[
I_0(X,\mathfrak{s})
=
\frac{c_1(\mathfrak{s})^2-\sigma(X)}8.
\]
The topological equivariant Dirac index is the Fourier reconstruction
\[
\ind^t_{\Z_p}\dirac_{X,\mathfrak{s}}
=
\frac1p
\sum_{l=0}^{p-1}
\left(
\sum_{k=0}^{p-1}
I_k(X,\mathfrak{s})\zeta_p^{-kl}
\right)
[\C_{[l]}].
\]

\begin{lem}\label{lem: rationality}
The following hold:
\begin{enumerate}
\item
$\ind^t_{\Z_p}\dirac_{X,\mathfrak{s}}
\in R(\Z_p)\otimes\Q$;
\item
\[
\alpha_{\C}
\left(
\ind^t_{\Z_p}\dirac_{X,\mathfrak{s}}
\right)
=
\frac{c_1(\mathfrak{s})^2-\sigma(X)}8;
\]
\item if $X$ is closed, then
$\ind^t_{\Z_p}\dirac_{X,\mathfrak{s}}\in R(\Z_p)$ and it coincides
with the $\Z_p$-equivariant Dirac index.
\end{enumerate}
\end{lem}

\begin{proof}
For $a\in(\Z/p\Z)^\times$, Galois conjugation
$\sigma_a(\zeta_p)=\zeta_p^a$ carries
$I_k(X,\mathfrak{s})$ to $I_{ak}(X,\mathfrak{s})$, while
$I_0(X,\mathfrak{s})\in\Q$.  Hence every Fourier coefficient in the
definition of $\ind^t_{\Z_p}\dirac$ is Galois fixed and therefore
rational.  This proves (1).

Applying the augmentation to the representation-basis Fourier formula
and using
\[
\sum_{l=0}^{p-1}\zeta_p^{-kl}
=
\begin{cases}
p,&k=0,\\
0,&k\ne0,
\end{cases}
\]
gives (2).  If $X$ is closed there is no boundary correction, and the
equivariant APS formula is the ordinary equivariant index theorem.
Fourier inversion then gives the actual equivariant index, proving
(3).
\end{proof}

Here and below, if $
\xi\in R(\Z_p)\otimes\Q$,
the notation $\xi^+$ is understood as a formal complex representation
suspension in the category $\mathcal{C}^{sp}_{S^1\times\Z_p}$.
If $\xi\in R(\Z_p)$ is an integral virtual representation, this agrees
with the corresponding representation sphere spectrum.

Now we state our result on the equivariant Bauer--Furuta invariants for equivariant $\mathrm{Spin}^c$ structures: 

\begin{prop}\label{prop:eqBF}
There is a well-defined $S^1\times\Z_p$-equivariant relative
Bauer--Furuta morphism
\[
BF_{S^1\times\Z_p}(X,\mathfrak{s})
\colon
\bigl(\ind^t_{\Z_p}\dirac_{X,\mathfrak{s}}\bigr)^+
\longrightarrow
H^+_{\Z_p}(X)^+
\wedge
SWF_{S^1\times\Z_p}(Y,\mathfrak{t}).
\]
After a common real stabilization and the standard identifications
coming from the real linear part, its $S^1$-fixed map is
$\Z_p$-equivariantly homotopic to the stable linear inclusion
\[
A^+
\longrightarrow
\bigl(A\oplus H^+_{\Z_p}(X)\bigr)^+
\]
for a finite-dimensional real $\Z_p$-representation $A$.
In particular, if $b^+(X)=0$, the fixed map is a
$\Z_p$-homotopy equivalence.  Forgetting the $\Z_p$-action recovers
the ordinary relative $S^1$-equivariant Bauer--Furuta invariant.
\end{prop}

A spin version is discussed in
\cite[Section~7.2]{montague2022seiberg}.

\begin{proof}
Choose a $\Z_p$-invariant product metric near the boundary and a
$\Z_p$-invariant reference connection $A_0$ which is flat there.
Using the double Coulomb slice, the monopole map together with the
negative spectral projection on the boundary has linearization
\[
(\dirac_{A_0}^+,d^+)
\]
with the usual boundary operator.  A sufficiently large
$\Z_p$-invariant finite-dimensional approximation gives a based
$S^1\times\Z_p$-map
\[
U^+
\longrightarrow
(U')^+\wedge I^\nu_\lambda.
\]
Its real and complex representation differences are
\[
r-r'
=
V^0_\lambda(\mathbb R)-H^+_{\Z_p}(X)
\in RO(\Z_p)
\]
and
\[
h-h'
=
V^0_\lambda(\mathbb C)
+
\ind_{\Z_p}^{\mathrm{APS}}
\dirac_{X,\mathfrak{s},A_0,g}
\in R(\Z_p).
\]
The boundary-normalization identity
\[
\ind_{\Z_p}^{\mathrm{APS}}\dirac
=
\ind^t_{\Z_p}\dirac
+
\mathbf n(Y,\mathfrak{t},\widetilde\tau,g)
\]
combines the complex suspension with the definition of
$SWF_{S^1\times\Z_p}(Y,\mathfrak{t})$ and produces the displayed
morphism.

On $S^1$-fixed configurations the spinor vanishes, so the monopole map
reduces, up to the standard finite-dimensional homotopies, to its real
linear part.  Its cokernel is $H^+(X;\R)$ with the induced
$\Z_p$-action.  After a common real stabilization the fixed map is
therefore
\[
A^+\longrightarrow
\bigl(A\oplus H^+_{\Z_p}(X)\bigr)^+.
\]
All operators, approximating spaces, and identifications can be chosen
$\Z_p$-equivariantly.  If $b^+(X)=0$, this becomes the identity after
stabilization.  Continuation and enlargement of the approximation give
well-definedness.  Applying the augmentation to the representation
gradings recovers the ordinary relative Bauer--Furuta construction.
\end{proof}

Similarly, by combining the duality maps stated in \Cref{Vduality}, one can also treat a $4$-manifold $X$ with several boundary components. We state the result without proof:  

\begin{prop}\label{prop:equ_BF_mu}
Let $X$ be a $\Z_p$-equivariant $\mathrm{Spin}^c$ cobordism from
$\bigsqcup_{i=1}^nY_i$ to $\bigsqcup_{j=1}^mY'_j$ with $b_1(X)=0$.
Assume all boundary components are rational homology spheres preserved
componentwise by the action, and fix a global equivariant
$\mathrm{Spin}^c$ structure on $X$ with the stated boundary
restrictions.  Then there is an $S^1\times\Z_p$-equivariant morphism
\[
BF_{S^1\times\Z_p}(X,\mathfrak{s})
\colon
\bigl(\ind^t_{\Z_p}\dirac\bigr)^+
\wedge
\bigwedge_{i=1}^nSWF(Y_i)
\longrightarrow
H^+_{\Z_p}(X)^+
\wedge
\bigwedge_{j=1}^mSWF(Y'_j).
\]
After a common real stabilization, its $S^1$-fixed part is the
standard $\Z_p$-equivariant stable linear inclusion determined by
$H^+_{\Z_p}(X)$.  In particular, it is a $\Z_p$-homotopy equivalence
when $b^+(X)=0$.
\end{prop}

\begin{rem}\label{rem:ofd_bf}
We record a remark related to orbifold Seiberg--Witten theory. As in \Cref{rem:ofd}, for a knot $K \subset S^3$, a prime $p$, and $[k]\in \Z_p$, we obtain an orbifold Seiberg--Witten Floer homotopy type 
\[
SWF^{(p), [k]}_{\mathrm{ofd}}(K)
\]
realized as a fixed-point spectrum.  

Suppose we have a properly embedded surface $S$ in
$[0,1]\times S^3\# X$ from $K$ to $K'$, where $X$ is a fixed
oriented closed $4$--manifold, such that
\[
[S]\in
H_2([0,1]\times S^3\#X,\partial;\Z)
\]
is divisible by $p$.  Assume moreover that the branched cover
$\Sigma_p(S)$, together with its boundary components, satisfies the
hypotheses of \Cref{prop:equ_BF_mu}; in particular,
\[
b_1(\Sigma_p(S))
=
b_1(\Sigma_p(K))
=
b_1(\Sigma_p(K'))
=
0,
\]
and $\Sigma_p(S)$ is equipped with a global
$\Z_p$--equivariant $\mathrm{Spin}^c$ structure
$\widetilde{\mathfrak{s}}$ whose boundary restrictions are the chosen
equivariant structures on $\Sigma_p(K)$ and $\Sigma_p(K')$.
Then \Cref{prop:equ_BF_mu} gives an
$S^1\times\Z_p$--equivariant map
\[
BF_{S^1 \times \Z_p}(\Sigma_p(S), \mathfrak{s}) \colon \ind^t_{\Z_p}\dirac^+ \wedge  SWF_{S^1\times \Z_p}(\Sigma_p(K),\s_0) 
\longrightarrow H^+_{\Z_p}(\Sigma_p(S))^+ \wedge SWF_{S^1\times \Z_p}(\Sigma_p(K'),\s_0'). 
\]
For any $[k]\in \Z_p$, we may take the fixed-point part: 
\[
BF_{S^1 \times \Z_p}(\Sigma_p(S), \mathfrak{s})^{G_k}
\colon
\bigl((\ind^t_{\Z_p}\dirac)^{G_k}\bigr)^+
\wedge SWF^{(p),[k]}_{\mathrm{ofd}}(K)
\longrightarrow
\bigl(H^+_{\Z_p}(\Sigma_p(S))^{\Z_p}\bigr)^+
\wedge SWF^{(p),[k]}_{\mathrm{ofd}}(K'),
\]
which we call the \emph{orbifold Bauer--Furuta invariant} for $S$, denoted $BF^{(p), [k]}_{\mathrm{ofd}}(S, \s)$.  

It is natural to expect that the fixed-point construction above admits
an extension to $4$--orbifolds with codimension-two singular locus,
without requiring a global branched cover.  We do not construct such
an orbifold Bauer--Furuta theory here.  For further discussion of orbifold Spin$^c$ structures and orbifold Seiberg--Witten theory with codimension-two singularities, see \cite{baldridge2001seiberg, chen2004pseudoholomorphic, chen2006smooth, chen2012seiberg, lebrun2015edges, chen2020class}. 
\end{rem}

Our goal for the rest of this section is to establish the $\Z_p$-equivariant adjunction relation stated in \Cref{prop: BF gluing}.
We first prove a gluing formula along a rational homology sphere, using the construction of Khandhawit--Lin--Sasahira~\cite[Sections~6.2--6.4]{khandhawit2023unfolded}; see also~\cite{manolescu2007gluing}. References to that paper use the numbering of arXiv:1809.09151v4.
We explain how their construction extends to the $S^1\times\Z_p$-equivariant setting.
Throughout, $BF_{S^1\times\Z_p}(M,\s_M)$ denotes the normalized relative invariant of \Cref{prop:eqBF,prop:equ_BF_mu}, with the representation suspensions specified there.
For several outgoing boundary components, $SWF_{S^1\times\Z_p}(\partial M)$ denotes their smash product; the empty product is $S^0$.

\begin{thm}\label{general gluing}
Let $(X_0,\s_0)$ and $(X_1,\s_1)$ be compact, connected, $\Z_p$-equivariant $\mathrm{Spin}^c$ $4$-manifolds with $b_1(X_0)=b_1(X_1)=0$ and
\[
\partial X_0=Z_0\sqcup Y',\qquad \partial X_1=-Y'\sqcup Z_1.
\]
Assume that $Y'$ is connected and all boundary components are rational homology spheres preserved componentwise by the action.
Fix a $\Z_p$-equivariant $\mathrm{Spin}^c$ identification along $Y'$, with compatible lifts, and put $(X,\s)=(X_0,\s_0)\cup_{Y'}(X_1,\s_1)$.
Choose invariant product metrics and invariant reference connections which are flat near the boundary and agree on $Y'$.
Under the gluing identifications
\[
\begin{aligned}
H^+_{\Z_p}(X)&\cong H^+_{\Z_p}(X_0)\oplus H^+_{\Z_p}(X_1),\\
\ind^t_{\Z_p}\dirac_{X,\s}
&=\ind^t_{\Z_p}\dirac_{X_0,\s_0}+\ind^t_{\Z_p}\dirac_{X_1,\s_1},
\end{aligned}
\]
one has
\begin{equation}\label{eq: equivariant BF pairing}
\begin{aligned}
BF_{S^1\times\Z_p}(X,\s)
\simeq{}&(\id\wedge\eta_{Y'}\wedge\id) \circ\bigl(BF_{S^1\times\Z_p}(X_0,\s_0)
       \wedge BF_{S^1\times\Z_p}(X_1,\s_1)\bigr).
\end{aligned}
\end{equation}
Here the target factors of the smash product are first reordered, by the canonical symmetry maps, as
\[
\begin{gathered}
\bigl(H^+_{\Z_p}(X_0)\oplus H^+_{\Z_p}(X_1)\bigr)^+
\wedge SWF_{S^1\times\Z_p}(Z_0) \wedge SWF_{S^1\times\Z_p}(Y')
\wedge SWF_{S^1\times\Z_p}(-Y')
\wedge SWF_{S^1\times\Z_p}(Z_1).
\end{gathered}
\]
The map $\eta_{Y'}$ is the duality map of \Cref{Vduality}, constructed from the index pairs for these boundary data.
The first identity acts on the $H^+$ and $Z_0$ factors, and the last on the $Z_1$ factor.
The homotopy is $S^1\times\Z_p$-equivariant.

\end{thm}

Along $Y'$ use a common collar coordinate, with $\nu$ pointing from $X_0$ to $X_1$; on $Z_i$ use the outward normal $\nu_i$.

We follow the proof in \cite{khandhawit2023unfolded}. The change needed here concerns the normalization of the harmonic gauge transformations. 
In \cite[Lemma~6.21]{khandhawit2023unfolded}, the functions satisfy $f_i(\hat o)=0$ at a chosen point $\hat o\in Y'$. We replace this condition by $\int_{Y'}f_i\,d\mathrm{vol}_{Y'}=0$, which is preserved by the $\Z_p$-action.

\begin{proof}[Proof of \Cref{general gluing}]
Put $G=S^1\times\Z_p$. We follow the proof of
\cite[Theorem~6.1, Sections~6.2--6.4]{khandhawit2023unfolded},
with $(Y_0,Y_1,Y_2)=(-Z_0,-Z_1,Y')$.
The homological hypothesis of that theorem is automatic in our setting.
Moreover, $b_1(X_i)=0$ and connectedness of $Y'$ imply
$H^1(X;\R)=H^1(X,Y';\R)=0$.
Thus the Picard tori are points, and the type-A and type-R invariants
reduce to the relative invariants used here.

Choose invariant metrics, reference connections, and Sobolev inner
products. Boundary restriction, the Seiberg--Witten maps, and the
spectral projections then commute with $G$.
We explain the modification needed for the change of boundary
conditions in \cite[Section~6.4, Step~5]{khandhawit2023unfolded}.

The harmonic gauge transformations in
\cite[Lemma~6.21]{khandhawit2023unfolded}
are specified by functions satisfying $f_i(\hat o)=0$ at a chosen
point $\hat o\in Y'$.
We replace this normalization by
$
\int_{Y'}f_i\,d\mathrm{vol}_{Y'}=0$,
keeping all the other conditions unchanged.
Since $f_0|_{Y'}=f_1|_{Y'}$, subtracting the common constant
$
c=\frac{1}{\operatorname{vol}(Y')}
  \int_{Y'}f_0\,d\mathrm{vol}_{Y'}$
from both functions identifies the space used in that lemma with
the space obtained by imposing the mean-zero normalization instead.
The inverse identification is obtained by subtracting their common
value at $\hat o$.
These identifications leave $df_i$ and the operator
$\rho_t(f_0,f_1)$ of the cited lemma unchanged.
Consequently, the existence, uniqueness, and continuous dependence
of the harmonic gauge transformations follow from that lemma.

The new normalization and all the operators defining $\rho_t$
commute with the $\Z_p$-action.
Uniqueness therefore gives
\[
\rho_t^{-1}(g\cdot v)=g\cdot\rho_t^{-1}(v)
\qquad (g\in\Z_p).
\]
Hence the harmonic gauge transformations used to change the boundary
conditions depend $\Z_p$-equivariantly on the configurations.
They also commute with the constant $S^1$-action.
Thus the c-homotopy of
\cite[Proposition~6.22]{khandhawit2023unfolded}
is $G$-equivariant with this normalization.
The passage to the double Coulomb slices proceeds as in
\cite[Lemma~6.24 and Corollary~6.25]{khandhawit2023unfolded}.

All remaining deformations in the proof use equivariant operators
and can be made with invariant auxiliary data.
Taking the finite-dimensional approximation spaces and index pairs
$G$-invariant, the comparison in
\cite[Section~6.4, Step~7]{khandhawit2023unfolded}
and the estimate of its Lemma~6.26 apply without change.
We therefore obtain a $G$-equivariant stable homotopy between the
relative invariant of $X$ and the duality pairing of the relative
invariants of $X_0$ and $X_1$.

Finally, the intersection forms split equivariantly across the
rational homology sphere $Y'$, giving the asserted identification
of the $H^+$-representations.
For the complex suspensions, put $K=\ker\diracpartial_{Y'}$.
With our nonpositive APS convention, equivariant index gluing gives
\[
\ind^{\mathrm{APS}}_{\Z_p}\dirac_{X,\s}
=
\ind^{\mathrm{APS}}_{\Z_p}\dirac_{X_0,\s_0}
+
\ind^{\mathrm{APS}}_{\Z_p}\dirac_{X_1,\s_1}
+[K].
\]
The normalization
$\ind^{\mathrm{APS}}_{\Z_p}\dirac_M
=\ind^t_{\Z_p}\dirac_M+\mathbf n(\partial M)$
and the identity
$\mathbf n(Y')+\mathbf n(-Y')=-[K]$
cancel the kernel term.
Thus the normalized indices add in $R(\Z_p)\otimes\Q$.
Under the same normalization, the duality pairing is
$\eta_{Y'}$ of \Cref{Vduality}, proving
\eqref{eq: equivariant BF pairing}.
\end{proof}

We further suppose that $Y'$ is orientation-preserving diffeomorphic to $L(n,1)$ for some $0 \neq n \in \Z$, and that the $\Z_p$-action on $Y'$ is induced by a linear $S^1$-action on the total space of the disk bundle $O(n) \to S^2$. These actions preserve the positive scalar curvature metric on $L(n,1)$.

\begin{defn} \label{defn: augmentation map in group ring}
For any element $\mathbf{n} = \sum_{i=0}^{p-1} n_i \cdot [i] \in \Z[\Z_p]$, we associate the $S^1 \times \Z_p$-representation
\[
\bigoplus_{i=0}^{p-1} \C_{[i]}^{\,n_i}.
\]
For simplicity, we sometimes abbreviate this as $\C^{\mathbf{n}}$. 
Given $\mathbf{m}, \mathbf{n} \in \Z[\Z_p]$ (or $\Q[\Z_p]$), we write $\mathbf{m} \geq \mathbf{n}$ if $\mathbf{m} - \mathbf{n}$ has nonnegative coefficients. These definitions extend naturally to elements of $\Q[\Z_p]$ as well.
\end{defn}

Before moving on, we recall some facts about $S^1 \times \Z_p$-representations and certain subgroups of $S^1 \times \Z_p$.  
For each $k = 0, \dots, p-1$, consider the order $p$ subgroup
\[
G_k
=
\left\{
\left(
e^{-\frac{2\pi i\ell k}{p}},
[\ell]
\right)
\;\middle\vert\;
\ell\in\Z
\right\}
\subset S^1\times\Z_p.
\]
Then $
(\C^{\mathbf n})^{G_k}=\C^{n_k}$.

\begin{lem} \label{lem: stabilizers in S1xZp}
Suppose $S^1 \times \Z_p$ acts continuously on a topological space $X$. Assume that the induced $S^1$-action has no finite stabilizers on $X$, i.e., the only stabilizers are $1$ or $S^1$. Then for any $x \in X$, the stabilizer of $x$ under the $S^1 \times \Z_p$-action is one of the following:
\[
1, \qquad S^1, \qquad S^1 \times \Z_p, \qquad G_k \ (k=0,\dots,p-1).
\]
\end{lem}

\begin{proof}
Let
\[
H\subset S^1\times\Z_p
\]
be the stabilizer of a point $x\in X$.

If $
S^1\times\{0\}\subset H$,
then the image of $H$ under the projection
$
S^1\times\Z_p\longrightarrow\Z_p$
is either trivial or all of $\Z_p$.  Hence
\[
H=S^1
\qquad\text{or}\qquad
H=S^1\times\Z_p.
\]

Now suppose that $
H\cap(S^1\times\{0\})=\{1\}$.
Then the projection $
\pi\colon H\longrightarrow\Z_p $
is injective.  Since $p$ is prime, its image is either trivial or all
of $\Z_p$. If $\pi(H)=\{0\}$, then $H=1$. If $\pi(H)=\Z_p$, then $H$ is the graph of a homomorphism $ 
\Z_p\to S^1$. Every such homomorphism has the form $
[\ell]\longmapsto e^{-2\pi i k\ell/p}$
for a unique $k\in\{0,\dots,p-1\}$.  Therefore
\[
H
=
\left\{
\left(
e^{-2\pi i k\ell/p},
[\ell]
\right)
\;\middle\vert\;
\ell\in\Z
\right\}
=
G_k.
\]
This proves the claim.
\end{proof}

The following gives the $\Z_p$-equivariant adjunction relation. 
\begin{thm} \label{prop: BF gluing}
Let $(X,\s)$ be a $\Z_p$-equivariant $\mathrm{Spin}^c$ $4$-manifold
satisfying the hypotheses for the equivariant Bauer--Furuta invariant,
including $b_1(X)=0$, and suppose that its boundary is a disjoint union
of $\Z_p$-equivariant $\mathrm{Spin}^c$ rational homology $3$-spheres.

Let $S\subset X$ be a negatively embedded $2$-sphere such that a
$\Z_p$-invariant disk neighborhood $\nu(S)$ carries the restriction of
a linear circle action of the kind used in Section~5.1.

Let $\s'$ be another $\Z_p$-equivariant $\mathrm{Spin}^c$ structure
satisfying
\[
\s' \big|_{X \smallsetminus \overset{\circ}{\nu}(S)}
=
\s \big|_{X \smallsetminus \overset{\circ}{\nu}(S)}.
\]

Suppose that
\[
\ind^t_{\Z_p}\dirac_{X,\s'}
-
\ind^t_{\Z_p}\dirac_{X,\s}
\geq 0
\]
coefficientwise in $R(\Z_p)$. Then, after the common representation
suspension identifications,
\[
BF_{S^1\times\Z_p}(X,\s)
\simeq
BF_{S^1\times\Z_p}(X,\s')
\circ
U^{
\ind^t_{\Z_p}\dirac_{X,\s'}
-
\ind^t_{\Z_p}\dirac_{X,\s}
}.
\]

If instead
\[
\ind^t_{\Z_p}\dirac_{X,\s}
-
\ind^t_{\Z_p}\dirac_{X,\s'}
\geq 0,
\]
the analogous statement holds with $\s$ and $\s'$ interchanged.
\end{thm}

We also need a lemma regarding equivariant metrics of positive scalar curvature on lens spaces.

\begin{lem} \label{lem: equivariant PSC metrics on lens spaces}
For any integer $n \neq 0$ and any finite-order diffeomorphism $\tau$ of a lens space $Y= \partial O(-n)$, there exists a $\tau$-equivariant metric of positive scalar curvature on $Y$.
\end{lem}

\begin{proof}
The subgroup $\langle \tau \rangle \subset \mathrm{Diff}(Y)$ is finite, and $Y$ is a spherical space form.  
Thus the lemma follows from \cite[Theorem~1.1]{chow2024equivariant}.
\end{proof}

Now we prove \Cref{prop: BF gluing}.

\begin{proof}[Proof of \Cref{prop: BF gluing}]
Put $G=S^1\times\Z_p$ and write
\[
N=\nu(S),\qquad
Q=X\smallsetminus\operatorname{int}(N),\qquad
Y'=\partial N.
\]
Let $\mathfrak{t}$ be the common equivariant restriction of
$\mathfrak{s}$ and $\mathfrak{s}'$ to $Y'$.

By \Cref{lem: equivariant PSC metrics on lens spaces},
choose a $\Z_p$-invariant positive scalar curvature metric on $Y'$.
Use the same boundary metric and flat reference connection for both
$\mathrm{Spin}^c$ structures, and extend these data over $N$ and $Q$,
with product structures near $Y'$.
By \Cref{PSC-vanish}, fix the resulting identification
\[
SWF_G(Y',\mathfrak{t})
\simeq
(\C^{\mathbf r})^+
\]
for a formal rational representation class $\mathbf r$,
using the fixed-point normalization of \Cref{prop:eqBF}.

Write
\[
\mathbf a
=
\ind^t_{\Z_p}\dirac_{N,\mathfrak{s}|_N},
\qquad
\mathbf a'
=
\ind^t_{\Z_p}\dirac_{N,\mathfrak{s}'|_N}.
\]
Since the boundary data agree, we see $
\mathbf a-\mathbf r,\,
\mathbf a'-\mathbf r
\in R(\Z_p).$
Moreover, index additivity in \Cref{general gluing} gives
\[
\Delta:=\mathbf a'-\mathbf a
=
\ind^t_{\Z_p}\dirac_{X,\mathfrak{s}'}
-
\ind^t_{\Z_p}\dirac_{X,\mathfrak{s}}.
\]

Suppose first that $\Delta\geq0$ coefficientwise.
Let $U^\Delta$ denote the standard representation-sphere inclusion,
obtained by suspending the map
$
S^0\longrightarrow(\C^\Delta)^+$
which sends the non-basepoint to the origin.
Consider
\[
\begin{aligned}
f=BF_G(N,\mathfrak{s}|_N),\quad g=BF_G(N,\mathfrak{s}'|_N)\circ U^\Delta.
\end{aligned}
\]
These are stable $G$-maps with the same source and target:
\[
f,g\colon
(\C^{\mathbf a})^+
\longrightarrow
(\C^{\mathbf r})^+.
\]

Since $b_1(N)=b^+(N)=0$, \Cref{prop:eqBF} identifies the
$S^1$-fixed restrictions of both relative invariants with the identity,
after a common real stabilization.
These identifications use the same real linear part and boundary data.
The map $U^\Delta$ also restricts to the identity on $S^1$-fixed points.
Thus $f^{S^1}$ and $g^{S^1}$ satisfy the hypothesis of
\Cref{lem: adjunction}.
After formally desuspending by $\mathbf r$, that lemma applies to
integral virtual representation spheres and yields
\[
BF_G(N,\mathfrak{s}|_N)
\simeq
BF_G(N,\mathfrak{s}'|_N)\circ U^\Delta.
\]

The restrictions of $\mathfrak{s}$ and $\mathfrak{s}'$ to $Q$ agree.
Furthermore, Mayer--Vietoris across the connected rational homology
sphere $Y'$ gives $b_1(Q)=b_1(X)=0$.
Apply \Cref{general gluing} to $X=N\cup_{Y'}Q$ and pair the preceding
homotopy with the common relative invariant of $Q$.
This gives
\[
BF_G(X,\mathfrak{s})
\simeq
BF_G(X,\mathfrak{s}')\circ U^\Delta,
\]
under the common representation suspension identifications. If $-\Delta\geq0$, interchange $\mathfrak{s}$ and $\mathfrak{s}'$. This completes the proof. 
\end{proof}

\section{Equivariant lattice homotopy type}

\subsection{$S^1$-action on plumbed 4-manifolds} \label{subsec: eqv plumbing action}

Given an integer $n$, consider the disk bundle $p \colon O(n) \to S^2$ of Euler number $n$.  
Choose closed disks $D_+, D_- \subset S^2$ such that $D_+ \cup D_- = S^2$ and $D_+ \cap D_-$ is a circle.  
We also choose trivializations
\[
p|_{p^{-1}(D_\pm)} \colon p^{-1}(D_\pm) \cong D^2 \times D_{\pm} \xrightarrow{(x,y) \mapsto y} D_{\pm}.
\]
Then $p^{-1}(D_+)$ and $p^{-1}(D_-)$ are glued along their boundaries as follows:
\[
p^{-1}(\partial D_+) \cong D^2 \times S^1 \xrightarrow{(z,e^{i\theta}) \mapsto (e^{in\theta}z, e^{i\theta})} 
D^2 \times S^1 \cong p^{-1}(\partial D_-).
\]
Thus, to construct an $S^1$-action on $O(n)$, it suffices to define $S^1$-actions on $p^{-1}(D_\pm)$ and verify that, when restricted to $\partial D_\pm$, the action commutes with the gluing map.  
In this subsection, we describe $S^1$-actions on $O(n)$ arising in this way and explain their relation to Seifert fibered spaces.

First consider the $S^1$-action on $p^{-1}(D_\pm)$ defined by
\[
e^{i\theta} \cdot (z,w) = \left(e^{i\theta}z, w\right), \quad (z,w) \in D^2 \times D_\pm.
\]
This action clearly commutes with the gluing map and therefore defines a smooth $S^1$-action on the total space $O(n)$.  
We call this the \emph{fiber rotation} (or \emph{fiberwise rotation}).

In general, for any pair of parameters $(p_+,q_+), (p_-,q_-) \in \Z^2 \smallsetminus \{(0,0)\}$, we may consider the action
\[
e^{i\theta} \cdot (z,w) = \left(e^{ip_\pm \theta}z,\, e^{iq_\pm \theta} w\right), 
\qquad (z,w) \in D^2 \times D_\pm.
\]
This action commutes with the gluing map if and only if the following linear relation holds:
\[
\begin{pmatrix}
    p_- \\ q_-
\end{pmatrix}
=
\begin{pmatrix}
    1 & n \\ 0 & 1
\end{pmatrix}
\begin{pmatrix}
    p_+ \\ q_+
\end{pmatrix}.
\]
We call the $S^1$-actions on $O(n)$ arising in this way \emph{linear actions}.  
We also call a self-diffeomorphism of $O(n)$ arising in this way a \emph{linear diffeomorphism}.  
Note that all linear actions on $O(n)$ preserve the zero-section of $p$ setwise.  
Clearly, the fiberwise rotation is a linear action.

Now let $Y$ be a Seifert fibered rational homology sphere.   Let $\Gamma$ be the normalized star-shaped negative-definite almost rational plumbing graph determined by the fixed Seifert presentation of $Y$, so that $Y\cong\partial W_\Gamma$. Note that $W_\Gamma$ is obtained by gluing disk bundles, i.e.,
\[
W_\Gamma = \bigcup_{v \in V(\Gamma)} D_v.
\]
We endow the disk bundle $D_{v_c}$ associated to the central node $v_c$ of $\Gamma$ with the fiberwise rotation.  
Then Orlik \cite[Section~2, Corollary~5]{Orlik:1972-1} showed that there exist unique linear actions on $D_v$ for each $v \in V(\Gamma) \smallsetminus \{v_c\}$ such that they glue together to a well-defined smooth $S^1$-action on the entire $4$-manifold $W_\Gamma$.  
Furthermore, the induced $S^1$-action on $Y = \partial W_\Gamma$ is fixed-point-free and coincides with the Seifert action of $Y$.

\subsection{$\mathbb{Z}_p$-equivariant $\mathrm{Spin}^c$ structures}
\label{subsec: eqv spin c structure classification}

Choose any prime $p$. We first recall the definition of $\Z_p$-equivariant $\mathrm{Spin}^c$ structures.
\begin{defn}
Given a connected smooth oriented $n$-manifold $X$ with an orientation-preserving smooth $\Z_p$-action generated by $\tau$, a $\Z_p$-equivariant $\mathrm{Spin}^c$ structure consists of a $\mathrm{Spin}^c$ structure $\mathfrak{s}$ and a lift of the action to its principal bundle $P(\mathfrak{s})$, fitting into the diagram
\[
\begin{CD}
P(\mathfrak{s}) @>{\widetilde\tau}>> P(\mathfrak{s})\\
@V{\pi}VV @VV{\pi}V\\
\mathrm{Fr}(X) @>{\tau_*}>> \mathrm{Fr}(X),
\end{CD}
\qquad \widetilde\tau^p=\id.
\]
Here $\mathrm{Fr}(X)$ is the oriented orthonormal frame bundle for an invariant metric, taken to be a product near the boundary, and $\widetilde\tau$ commutes with the right $\mathrm{Spin}^c(n)$-action.
\end{defn}

Isomorphisms are required to commute with the group actions and to cover the identity on the frame bundle. We describe the classification using equivariant classifying spaces, beginning with the properties that we need.

\begin{rem}
Let $G$ and $H$ be compact Lie groups and let $X$ be a $G$-CW complex. A $G$-equivariant principal $H$-bundle is a principal $H$-bundle over $X$ with a left $G$-action commuting with its right $H$-action. There is a universal such bundle
\[
E_GH\longrightarrow B_GH
\]
and a natural bijection between $[X,B_GH]^G$ and the set of equivariant isomorphism classes of principal $H$-bundles over $X$; see \cite{husemoller1966fibre,lashof1982equivariant,lashof1986generalized,may1990some}.

We recall the join construction used for this classification. For each closed subgroup $K\subset G$ and homomorphism $\rho\colon K\to H$, take the equivariant principal bundle $G\times_K H\to G/K$, where $K$ acts on $H$ by left multiplication through $\rho$. Choose one representative of each conjugacy class of these data and put
\[
E=\bigsqcup_{(K,\rho)}G\times_K H,
\qquad E_GH=E^{*\infty},
\qquad B_GH=E_GH/H.
\]
The left $G$-action and right $H$-action extend to the join, giving the universal bundle above. This is the product case of the more general classification of principal $(H;\Gamma)$-bundles for an extension $1\to H\to\Gamma\to G\to1$; only the product case is used here.
\end{rem}

The homomorphism $\mathrm{Spin}^c(n)\to SO(n)$ induces a map of equivariant classifying spaces. Thus an equivariant $\mathrm{Spin}^c$ structure is described by a lift of the classifying map of $\mathrm{Fr}(X)$, together with the identification of its associated $SO(n)$-bundle with $\mathrm{Fr}(X)$. Once one such lift is chosen, the difference between two lifts is an equivariant line bundle. We spell out this last point to specify the identifications used below.

\begin{lem}\label{lem: eqv line bundles and spin c torsor}
Let $G=\Z_p$. Equivariant complex line bundles on $X$ are classified by their equivariant first Chern class, or equivalently by the equivariant Euler class of their underlying oriented real plane bundle:
\[
\operatorname{Pic}_G(X)
\cong [X,B_GU(1)]^G
\xrightarrow{\ c_1^G\ } H^2_G(X;\Z).
\]
If $\mathrm{Spin}^c_G(X)$ is nonempty, it is a torsor for this group under tensoring with equivariant line bundles. After choosing a reference structure $\widetilde{\mathfrak{s}}_0$, the resulting bijection
\[
H^2_G(X;\Z)\longrightarrow\mathrm{Spin}^c_G(X),
\qquad c_1^G(E)\longmapsto\widetilde{\mathfrak{s}}_0\otimes E
\]
commutes with restriction and with forgetting the action, provided that the reference structures are chosen compatibly.
\end{lem}
\begin{proof}
Since $U(1)$ is a torus, \cite[Corollary~1.6]{rezk2018classifying} identifies the equivariant classification with the classification of line bundles on the Borel construction:
\[
[X,B_GU(1)]^G
\cong[EG\times_GX,BU(1)]
\cong H^2_G(X;\Z).
\]
The last map sends $E$ to $c_1(EG\times_GE)$ and is compatible with tensor products and pullbacks.

Let $P_0$ and $P$ represent two equivariant $\mathrm{Spin}^c$ structures. At each $x\in X$, the $\mathrm{Spin}^c(n)$-equivariant isomorphisms $(P_0)_x\to P_x$ covering the identity on the frame fiber form a $U(1)$-torsor. These torsors assemble into a principal $U(1)$-bundle. It inherits the action
\[
(g\cdot\Phi)(p)=g\cdot\Phi(g^{-1}\cdot p),
\]
and hence is the unit circle bundle of an equivariant line bundle $E$. Evaluation identifies $P$ with $P_0\otimes E$. Conversely, applying this construction to $P_0\otimes E$ recovers $E$, proving that twisting is free and transitive. All these constructions commute with restriction and with forgetting the action.
\end{proof}

In particular, the identifications with cohomology require a reference $\mathrm{Spin}^c$ structure; they are not canonical identifications of groups. The class used here is that of the difference line bundle, since
\[
c_1^G\bigl(\det(\widetilde{\mathfrak{s}}_0\otimes E)\bigr)
=c_1^G(\det\widetilde{\mathfrak{s}}_0)+2c_1^G(E).
\]
Thus the determinant class alone need not distinguish the structures when $H^2_G(X;\Z)$ has $2$-torsion.

We write $\mathcal N\colon\mathrm{Spin}^c_{\Z_p}(X)\to\mathrm{Spin}^c(X)$ for the forgetful map. The following observation gives an alternative description of equivariant spin$^c$ structures. 

\begin{lem}\label{lem: order p equivariant lift torsor}
Assume that $H^1(X;\Z)=0$ and $\tau^*\mathfrak{s}\cong\mathfrak{s}$. Then $\mathfrak{s}$ admits an order-$p$ lift. The isomorphism classes of these lifts form a free transitive $\Z_p$-set under character twisting.
\end{lem}
\begin{proof}
Choose a principal-bundle lift $\widetilde\tau_0$ of $\tau$. Its $p$th power is a gauge transformation $u\colon X\to S^1$ with $u\circ\tau=u$. Since $H^1(X;\Z)=0$, write $u=e^{2\pi ih}$. The function $h\circ\tau-h$ is a constant integer, which is zero by summing around an orbit. Consequently $e^{-2\pi ih/p}\widetilde\tau_0$ has order $p$.

For two order-$p$ lifts, write their ratio as $e^{2\pi if}$. The order condition gives $\sum_{j=0}^{p-1}f(\tau^jx)=m\in\Z$. After subtracting $m/p$ from $f$, this sum is zero, and
\[
h(x)=\frac1p\sum_{j=0}^{p-1}j f(\tau^jx)
\qquad\text{satisfies}\qquad h\circ\tau-h=f.
\]
The corresponding gauge transformation identifies the two lifts after twisting by the character $\zeta_p^m$. Finally, a gauge transformation identifying a lift with its $k$-twist has a logarithm, and its transformation equation implies $k=0$ in $\Z_p$ by summing around an orbit. Thus the action is free and transitive.
\end{proof}

We first consider a disk bundle $X$ over $S^2$ with a linear $S^1$-action, restricted to $\Z_p$. Let $S_X$ be its zero-section. The circle action preserves every underlying $\mathrm{Spin}^c$ structure, and $H^1(X;\Z)=0$, so equivariant lifts exist by the preceding lemma. Since $X$ equivariantly retracts onto $S_X$, \Cref{lem: eqv line bundles and spin c torsor} gives, after choosing reference structures, the commutative square
\[
\xymatrix{
\mathrm{Spin}^c_{\Z_p}(X)\ar[r]^{\mathcal N}\ar[d]^\cong
&\mathrm{Spin}^c(X)\ar[d]^\cong\\
H^2_{\Z_p}(S_X;\Z)\ar[r]&H^2(S_X;\Z).
}
\]
The action on $S_X$ is either trivial or a rotation. It acts trivially on cohomology, and the Serre spectral sequence of $S_X\to (S_X)_{h\Z_p}\to B\Z_p$ has
\[
E_2^{i,j}=H^i(B\Z_p;\Z)\otimes H^j(S_X;\Z).
\]
Both factors are supported in even degrees, so all differentials vanish. In degree two we obtain
\[
0\longrightarrow H^2(B\Z_p;\Z)
\longrightarrow H^2_{\Z_p}(S_X;\Z)
\longrightarrow H^2(S_X;\Z)\longrightarrow0.
\]
The last group is free, so this sequence splits. Its kernel is generated by the equivariant line bundle with character $\zeta_p$. We have proved the following lemma.

\begin{lem}\label{lem: eqv spin c over disk bundles}
Let $X$ be a disk bundle over $S^2$ with a linear $S^1$-action, restricted to $\Z_p$. After choosing reference lifts, there is a bijection
\[
\mathrm{Spin}^c_{\Z_p}(X)\cong\mathrm{Spin}^c(X)\times\Z_p,
\]
whose first coordinate is the underlying $\mathrm{Spin}^c$ structure. On each fiber of the forgetful map, character twisting acts by translation on the second coordinate.
\end{lem}

\begin{cor}\label{cor: eqv lift of spin c given by inv point}
Let $U\subset X$ be a setwise invariant contractible open subset. Given $\mathfrak{s}_X\in\mathrm{Spin}^c(X)$ and $\widetilde{\mathfrak{s}}_U\in\mathrm{Spin}^c_{\Z_p}(U)$ with $\mathcal N(\widetilde{\mathfrak{s}}_U)=\mathfrak{s}_X|_U$, there is a unique equivariant lift $\widetilde{\mathfrak{s}}_X$ with these underlying and restricted structures.
\end{cor}
\begin{proof}
Since $U$ is contractible, the projection $U_{h\Z_p}\to B\Z_p$ induces an isomorphism on cohomology. Hence restriction identifies the character subgroup of $H^2_{\Z_p}(X;\Z)$ with $H^2_{\Z_p}(U;\Z)\cong\Z_p$. The assertion follows from the torsor description.
\end{proof}

We now describe the second coordinate in terms of the lift at a fixed point $x$. An oriented complex decomposition of $T_xX$ reduces the structure group of the frame bundle to $T=SO(2)\times SO(2)$. The corresponding reduction of the $\mathrm{Spin}^c$ bundle has structure group
\[
\widetilde T
=\frac{U(1)\times U(1)\times U(1)}{\{\pm(1,1,1)\}},
\qquad [(x,y,z)]\longmapsto(xy,xy^{-1})\in T.
\]
The kernel is the central $U(1)$. We fix the local complex coordinates, and the orientations on the fixed surfaces, throughout this discussion.

Suppose first that the fixed component through $x$ is a surface and the normal action is multiplication by $\zeta_p^k$, where $0<k<p$. All order-$p$ lifts are
\begin{equation}\label{eqn: eqv number in codim 2}
[(x,y,z)]\longmapsto
\left[\left(\zeta_{2p}^{k}x,\zeta_{2p}^{-k}y,
\zeta_{2p}^{k+2m}z\right)\right],\qquad m\in\Z_p.
\end{equation}
If $x$ is an isolated fixed point with tangent action $(z,w)\mapsto(\zeta_p^{k_1}z,\zeta_p^{k_2}w)$, where $0<k_1,k_2<p$, the lifts are
\begin{equation}\label{eqn: eqv number in codim 4}
[(x,y,z)]\longmapsto
\left[\left(\zeta_{2p}^{k_1+k_2}x,\zeta_{2p}^{k_1-k_2}y,
\zeta_{2p}^{k_1+k_2+2m}z\right)\right],\qquad m\in\Z_p.
\end{equation}
In either case, the $p$th powers of the three multipliers are the same sign, so the lift has order $p$ in $\widetilde T$. The displayed lifts differ by the central characters and exhaust all possibilities.

\begin{defn}\label{def: eqv number and det line bundle}
The parameter $m$ in \eqref{eqn: eqv number in codim 2} or \eqref{eqn: eqv number in codim 4} is the \emph{equivariance number} $n^x_{\mathrm{eqv}}(\widetilde{\mathfrak{s}})$.
\end{defn}

The equivariance number is constant on a connected oriented fixed component. In particular, for a fiberwise rotation the fixed set is the zero-section, and we omit $x$ from the notation. In general, we continue to use the specified fixed point; a change of fixed point may depend on the underlying $\mathrm{Spin}^c$ structure.

\begin{rem}\label{rem: defining eqv number using det line bundle}
The spinor characters at an isolated fixed point are
\[
S^+_x:\ (\zeta_p^{k_1+k_2+m},\zeta_p^m),
\qquad S^-_x:\ (\zeta_p^{k_1+m},\zeta_p^{k_2+m}).
\]
Thus the determinant character is $\zeta_p^{k_1+k_2+2m}$, and at a fixed surface it is $\zeta_p^{k+2m}$. For odd $p$ these determine $m$ once the tangent weights are fixed. For $p=2$ the determinant character does not detect $m$, which is why we use the local lifts above.
\end{rem}

Recall the twisting operation: for $[k]\in\Z_p$, multiply the lift of the generator by the central scalar $\zeta_p^k$. Denote the resulting structure by $\mathrm{tw}_k(\widetilde{\mathfrak{s}})$. This operation leaves $\mathcal N(\widetilde{\mathfrak{s}})$ unchanged and satisfies
\[
n^x_{\mathrm{eqv}}(\mathrm{tw}_k(\widetilde{\mathfrak{s}}))
=n^x_{\mathrm{eqv}}(\widetilde{\mathfrak{s}})+[k].
\]
Consequently the equivariance number is surjective on each fiber of $\mathcal N$. By \Cref{lem: eqv spin c over disk bundles}, each such fiber has $p$ elements, so it is also injective.

\begin{lem}\label{lem: equivariance number}
For a disk bundle as above and a specified fixed point with fixed local coordinates, the map
\[
\widetilde{\mathfrak{s}}\longmapsto
\bigl(\mathcal N(\widetilde{\mathfrak{s}}),n^x_{\mathrm{eqv}}(\widetilde{\mathfrak{s}})\bigr)
\]
is a bijection $\mathrm{Spin}^c_{\Z_p}(X)\cong\mathrm{Spin}^c(X)\times\Z_p$.
\end{lem}

\begin{cor}\label{lem: twisting is universal on disk bundles}
Two equivariant $\mathrm{Spin}^c$ structures on such a disk bundle with the same underlying structure differ by a unique character twist.
\end{cor}
\begin{proof}
Take the difference of their equivariance numbers and apply \Cref{lem: equivariance number}.
\end{proof}

\begin{lem}\label{lem: eqv number depends on inv point}
Let $U_x$ be an invariant contractible neighborhood of the specified fixed point $x$. If two equivariant structures agree on $U_x$, then their equivariance numbers at $x$ agree.
\end{lem}
\begin{proof}
Write $\widetilde{\mathfrak{s}}_2=\widetilde{\mathfrak{s}}_1\otimes E$ using \Cref{lem: eqv line bundles and spin c torsor}. Since the structures agree on $U_x$, the equivariant line bundle $E|_{U_x}$ is trivial. Its restriction to $x$ therefore has trivial character. In the local expressions above, tensoring by a character adds that character to $m$, proving the assertion.
\end{proof}

We next consider the relative lifting problem. Let $X=D(L)$ be the disk bundle of a line bundle $L\to S^2$ with Euler number $-n$, where $n\ne0$, and let $Y=\partial X$. Suppose the generator acts on $L$ by fiberwise multiplication by $\zeta_p$. We wish to prescribe both an underlying structure on $X$ and an equivariant lift of its boundary restriction.

Choose a degree-one line bundle $H\to S^2$ and a reference equivariant structure $\widetilde{\mathfrak{s}}_0$ with equivariance number zero. Give $\pi^*H$ the lift acting trivially on fibers, and let $\C_j$ be the trivial line bundle with character $\zeta_p^j$. We use the coordinates
\[
(i,[j])\longleftrightarrow
\widetilde{\mathfrak{s}}_0\otimes\pi^*H^{\otimes i}\otimes\C_j
\]
on $X$, and use $\widetilde{\mathfrak{s}}_0|_Y$ as the reference on $Y$. In particular, $i$ is the degree of the twisting line bundle, rather than the first Chern number of the $\mathrm{Spin}^c$ structure.

\begin{lem}\label{lem: eqv restriction map}
Using $\pi^*H|_Y$ as the generator on the boundary, the equivariant restriction map is
\[
\mathrm{Spin}^c_{\Z_p}(X)\cong\Z\oplus\Z_p
\longrightarrow\mathrm{Spin}^c_{\Z_p}(Y)\cong\Z_{np},
\qquad(i,[j])\longmapsto[i+nj].
\]
\end{lem}
\begin{proof}
The action on $Y$ is free, and $Y/\Z_p$ is the circle bundle of Euler number $-np$. The pullback of $H$ generates $H^2(Y/\Z_p;\Z)\cong\Z_{np}$. The quotient maps give the commutative square
\[
\xymatrix{
H^2(X/\Z_p;\Z)\ar[r]\ar[d]
&H^2(Y/\Z_p;\Z)\ar[d]^\cong\\
H^2_{\Z_p}(X;\Z)\ar[r]^r&H^2_{\Z_p}(Y;\Z),
}
\]
so the ordinary degree-one class restricts to the indicated generator. To determine the character coordinate, put $u=c_1(H)$ and $v=c_1^{\Z_p}(\C_1)$. Since the zero-section is fixed,
\[
H^2_{\Z_p}(X;\Z)=\Z u\oplus\Z_p v,
\qquad e_{\Z_p}(L)=-nu+v.
\]
The equivariant Thom isomorphism and the exact sequence of $(X,Y)$ give
\[
\Z\xrightarrow{\,1\mapsto-nu+v\,}
\Z u\oplus\Z_p v\xrightarrow{\ r\ }
H^2_{\Z_p}(Y;\Z)\longrightarrow0.
\]
Thus $r(u)$ has order $|np|$ and $r(v)=nr(u)$. By \Cref{lem: eqv line bundles and spin c torsor}, restriction of the structures is restriction of their difference classes, and
\[
r(iu+jv)=(i+nj)r(u).
\]
This proves the formula with the chosen generators.
\end{proof}

\begin{lem}\label{lem: unique relative eqv lifting}
For a disk bundle of nonzero Euler number with fiberwise $\Z_p$-rotation, an underlying $\mathrm{Spin}^c$ structure on $X$ and a compatible equivariant boundary structure determine a unique equivariant lift on $X$.
\end{lem}
\begin{proof}
The restriction and forgetful maps fit into the diagram
\[
\xymatrix@C=8em@R=4em{
\Z\oplus\Z_p
  \ar[r]^{(i,[j])\mapsto[i+nj]}
  \ar[d]_{(i,[j])\mapsto i}
&
\Z_{np}
  \ar[d]^{[k]\mapsto[k]}
\\
\Z
  \ar[r]_{i\mapsto[i]}
&
\Z_n.
}
\]
If the prescribed underlying structure has coordinate $i$ and the boundary lift has coordinate $[k]\in\Z_{np}$, compatibility says $k\equiv i\pmod n$. There is exactly one $[j]\in\Z_p$ with $[i+nj]=[k]$, as required.
\end{proof}

\begin{lem}\label{lem: general relative unique extension}
Let $X$ be a disk bundle of nonzero Euler number with an arbitrary linear $S^1$-action, restricted to $\Z_p$. An underlying structure on $X$ and a compatible equivariant boundary structure determine a unique equivariant lift on $X$.
\end{lem}
\begin{proof}
Both $X$ and its lens-space boundary have $H^1(-;\Z)=0$. For the fixed underlying structures, their sets of lifts are free transitive $\Z_p$-sets by \Cref{lem: order p equivariant lift torsor}. Restriction commutes with character twisting, so it is a bijection on these sets.
\end{proof}

The same argument applies to the plumbing. Recall that $W_\Gamma$ is the specified negative-definite star-shaped plumbing, with the circle action of \Cref{subsec: eqv plumbing action}, and put $Y=\partial W_\Gamma$.

\begin{lem}\label{lem: unique relative lifting for plumbing}
An underlying $\mathrm{Spin}^c$ structure on $W_\Gamma$ and a compatible equivariant boundary lift on $Y$ determine a unique equivariant lift on $W_\Gamma$.
\end{lem}
\begin{proof}
The circle action preserves every underlying structure, and $H^1(W_\Gamma;\Z)=H^1(Y;\Z)=0$. Apply \Cref{lem: order p equivariant lift torsor} to both manifolds and use the compatibility of restriction with twisting.
\end{proof}

Use a fixed point on the connected central sphere to define $n_{\mathrm{eqv}}$ on $W_\Gamma$. Character twisting changes this number by the corresponding element of $\Z_p$, exactly as for a disk bundle. We obtain the following extension of \Cref{lem: equivariance number}.

\begin{lem}\label{lem: n eqv and twisting on plumbing}
The map
\[
\mathrm{Spin}^c_{\Z_p}(W_\Gamma)\longrightarrow
\mathrm{Spin}^c(W_\Gamma)\times\Z_p,
\qquad\widetilde{\mathfrak{s}}\longmapsto
\bigl(\mathcal N(\widetilde{\mathfrak{s}}),n_{\mathrm{eqv}}(\widetilde{\mathfrak{s}})\bigr)
\]
is a bijection. Character twisting acts by translation on the second coordinate.
\end{lem}

\begin{cor}\label{cor: twisting on seifert QHS}
Two equivariant $\mathrm{Spin}^c$ structures on $Y$ with the same underlying structure differ by a unique character twist.
\end{cor}
\begin{proof}
Extend the underlying structure to $W_\Gamma$ and apply \Cref{lem: unique relative lifting for plumbing,lem: n eqv and twisting on plumbing}.
\end{proof}

\begin{lem}\label{lem: eqv SWF of different eqv lifts}
Assume $p\nmid|H_1(Y;\Z)|$, and let $\widetilde{\mathfrak{s}}'=\mathrm{tw}_k(\widetilde{\mathfrak{s}})$. The corresponding equivariant Floer spectra agree after reparametrizing the action by
\[
\varphi_{[k]}\colon S^1\times\Z_p\longrightarrow S^1\times\Z_p,
\qquad(x,[n])\longmapsto(\zeta_p^{kn}x,[n]).
\]
\end{lem}
\begin{proof}
The configuration spaces and Seiberg--Witten maps are unchanged, while the action on spinors is multiplied by the indicated character. Thus the finite-dimensional approximations and Conley indices are related by $\varphi_{[k]}$. The same is true of the normalization: each equivariant eta term and each fixed-point correction is multiplied by $\zeta_p^{kn}$ at the element $[n]$, so
\[
\mathbf n(Y,\widetilde{\mathfrak{s}}',g)
=[\C_{[k]}]\,\mathbf n(Y,\widetilde{\mathfrak{s}},g).
\]
This is precisely the change in the complex suspension under $\varphi_{[k]}$.
\end{proof}

\begin{lem}\label{lem: splittings are eqv spin c lifts}
Assume $p\nmid|H_1(Y;\Z)|$. For a fixed invariant underlying structure $\mathfrak{s}$, the map
\[
\mathcal S_{Y,\mathfrak{s}}\colon\mathrm{Split}(Y,\mathfrak{s})
\longrightarrow\mathrm{Spin}^c_{\Z_p}(Y,\mathfrak{s})
\]
is bijective, where a splitting is a right inverse of $G_{\mathfrak{s}}\to\Z_p$.
\end{lem}
\begin{proof}
After fixing one splitting, any other is its product with a character $\Z_p\to S^1$. This changes the corresponding lift by the same character twist. The assertion follows from \Cref{cor: twisting on seifert QHS}.
\end{proof}

For completeness, we record the distinction between even and odd spin lifts when $p=2$.
\begin{lem}\label{lem: eqv spin str equals self-conj eqv spin c}
Assume that $|H_1(Y;\Z)|$ is odd and that the involution admits an order-two lift to the spin bundle. Then the two spin lifts induce exactly the two self-conjugate equivariant $\mathrm{Spin}^c$ lifts of the unique spin structure. They differ by $1$-twisting.
\end{lem}
\begin{proof}
The spin lifts are $\widetilde\tau$ and $-\widetilde\tau$. Their induced $\mathrm{Spin}^c$ structures are self-conjugate and differ by the nontrivial character. By \Cref{lem: order p equivariant lift torsor}, they exhaust the equivariant lifts of the underlying structure.
\end{proof}

\begin{cor}\label{cor: eqv pin(2) SWF of different eqv spin}
Under the even-type hypothesis of \Cref{lem: eqv spin str equals self-conj eqv spin c}, the two $\mathrm{Pin}(2)\times\Z_2$ Floer spectra agree after reparametrizing by $(x,[n])\mapsto((-1)^nx,[n])$.
\end{cor}

\subsection{Equivariant $\mathrm{Spin}^c$ computation sequences and $\Z_p$-labelled planar graded roots} \label{subsec: eqv spin c comp seq}

Let $\Gamma$ be a star-shaped negative-definite almost rational plumbing graph such that $Y = \partial W_\Gamma$ is a rational homology sphere. 
As reviewed in \Cref{subsec: eqv plumbing action}, $W_\Gamma$ admits a smooth $S^1$-action that acts linearly on each disk bundle $D_v$ ($v \in V(\Gamma)$), and in particular, acts on the disk bundle $D_{v_c}$ corresponding to the central node $v_c$ via fiberwise rotation. 
The induced $S^1$-action on $Y$ is the Seifert action, and, as always, we will only consider its $\Z_p$-subaction. 
Throughout the paper, we will also assume that $p$ does not divide $|H_1(Y;\Z)| = |H^2(Y;\Z)|$.

Choose a $\Z_p$-equivariant $\mathrm{Spin}^c$ structure $\tilde{\s}$ on $Y$, and let $\s = \mathcal{N}(\tilde{\s})$ be its underlying $\mathrm{Spin}^c$ structure. 
As discussed in \Cref{rem: computation sequence of spin c str}, by taking $v_c$ as the base node we obtain the $\mathrm{Spin}^c$ computation sequence
\[
\mathrm{sp}_{\s}\bigl(x_{\s}(0)\bigr),\,
\mathrm{sp}_{\s}\bigl(x^{\s}_{0,0}\bigr),\dots,\,
\mathrm{sp}_{\s}\bigl(x^{\s}_{0,n_0-1}\bigr),\,
\mathrm{sp}_{\s}\bigl(x_{\s}(1)\bigr),\dots
\]
of $(\Gamma,\s)$. 
This sequence has the following properties:
\begin{itemize}
    \item $\mathrm{sp}_{\s}(x^{\s}_{i,0}) 
    = \mathrm{sp}_{\s}(x_{\s}(i)) + PD[S_{v_c}]$;
    \item $\mathrm{sp}_{\s}(x^{\s}_{i,j+1}) 
    = \mathrm{sp}_{\s}(x^{\s}_{i,j}) + PD[S_v] 
    \quad \text{for some } v\in V(\Gamma)\smallsetminus\{v_c\}$;
    \item Each $\mathrm{Spin}^c$ structure in the sequence restricts to $\s$ on $Y$.
\end{itemize}
By \Cref{lem: unique relative lifting for plumbing}, each $\mathrm{Spin}^c$ structure in this sequence admits a unique $\Z_p$-equivariant lift restricting to $\tilde{\s}$ on $Y$. 
We denote the resulting sequence of $\Z_p$-equivariant $\mathrm{Spin}^c$ structures on $W_\Gamma$ by
\[
\widetilde{\mathrm{sp}}_{\tilde{\s}}\bigl(x_{\s}(0)\bigr),\,
\widetilde{\mathrm{sp}}_{\tilde{\s}}\bigl(x^{\s}_{0,0}\bigr),\dots,\,
\widetilde{\mathrm{sp}}_{\tilde{\s}}\bigl(x^{\s}_{0,n_0-1}\bigr),\,
\widetilde{\mathrm{sp}}_{\tilde{\s}}\bigl(x_{\s}(1)\bigr),\dots
\]
Each $\Z_p$-equivariant $\mathrm{Spin}^c$ structure in this sequence now restricts to $\tilde{\s}$ on $Y$. 
Moreover, by \Cref{lem: unique relative eqv lifting,lem: general relative unique extension}, we have:
\begin{itemize}
    \item $\widetilde{\mathrm{sp}}_{\tilde{\s}}(x^{\s}_{i,0})$ and $\widetilde{\mathrm{sp}}_{\tilde{\s}}(x_{\s}(i))$ differ only in the interior of the central disk bundle $D_{v_c}$;
    \item $\widetilde{\mathrm{sp}}_{\tilde{\s}}(x^{\s}_{i,j+1})$ and $\widetilde{\mathrm{sp}}_{\tilde{\s}}(x^{\s}_{i,j})$ differ only in the disk bundle $D_v$ for some $v\in V(\Gamma)\smallsetminus\{v_c\}$.
\end{itemize}

\begin{rem}
    We provide a detailed explanation of why this construction works. 
    Rewrite the nonequivariant $\mathrm{Spin}^c$ computation sequence as 
    $\mathfrak{s}_1,\mathfrak{s}_2,\dots$. 
    By \Cref{lem: unique relative lifting for plumbing}, let 
    $\tilde{\s}_1,\tilde{\s}_2,\dots$ denote their unique $\Z_p$-equivariant lifts 
    that restrict to $\tilde{\s}$ on $Y$. 
    Suppose that $\mathfrak{s}_k$ and $\mathfrak{s}_{k+1}$ differ only on $D_v$ 
    for some node $v\in V(\Gamma)$. 
    Then, by \Cref{lem: unique relative eqv lifting,lem: general relative unique extension}, 
    there exists a $\Z_p$-equivariant $\mathrm{Spin}^c$ structure $\tilde{\s}'_{k+1}$ on $W_\Gamma$, 
    which agrees with $\tilde{\s}_k$ outside the interior of $D_v$ and satisfies 
    $\mathcal{N}(\tilde{\s}'_{k+1}) = \mathfrak{s}_{k+1}$. 
    Hence we have
    \[
    \tilde{\s}_{k+1}\vert_Y = \tilde{\s}'_{k+1}\vert_Y = \tilde{\s},
    \qquad 
    \mathcal{N}(\tilde{\s}_{k+1}) = \mathcal{N}(\tilde{\s}'_{k+1}) = \mathfrak{s}_{k+1}.
    \]
    By uniqueness (\Cref{lem: unique relative lifting for plumbing}), 
    it follows that $\tilde{\s}_{k+1} = \tilde{\s}'_{k+1}$. 
    Therefore, $\tilde{\s}_k$ and $\tilde{\s}_{k+1}$ differ only in the interior of $D_v$.
\end{rem}

\begin{lem} \label{lem: eqv num increase for central node}
    For each integer $i \geq 0$, we have
    \[
    n_{\mathrm{eqv}}\!\left(\widetilde{\mathrm{sp}}_{\tilde{\s}}\!\left(x^{\s}_{i,0}\right)\right) 
    = n_{\mathrm{eqv}}\!\left(\widetilde{\mathrm{sp}}_{\tilde{\s}}\!\left(x_{\s}(i)\right)\right) + 1.
    \]
\end{lem}

\begin{proof}
Put $w=w(v_c)<0$ and use the coordinates of \Cref{lem: eqv restriction map}, with $n=-w$. Then
\[
\mathrm{res}^\partial_{\Z_p}(a,[j])=[a-wj]\in\Z_{(-w)p}.
\]
Since $\mathrm{sp}_{\mathfrak{s}}(x^{\mathfrak{s}}_{i,0})
=\mathrm{sp}_{\mathfrak{s}}(x_{\mathfrak{s}}(i))+PD[S_{v_c}]$, the degree of the twisting line bundle increases by $[S_{v_c}]^2=w$. Hence its first coordinate changes from $a$ to $a+w$. The two equivariant structures agree on $\partial D_{v_c}$, and
\[
(a+w)-w(j+1)=a-wj.
\]
By uniqueness in \Cref{lem: unique relative eqv lifting}, their character coordinates differ by $[1]$, proving the assertion.
\end{proof}

\begin{lem}\label{lem: eqv num is fixed in comp seq}
For any integers $i,j$ with $i\ge0$ and $0\le j<n_i$, we have
\[
n_{\mathrm{eqv}}\!\left(\widetilde{\mathrm{sp}}_{\tilde{\s}}\!\left(x^{\s}_{i,j}\right)\right)
=
n_{\mathrm{eqv}}\!\left(\widetilde{\mathrm{sp}}_{\tilde{\s}}\!\left(x^{\s}_{i,j+1}\right)\right).
\]
\end{lem}

\begin{proof}
    Choose any point $p \in S_{v_c}$ that is not contained in $D_v$ for any 
    $v \in V(\Gamma)\smallsetminus \{v_c\}$. 
    Since $\Z_p$ acts on $D_{v_c}$ by fiberwise rotation, $p$ is a fixed point. 
    Hence we may take an open ball neighborhood $U_p \subset D_{v_c}$ satisfying:
    \begin{itemize}
        \item $U_p \cap D_v = \emptyset$ for all $v \in V(\Gamma)\smallsetminus \{v_c\}$;
        \item $U_p$ is setwise $\Z_p$-invariant.
    \end{itemize}
    Now, since there exists some $v \in V(\Gamma)\smallsetminus \{v_c\}$ such that 
    $\widetilde{\mathrm{sp}}_{\tilde{\s}}\!\left(x^{\s}_{i,j}\right)$ and 
    $\widetilde{\mathrm{sp}}_{\tilde{\s}}\!\left(x^{\s}_{i,j+1}\right)$ 
    differ only in the interior of $D_v$, the claim follows from 
    \Cref{lem: eqv number depends on inv point}.
\end{proof}

We now explain how to turn this data into an enhanced version of planar graded roots, which we call \emph{$\Z_p$-labelled planar graded roots}. 
Given a group $G$, we say that an element 
\[
x = \sum_{g \in G} x_g \cdot g \in \Z[G]
\]
\emph{has nonnegative coefficients} if $x_g \geq 0$ for all $g \in G$, and we denote 
\[
|x| = \sum_{g \in G} x_g.
\]

\begin{defn}
    A $\Z_p$-labelled planar graded root is a tuple 
    \[
    \mathcal{R} = (R, \lambda_V, \{\lambda_{A,w}\}_{w \in V}),
    \]
    with the weight function of $V(R)$ denoted by $\chi$, where the following conditions are satisfied.\footnote{We will sometimes drop $w$ from the notation if it is clear from context.}
    \begin{itemize}
        \item $R$ is a planar graded root with node set $V$;
        \item $\lambda_V$ is a function from the set of leaves of $\mathcal{R}$ to $\Z[\Z_p]$;
        \item For each $w \in V(R)$, $\lambda_{A,w}$ is a function from $\mathrm{Angle}(w)$ to $\Z[\Z_p]$;
        \item For each simple angle $(v,v') \in \mathrm{Angle}(w)$, the elements
        \[
        \lambda_{A,w}(v,v') \qquad \text{ and } \qquad \lambda_{A,w}(v,v')+\lambda_V(v)-\lambda_V(v')
        \]
        in $\Z[\Z_p]$ have nonnegative coefficients;
        \item For each simple angle $(v,v') \in \mathrm{Angle}(w)$, we have
        \[
        |\lambda_{A,w}(v,v')| = \chi(w)-\chi(v),\
        \qquad 
        |\lambda_V(v)-\lambda_V(v')| = \chi(v)-\chi(v').
        \]
    \end{itemize}
    We call $R$ the \emph{underlying planar graded root} of $\mathcal{R}$, 
    $\lambda_V$ the \emph{leaf labelling} of $\mathcal{R}$, and 
    $\lambda_{A,w}$ the \emph{angle labelling} of $\mathcal{R}$ at $w$. 
    Two $\Z_p$-labelled planar graded roots are said to be \emph{equivalent} 
    if their underlying planar graded roots are equivalent and the leaf and angle labels agree up to an overall cyclic shift of $\Z[\Z_p]$ (via multiplication by a fixed element of $\Z_p$) and an overall addition of leaf labels by a fixed element of $\Z[\Z_p]$.
\end{defn}

Recall from \Cref{defn: graded root from abstract sequence} that eventually increasing sequences of integers give rise to planar graded roots. 
In a completely analogous way, we can upgrade this procedure to construct $\Z_p$-labelled planar graded roots for any prime $p$. 
This construction is modelled on \Cref{lem: eqv num increase for central node} and \Cref{lem: eqv num is fixed in comp seq}.

\begin{defn} 
    Given an eventually increasing sequence $\mathbf{n}=(n_i)_{i\ge 0}$ of integers, let $R_\mathbf{n}$ denote the associated planar graded root. 
    For clarity, we write the element $[i]\in \Z_p$ as $\tau_p^i$ when describing labels. 
    Using the notations of \Cref{defn: graded root from abstract sequence}, we endow the leaves and simple angles of $R_\mathbf{n}$ with $\Z_p$-labels as follows:
    \begin{itemize}
        \item For each leaf index $i_k\in I$, set
        \[
        \lambda_V(i_k)=\sum_{s=0}^{i_k-1}(n_{s+1}-n_s)\tau_p^s.
        \]
        \item For each simple angle
        $\alpha_k=(i_{k-1},i_k)$, choose $j_k$ on the maximum plateau
        between these leaves as in
        \Cref{defn: graded root from abstract sequence}, and set
        \[
        \lambda_A(\alpha_k)
        =
        \sum_{s=i_{k-1}}^{j_k-1}
        (n_{s+1}-n_s)\tau_p^s.
        \]
    \end{itemize}
    These leaf and angle labellings make $R_\mathbf{n}$ into a $\Z_p$-labelled planar graded root, denoted $\mathcal{R}_\mathbf{n}$.
\end{defn}

Since the sequence
\[
\mathbf{x}(\Gamma,\tilde{\s}) 
= \bigl\{\,\chi_{\s}(x_{\s}(0)),\,
\chi_{\s}(x_{\s}(1)),\,
\chi_{\s}(x_{\s}(2)),\dots\bigr\}
\]
is eventually increasing, and
\[
n_{\mathrm{eqv}}\bigl(\widetilde{\mathrm{sp}}_{\tilde{\s}}(x_{\s}(i))\bigr) 
= n_{\mathrm{eqv}}\bigl(\widetilde{\mathrm{sp}}_{\tilde{\s}}(x_{\s}(0))\bigr) + i,
\]
it is natural to make the following definition.

\begin{defn}\label{def: labelled gr root for Y and s}
We define
\[
\mathcal R_{\Gamma,\tilde{\s}}
=
\mathcal R_{\mathbf{x}(\Gamma,\tilde{\s})}
\]
using the fixed normalized plumbing $\Gamma$, its linear circle action,
the central base node, and the ordered equivariant computation sequence.
Changing the boundary lift by a character twist changes all labels by
the corresponding common cyclic reparametrization.
\end{defn}

\subsection{Equivariant Dirac indices, adjunction relations, and $S^1 \times \Z_p$-lattice model} \label{subsec: Zp lattice homotopy} 
Consider the computation sequence
\[
\widetilde{\mathrm{sp}}_{\tilde{\s}}(x_{\s}(0)),\,
\widetilde{\mathrm{sp}}_{\tilde{\s}}(x^{\s}_{0,0}),\,
\dots,\,
\widetilde{\mathrm{sp}}_{\tilde{\s}}(x^{\s}_{0,n_0-1}),\,
\widetilde{\mathrm{sp}}_{\tilde{\s}}(x_{\s}(1)),\,
\dots
\]
constructed in the previous subsection. 
To convert this sequence into a lattice homotopy type, 
for each $\Z_p$-equivariant $\mathrm{Spin}^c$ structure $\tilde{\s}$ appearing in the sequence, 
we consider the $(S^1 \times \Z_p)$-equivariant Bauer–Furuta invariant
\[
BF_{S^1 \times \Z_p}(W_\Gamma,\tilde{\s}) 
\colon \bigl(\mathrm{ind}^t_{\Z_p}\dirac_{W_\Gamma,\tilde{\s}}\bigr)^+ \longrightarrow SWF(Y,\tilde{\s}),
\]
as defined in \Cref{prop:eqBF}, where $\dirac_{W_\Gamma,\tilde{\s}}$ denotes the equivariant Dirac operator on $W_\Gamma$ associated to $\tilde{\s}$, and $\mathrm{ind}^t_{\Z_p}$ denotes the topological part of its $(S^1\times\Z_p)$-equivariant index, which lies in $R(\Z_p)\otimes\Q$.
We then glue these invariants together via adjunction relations, 
following the arguments of \cite{dai2023lattice}.

We will use the following two lemmas.

\begin{lem}\label{lem: borsuk ulam}
    Let $\mathbf{m},\mathbf{n}\in \Z[\Z_p]$. 
    Suppose there exists a based $(S^1 \times \Z_p)$-equivariant (stable) map
    \[
    f \colon (\C^\mathbf{m})^+ \longrightarrow (\C^\mathbf{n})^+
    \]
    such that the induced map on $S^1$-fixed points,
    \[
    f^{S^1} \colon (\C^0)^+ \longrightarrow (\C^0)^+,
    \]
    is a homotopy equivalence. 
    Then $\mathbf{n}-\mathbf{m}$ has nonnegative coefficients.
\end{lem}

\begin{proof}
Write
$\mathbf m=\sum_{k=0}^{p-1}m_k[\C_{[k]}]$ and
$\mathbf n=\sum_{k=0}^{p-1}n_k[\C_{[k]}]$.
Since the map is stable, after a common suspension we may represent it
by an actual $S^1\times\Z_p$-equivariant map
\[
f\colon (A\oplus\C^{\mathbf m'})^+
\longrightarrow
(A\oplus\C^{\mathbf n'})^+,
\]
where $A$ is a finite-dimensional real $\Z_p$-representation and all
coefficients of $\mathbf m'$ and $\mathbf n'$ are nonnegative.  The
same representation has been added to the source and target, so that
$\mathbf n'-\mathbf m'=\mathbf n-\mathbf m$.

Fix $k$.  Taking $G_k$-fixed points gives a map whose complex parts are
$\C^{m'_k}$ and $\C^{n'_k}$, and the quotient
$(S^1\times\Z_p)/G_k\cong S^1$ acts on them with the standard weight.
The common real fixed summand is the same on both sides and may be
suppressed in the Thom isomorphism.

Apply rational Borel $S^1$-equivariant cohomology to the inclusion of
the $S^1$-fixed sphere.  Naturality of the Thom class gives
\[
U^{m'_k}(f^{G_k})^*(1)=r\,U^{n'_k}
\]
where $r$ is the degree of $(f^{S^1})^{\Z_p}$. This degree is nonzero. Indeed, $f^{S^1}\colon A^+\to A^+$ has ordinary degree $\pm1$. Apply mod-$p$ Borel localization to $A^{\Z_p}\subset A$. The Euler class of the complementary representation is nonzero and cancels on the two sides of the naturality square, giving
\[
\deg\bigl((f^{S^1})^{\Z_p}\bigr)
\equiv\deg(f^{S^1})=\pm1\pmod p.
\]
For odd $p$ the real representations are orientable, and for $p=2$ mod-two Thom classes suffice. Thus $r\in\Q^\times$, and the preceding rational Thom-class identity implies $m'_k\leq n'_k$.  As the common suspension adds
the same coefficient to $m_k$ and $n_k$, this is equivalent to
$m_k\leq n_k$.

Since this holds for every $k$, the virtual representation
$\mathbf n-\mathbf m$ has nonnegative coefficients.
\end{proof}

\begin{lem}\label{lem: adjunction}
Let $\mathbf{m},\mathbf{n}\in\Z[\Z_p]$, and let
\[
f,g\colon
(\C^{\mathbf{m}})^+
\longrightarrow
(\C^{\mathbf{n}})^+
\]
be stable $S^1\times\Z_p$-equivariant maps in the chosen universe.

Suppose that, after a common stabilization, there is a finite-dimensional
real $\Z_p$-representation $A$ such that the induced $S^1$-fixed maps $
f^{S^1},g^{S^1}\colon A^+\longrightarrow A^+$
are both $\Z_p$-equivariantly homotopic to
$
\id_{A^+}$.
Then
\[
\mathbf{n}-\mathbf{m}\geq0 \text{ and }
f\simeq g
\]
through stable $S^1\times\Z_p$-equivariant maps.

In particular, if $
U^{\mathbf{n}-\mathbf{m}}
\colon
(\C^{\mathbf{m}})^+
\longrightarrow
(\C^{\mathbf{n}})^+$
denotes the standard representation-sphere map obtained from $
S^0\longrightarrow
(\C^{\mathbf{n}-\mathbf{m}})^+$,
then $
f\simeq g\simeq U^{\mathbf{n}-\mathbf{m}}$.
\end{lem}

\begin{proof}
After a common suspension, we may represent $f$ and $g$ by actual
$S^1\times\Z_p$-equivariant maps
\[
f,g\colon (A\oplus\C^{\mathbf m})^+
\longrightarrow
(A\oplus\C^{\mathbf n})^+,
\]
where $A$ is a finite-dimensional real $\Z_p$-representation on which
$S^1$ acts trivially.  Enlarging the common complex suspension if
necessary, we may assume that all coefficients of $\mathbf m$ and
$\mathbf n$ are nonnegative.  After a further common real suspension,
we may also assume that the prescribed $\Z_p$-equivariant homotopy
between the $S^1$-fixed maps and $\id_{A^+}$ is represented unstably.

Write $\mathbf m=\sum m_k[\C_{[k]}]$ and
$\mathbf n=\sum n_k[\C_{[k]}]$.  Taking $G_k$-fixed points and using
the residual circle $(S^1\times\Z_p)/G_k$, the equivariant Thom
isomorphism applied to the square determined by the inclusion of the
$S^1$-fixed sphere gives
\[
U^{m_k}(f^{G_k})^*(1)=U^{n_k}
\]
in $\Q[U]$.  Hence $m_k\leq n_k$ for every $k$, and therefore
$\mathbf n-\mathbf m\geq0$ coefficientwise.

It remains to prove uniqueness.  Choose a finite
$S^1\times\Z_p$-CW decomposition of
$(A\oplus\C^{\mathbf m})^+$ relative to $A^+$, compatible with the
orbit-type stratification.  Every relative cell has stabilizer either
$1$ or some $G_k$.  If such a cell is of the form
$(S^1\times\Z_p/H)\times D^r$, then its orbit has dimension one and
\[
r+1\leq
\dim_{\R}(A\oplus\C^{\mathbf m})^H.
\]
Since $\mathbf n-\mathbf m\geq0$ coefficientwise, we also have
$\dim_{\R}(A\oplus\C^{\mathbf n})^H
\geq\dim_{\R}(A\oplus\C^{\mathbf m})^H$.

Starting from the prescribed homotopy on $A^+$, extend it inductively
over these relative cells, keeping $f$ and $g$ fixed at the two ends
of the homotopy.  The obstruction for a cell with stabilizer $H$ lies
in
\[
\pi_r\!\left(((A\oplus\C^{\mathbf n})^+)^H\right),
\]
which vanishes by the preceding dimension inequality.  Thus the
homotopy extends over all cells, and $f$ and $g$ are stably
$S^1\times\Z_p$-equivariantly homotopic.

Finally, the standard map $U^{\mathbf n-\mathbf m}$ has the same
source and target and restricts to $\id_{A^+}$ on $S^1$-fixed points.
Applying the same uniqueness argument gives
$f\simeq g\simeq U^{\mathbf n-\mathbf m}$.
\end{proof}

Using \Cref{lem: adjunction,lem: borsuk ulam}, we prove the following lemma.

\begin{lem} \label{lem: BF is homotopic when weight is same}
    Let $\tilde{\s}, \tilde{\s}'$ be $\Z_p$-equivariant $\mathrm{Spin}^c$ structures on $W_\Gamma$ that agree outside the interior of $D_v$ for some $v \in V(\Gamma)$. 
    Suppose that 
    \[
    \mathcal{N}(\tilde{\s}') = \mathcal{N}(\tilde{\s}) + PD[S_v],
    \]
    and that the nonequivariant topological indices of the $\mathrm{Spin}^c$ Dirac operators for 
    $(W_\Gamma,\mathcal{N}(\tilde{\s}))$ and $(W_\Gamma,\mathcal{N}(\tilde{\s}'))$ coincide, i.e., 
    \[
    \ind^t \dirac_{W_\Gamma,\mathcal{N}(\tilde{\s})}
    \simeq
    \ind^t \dirac_{W_\Gamma,\mathcal{N}(\tilde{\s}')}.
    \]
    Then the $\Z_p$-equivariant indices of 
    $\dirac_{W_\Gamma,\tilde{\s}}$ and $\dirac_{W_\Gamma,\tilde{\s}'}$ 
    are stably isomorphic as $S^1 \times \Z_p$-representations, and we have
    \[
    BF_{S^1 \times \Z_p}(W_\Gamma,\tilde{\s})
    \sim
    BF_{S^1 \times \Z_p}(W_\Gamma,\tilde{\s}').
    \]
\end{lem}
\begin{proof}
    Write $W_\Gamma=D_v\cup Q$, where the two structures agree on $Q$. By index additivity in \Cref{general gluing}, their index difference is the difference on $D_v$. Once the homotopy is proved on $D_v$, the same theorem pairs it with the common relative invariant of $Q$. It therefore suffices to work on the disk bundle. In this setting, since $\Gamma$ is negative-definite, the weight $w(v)$ of $v$ is negative, which implies that $W_\Gamma = D_v$ is a negative-definite cobordism from $\emptyset$ to $\partial D_v = -L(n,1)$, where $n=-w(v)$. We write
    \[
    \mathrm{ind}^t_{\Z_p}\dirac_{D_v,\tilde{\s}} = \mathbf{m} 
    \qquad \text{and} \qquad  
    \mathrm{ind}^t_{\Z_p}\dirac_{D_v,\tilde{\s}'} = \mathbf{m}'
    \]
    for some $\mathbf{m}, \mathbf{m}' \in \Q[\Z_p]$.
    
    By \Cref{lem: equivariant PSC metrics on lens spaces}, $L(n,1)$ admits a  $\Z_p$-invariant metric of positive scalar curvature. Thus, by \Cref{PSC-vanish}, 
    \[
    SWF_{S^1 \times \Z_p}(-L(n,1)) \simeq (\C^\mathbf{r})^+
    \]
    for some $\mathbf{r}\in \Q[\Z_p]$. Therefore the equivariant Bauer--Furuta invariant for $(D_v,\tilde{\s})$ has the form
    \[
    BF_{S^1 \times \Z_p}(D_v,\tilde{\s})\colon 
    (\mathrm{ind}^t_{\Z_p}\dirac_{D_v,\tilde{\s}})^+ \longrightarrow (\C^{\mathbf{r}})^+.
    \]
    Forgetting the $\Z_p$-action gives the $S^1$-equivariant Bauer--Furuta invariant
    \[
    BF_{S^1}(D_v,\tilde{\s})\colon 
    (\mathrm{ind}^t \dirac_{D_v,\mathcal{N}(\tilde{\s})})^+ 
    \longrightarrow (\C^{|\mathbf{r}|})^+.
    \]
    Since $D_v$ is negative-definite, $(BF_{S^1}(D_v,\mathcal{N}(\tilde{\s})))^{S^1}$ and $(BF_{S^1}(D_v,\mathcal{N}(\tilde{\s}')))^{S^1}$ are homotopy equivalences. Hence, by \Cref{lem: borsuk ulam}, $\mathbf{m} \le \mathbf{r}$ and $\mathbf{m}' \le \mathbf{r}$.

    Write $k=\langle c_1(\mathcal{N}(\tilde{\s})),[S_v]\rangle$. Since $\mathcal{N}(\tilde{\s}')=\mathcal{N}(\tilde{\s})+PD[S_v]$ and $[S_v]^2=-n$, we have $\langle c_1(\mathcal{N}(\tilde{\s}')),[S_v]\rangle=k-2n$. Thus
\[
\begin{aligned}
|\mathbf{m}|
&=\alpha_{\mathbb{C}}\bigl(\mathrm{ind}^t_{\Z_p}\dirac_{D_v,\tilde{\s}}\bigr)
=-\frac{k^2-n}{8n}, \\
|\mathbf{m}'|
&=\alpha_{\mathbb{C}}\bigl(\mathrm{ind}^t_{\Z_p}\dirac_{D_v,\tilde{\s}'}\bigr)
=-\frac{(k-2n)^2-n}{8n}.
\end{aligned}
\]
Hence $k^2=(k-2n)^2$, i.e., $4n(n-k)=0$. Since $n\neq0$, we deduce $k=n$, which implies
\[
|\mathbf{m}|=-\frac{n^2-n}{8n}=|\mathbf{m}'|.
\]

   On the other hand, by the $d$-invariant formula for lens spaces \cite[Proposition~4.8]{ozsvath2003absolutely}\footnote{Here we actually need to calculate the monopole Fr\o yshov invariant $\delta$ of lens spaces. However, as noted in \cite[Remark~1.1]{LRS18}, the isomorphisms between monopole and Heegaard Floer homologies preserve $\Q$-gradings. Thus, for any $\mathrm{Spin}^c$ rational homology 3-sphere $(Y,\s)$,
   \[
   \delta(Y, \s) = \tfrac{1}{2} d(Y, \s).
   \]} 
   we have
   \[
   |\mathbf{r}| = \frac{1}{2}d(-L(n,1),\mathcal{N}(\tilde{\s})\vert_{\partial D_v}) 
   = -\frac{n^2-n}{8n},
   \]
   so $|\mathbf{m}|=|\mathbf{m}'|=|\mathbf{r}|$. Since $\mathbf{m}\le \mathbf{r}$ and $\mathbf{m}'\le \mathbf{r}$, we conclude $\mathbf{m}=\mathbf{m}'=\mathbf{r}$. Thus
   \[
   \mathrm{ind}^t_{\Z_p} \dirac_{D_v,\tilde{\s}} \simeq 
   \mathrm{ind}^t_{\Z_p}\dirac_{D_v,\tilde{\s}'},
   \]
   proving the first part of the lemma.

   Finally, use the same invariant metric, boundary data, and representation-sphere identification of the common boundary Floer spectrum for both maps. Since $H^+(D_v;\R)=0$, \Cref{prop:eqBF} shows that their $S^1$-fixed restrictions are both $\Z_p$-equivariantly homotopic to the identity after a common real stabilization, using the same identifications coming from the real linear part. The second part therefore follows from \Cref{lem: adjunction}.
\end{proof}

Next, we consider the disk bundle $X:= D_{v_c}$ for the central node $v_c$ of $\Gamma$.
Let $Y$ be the boundary of $X$. In this case, as discussed in \Cref{subsec: eqv plumbing action}, the $\Z_p$-action on $D_{v_c}$ is given by the fiberwise rotation. Recall from \Cref{lem: equivariance number} that we have an identification
\[
\mathrm{Spin}^c_{\Z_p}(X)\;\xrightarrow{\;\;\cong\;\;}\;\mathrm{Spin}^c(X)\times \Z_p; \qquad \tilde{\s}\longmapsto\left(\,\mathcal{N}(\tilde{\s}),\,n_{\mathrm{eqv}}(\tilde{\s})\,\right).
\] Thus, for any $n \in w(v_c) + 2\Z$ and $\alpha\in \Z_p$, we will denote by $\tilde{\s}_{n,\alpha}$ the unique $\Z_p$-equivariant $\mathrm{Spin}^c$ structure on $D_{v_c}$ such that $c_1(\mathcal{N}(\tilde{\s}_{n,\alpha})) = n$ (where we identify $H^2(D_{v_c};\Z)\cong \Z$) and $n_{\mathrm{eqv}}(\tilde{\s}_{n,\alpha}) = \alpha$. 

From now on, put $n=-w(v_c)>0$.  Consider the generator $\gamma = [1] \in \Z_p$. We compute the Lefschetz number difference
\[
\begin{split}
\mathcal{I}_{k,s} 
&= \mathrm{Tr}_\gamma \!\left(\mathrm{ind}^{\mathrm{APS}}_{\Z_p}\dirac_{\nu(S_{v_c}),\tilde{\mathfrak{s}}_{-n+2k,[s]}}\right) 
 - \mathrm{Tr}_\gamma \!\left(\mathrm{ind}^{\mathrm{APS}}_{\Z_p}\dirac_{\nu(S_{v_c}),\tilde{\mathfrak{s}}_{n+2k,[s-1]}}\right) \\
&= \mathrm{Tr}_\gamma \!\left(\mathrm{ind}^t_{\Z_p}\dirac_{\nu(S_{v_c}),\tilde{\mathfrak{s}}_{-n+2k,[s]}}\right) 
 - \mathrm{Tr}_\gamma \!\left(\mathrm{ind}^t_{\Z_p}\dirac_{\nu(S_{v_c}),\tilde{\mathfrak{s}}_{n+2k,[s-1]}}\right),
\end{split}
\]
for any $k \in \Z$ and $s \in \Z_p$. 
Here the second equality holds because, by \Cref{lem: eqv restriction map}, the $\Z_p$-equivariant $\mathrm{Spin}^c$ structures $\tilde{\mathfrak{s}}_{-n+2k,[s]}$ and $\tilde{\mathfrak{s}}_{n+2k,[s-1]}$ on $D_{v_c} = \nu(S_{v_c})$ restrict to the same $\Z_p$-equivariant $\mathrm{Spin}^c$ structure on $\partial D_{v_c}$. 
Note also that the $\Z_p$-fixed point set of $D_{v_c}$ does not meet $\partial D_{v_c}$. 
Therefore, by \Cref{even_equivariant_index}, the $\gamma$-trace of the $\Z_{p}$-equivariant index can be computed as
\[
\begin{split}
    \zeta_{2p}^{-1}\mathrm{ind}_\gamma^t(\dirac_{\nu(S_{v_c}),\tilde{\mathfrak{s}}_{-n+2k,[s]}}) 
    &= -\tfrac{1}{4}\,\zeta_p^s\!\left(i(-n+2k)\csc\tfrac{\pi}{p} + n\csc\tfrac{\pi}{p}\cot\tfrac{\pi}{p}\right) \\
    &= -\tfrac{1}{4}\,\zeta_p^s \csc^2\tfrac{\pi}{p}\!\left(n\cos\tfrac{\pi}{p} + i(-n+2k)\sin\tfrac{\pi}{p}\right), \\
    \zeta_{2p}^{-1}\mathrm{ind}_\gamma^t(\dirac_{\nu(S_{v_c}),\tilde{\mathfrak{s}}_{n+2k,[s-1]}}) 
    &= -\tfrac{1}{4}\,\zeta_p^{\,s-1}\csc^2\tfrac{\pi}{p}\!\left(n\cos\tfrac{\pi}{p} + i(n+2k)\sin\tfrac{\pi}{p}\right).
\end{split}
\]

\noindent Hence we obtain
\[
\begin{split}
    \zeta_{2p}^{-1}\mathcal{I}_{k,s}
&= \zeta_{2p}^{-1}\Bigl(
\mathrm{ind}_\gamma^{\mathrm{APS}}\!\left(\dirac_{\nu(S_{v_c}),\tilde{\mathfrak{s}}_{-n+2k,[s]}}\right)
-
\mathrm{ind}_\gamma^{\mathrm{APS}}\!\left(\dirac_{\nu(S_{v_c}),\tilde{\mathfrak{s}}_{n+2k,[s-1]}}\right)
\Bigr) \\
    &= \tfrac{\zeta_p^{s-1}\csc^2\!\tfrac{\pi}{p}}{4}
       \left( -\zeta_p\!\left(n\cos\tfrac{\pi}{p} + i(-n+2k)\sin\tfrac{\pi}{p}\right) 
              + \left(n\cos\tfrac{\pi}{p} + i(n+2k)\sin\tfrac{\pi}{p}\right)\right) \\
    &= \tfrac{\zeta_p^{s-1}\csc^2\!\tfrac{\pi}{p}}{4}
       \left(-\zeta_p\!\left(n\zeta_{2p}^{-1} + 2ik\sin\tfrac{\pi}{p}\right) 
              + \left(n\zeta_{2p} + 2ik\sin\tfrac{\pi}{p}\right)\right) \\
    &= \tfrac{ik\,\zeta_p^{s-1}\csc\tfrac{\pi}{p}}{2}(1-\zeta_p).
\end{split}
\]
Since
\[
\zeta_p - 1 
   = \bigl(\cos\tfrac{2\pi}{p}-1\bigr) + i\sin\tfrac{2\pi}{p} 
   = 2\sin\tfrac{\pi}{p}\bigl(-\sin\tfrac{\pi}{p} + i\cos\tfrac{\pi}{p}\bigr) 
   = 2i\zeta_{2p}\sin\tfrac{\pi}{p},
\]
we deduce
\[
\begin{aligned}
\mathcal{I}_{k,s}
&= \zeta_{2p}\cdot
\tfrac{ik\,\zeta_p^{s-1}\csc\tfrac{\pi}{p}}{2}(1-\zeta_p)= \zeta_{2p}\cdot
\tfrac{ik\,\zeta_p^{s-1}\csc\tfrac{\pi}{p}}{2}
\cdot\bigl(-2i\zeta_{2p}\sin\tfrac{\pi}{p}\bigr)= k\zeta_p^{s-1}\zeta_{2p}^2
= k\zeta_p^s.
\end{aligned}
\]
Using this computation, we obtain the following lemma.

\begin{lem} \label{lem: eqv index change for c1 increment}
    For any $k \in \Z$ and $s \in \Z_p$, we have
    \[
        \C_{[s]}^{-k} \oplus \mathrm{ind}^t_{\Z_p}\!\left(\dirac_{\nu(S_{v_c}),\tilde{\s}_{-n+2k,[s]}}\right) 
        \simeq \mathrm{ind}^t_{\Z_p}\!\left(\dirac_{\nu(S_{v_c}),\tilde{\s}_{n+2k,[s-1]}}\right).
    \]
\end{lem}
\begin{proof}
    We begin with the difference of complex dimensions:
    \[
   \alpha_\C \!\left(
\mathrm{ind}^t
\dirac_{\nu(S_{v_c}),\mathcal{N}(\tilde{\s}_{-n+2k,[s]})}
\right)
-
\alpha_\C \!\left(
\mathrm{ind}^t
\dirac_{\nu(S_{v_c}),\mathcal{N}(\tilde{\s}_{n+2k,[s-1]})}
\right)
=
\frac{-(-n+2k)^2+n}{8n}
-
\frac{-(n+2k)^2+n}{8n}
=
k.
    \]
    Thus there exists some $\mathbf{n}\in\Q[\Z_p]$ with $|\mathbf{n}|=-k$ such that
\[
\C^{\mathbf{n}}\oplus
\mathrm{ind}^t_{\Z_p}\!\left(\dirac_{\nu(S_{v_c}),\tilde{\s}_{-n+2k,[s]}}\right)
\simeq
\mathrm{ind}^t_{\Z_p}\!\left(\dirac_{\nu(S_{v_c}),\tilde{\s}_{n+2k,[s-1]}}\right).
\]
Writing $\mathbf{n}=\sum_{t=0}^{p-1}n_t\cdot[t]$, the trace relation from the computation of $\mathcal{I}_{k,s}$ gives
\[
\sum_{t=0}^{p-1}n_t\zeta_p^t=-\mathcal{I}_{k,s}=-k\zeta_p^s.
\]

    \noindent{\bf \underline{Case $p=2$}:}
    The relations reduce to
    \[
        n_{[0]} + n_{[1]} = -k, 
        \qquad 
        n_{[0]} - n_{[1]} = -(-1)^s k.
    \]
    Solving, we find
    \[
        \mathbf{n} 
        = -k \left( \tfrac{1+(-1)^s}{2}\cdot [0] + \tfrac{1-(-1)^s}{2}\cdot [1] \right) 
        = -k \cdot [s],
    \]
    as desired.

    \noindent{\bf \underline{Case $p>2$}:}
    The relation implies that $\zeta_p$ is a root of the polynomial
    \[
        kx^s + (n_0 + n_1x + \cdots + n_{p-1}x^{p-1}) \in \Q[x].
    \]
   Since the minimal polynomial of $\zeta_p$ over $\Q$ is
$1+x+\cdots+x^{p-1}$, there exists some $m\in \Q$ such that
    \[
        m(1+x+\cdots+x^{p-1}) = kx^s + (n_0 + n_1x + \cdots + n_{p-1}x^{p-1}).
    \]
    Evaluating at $x=1$, we obtain
    \[
        pm = k + (n_0+n_1+\cdots+n_{p-1}) = k + |\mathbf{n}| = 0,
    \]
    so $m=0$. Thus 
    \[
        n_0 + n_1x + \cdots + n_{p-1}x^{p-1} = -kx^s,
    \]
    which means $\mathbf{n} = -k\cdot [s]$, completing the proof.
\end{proof}

\begin{lem} \label{lem: weight change at central node}
Let $\tilde{\s}$ be a $\Z_p$-equivariant $\mathrm{Spin}^c$ structure on $W_\Gamma$, and let 
$\tilde{\s}'$ be another such structure which agrees with $\tilde{\s}$ outside the interior of $D_{v_c}$ and satisfies
\[
\mathcal{N}(\tilde{\s}') = \mathcal{N}(\tilde{\s}) + PD[S_{v_c}].
\]
Consider the index difference
\[
\Delta_{\tilde{\s}} = 
\alpha_{\C}\!\left( \mathrm{ind}^t_{\Z_p}\dirac_{W_\Gamma, \mathcal{N}(\tilde{\s}')} \right)
- \alpha_{\C}\!\left( \mathrm{ind}^t_{\Z_p}\dirac_{W_\Gamma, \mathcal{N}(\tilde{\s})} \right).
\]
Then there is a stable equivalence
\[
\mathrm{ind}^t_{\Z_p}\dirac_{W_\Gamma, \tilde{\s}'} \simeq
\mathrm{ind}^t_{\Z_p}\dirac_{W_\Gamma, \tilde{\s}} 
\oplus \C^{\Delta_{\tilde{\s}}}_{\,n_{\mathrm{eqv}}(\tilde{\s}')}.
\]
Furthermore, if we denote by $U_\alpha$ the stable $(S^1 \times \Z_p)$-equivariant homotopy class of the inclusion
\[
(\C^0)^+ \hookrightarrow (\C^{1\cdot \alpha})^+
\]
for each $\alpha \in \Z_p$, then we have $(S^1 \times \Z_p)$-equivariant homotopies
\[
\begin{aligned}
BF_{S^1 \times \Z_p}(W_\Gamma, \tilde{\s}')
\circ
U^{\Delta_{\tilde{\s}}}_{\,n_{\mathrm{eqv}}(\tilde{\s}')}
&\sim
BF_{S^1 \times \Z_p}(W_\Gamma, \tilde{\s})
&& \text{if } \Delta_{\tilde{\s}} > 0, \\
BF_{S^1 \times \Z_p}(W_\Gamma, \tilde{\s}')
&\sim
BF_{S^1 \times \Z_p}(W_\Gamma, \tilde{\s})
&& \text{if } \Delta_{\tilde{\s}} = 0, \\
BF_{S^1 \times \Z_p}(W_\Gamma, \tilde{\s}')
&\sim
BF_{S^1 \times \Z_p}(W_\Gamma, \tilde{\s})
\circ
U^{-\Delta_{\tilde{\s}}}_{\,n_{\mathrm{eqv}}(\tilde{\s}')}
&& \text{if } \Delta_{\tilde{\s}} < 0.
\end{aligned}
\]
\end{lem}

\begin{proof}
    Decompose $W_\Gamma=D_{v_c}\cup Q$, with the common structure on $Q$. By \Cref{general gluing}, the index difference is computed on $D_{v_c}$, where \Cref{lem: eqv index change for c1 increment} gives the first assertion. The same gluing theorem carries a homotopy on $D_{v_c}$ to one on $W_\Gamma$, so it remains to prove the local assertion. 

    For the second part, we may assume without loss of generality that $\Delta_{\tilde{\s}} \ge 0$. 
    Define
    \[
    f =
BF_{S^1 \times \Z_p}(D_{v_c}, \tilde{\s}')
\circ
U^{\Delta_{\tilde{\s}}}_{\,n_{\mathrm{eqv}}(\tilde{\s}')}
    \qquad  \text{ and } \qquad 
    g = BF_{S^1 \times \Z_p}(D_{v_c}, \tilde{\s}).
    \]
    Then, by the first part of the lemma, 
    we may regard $f$ and $g$ as $(S^1 \times \Z_p)$-equivariant stable maps of the form
    \[
    f,g\colon\bigl(\mathrm{ind}^t_{\Z_p}\dirac_{D_{v_c},\tilde{\s}}\bigr)^+
    \longrightarrow\left(\C^\mathbf r\right)^+
    \]
    for some $\mathbf{r} \in \Q[\Z_p]$. 
    Since $H^+(D_{v_c};\R)=0$, \Cref{prop:eqBF} shows that the $S^1$-fixed restrictions of the two Bauer--Furuta maps are both $\Z_p$-equivariantly homotopic to the identity after a common real stabilization, using the same boundary data and real-part identifications. Moreover, $U^{\Delta_{\tilde{\s}}}_{n_{\mathrm{eqv}}(\tilde{\s}')}$ restricts to the identity on $S^1$-fixed points. Thus $f^{S^1}$ and $g^{S^1}$ satisfy the hypothesis of \Cref{lem: adjunction}, which gives the required $S^1\times\Z_p$-equivariant homotopy.
\end{proof}

Recall that, in \Cref{subsec: eqv spin c comp seq}, given a $\Z_p$-equivariant $\mathrm{Spin}^c$ structure $\tilde{\s}$ on $Y=\partial W_\Gamma$, we have constructed the $\Z_p$-equivariant $\mathrm{Spin}^c$ computation sequence
\[
\widetilde{\mathrm{sp}}_{\tilde{\s}}(x_{\s}(0)),
\widetilde{\mathrm{sp}}_{\tilde{\s}}(x^{\s}_{0,0}),\dots,
\widetilde{\mathrm{sp}}_{\tilde{\s}}(x^{\s}_{0,n_0-1}),
\ \widetilde{\mathrm{sp}}_{\tilde{\s}}(x_{\s}(1)),\dots
\]
For simplicity, we assume that 
\[
n_{\mathrm{eqv}}\left(
\widetilde{\mathrm{sp}}_{\tilde{\s}}(x_{\s}(0))
\right)=[-1]\in\Z_p;
\]
otherwise, we simply cyclically permute the elements of $\Z_p$ during our construction. We will also rewrite the above sequence as follows:
\[
\mathfrak{s}_1,\mathfrak{s}_2,\mathfrak{s}_3,\dots
\]
For each integer $i\ge 0$, define
\[
\Delta_i =
\alpha_\C\!\left(
\mathrm{ind}^t_{\Z_p}\dirac_{W_\Gamma,\mathfrak{s}_{i+2}}
-
\mathrm{ind}^t_{\Z_p}\dirac_{W_\Gamma,\mathfrak{s}_{i+1}}
\right).
\]
By the formula
$\mathrm{ind}^t\dirac_{W_\Gamma,\mathrm{sp}_{\mathfrak{s}}(x)}
=\mathrm{const}-\chi_{\mathfrak{s}}(x)$, the eventual increase of the
weight sequence implies that these index increments are eventually
nonpositive.  Hence we may choose $N>0$ such that
$\Delta_i\le 0$ for every $i\ge N$.

Then, by \Cref{prop: BF gluing,lem: eqv num increase for central node,lem: eqv num is fixed in comp seq,lem: BF is homotopic when weight is same,lem: weight change at central node}, we see that the sequence $(\mathfrak{s}_i)_{i\ge 1}$ has the following properties.
\begin{itemize}
    \item For each $i\ge 0$, there exists a node
    $v_i\in V(\Gamma)$ such that
    \[
    \mathcal{N}(\mathfrak{s}_{i+2})
    =
    \mathcal{N}(\mathfrak{s}_{i+1})
    +
    PD[S_{v_i}].
    \]

    \item If $\Delta_i=0$, then there exists a stable
    $(S^1\times\Z_p)$-equivariant homotopy
    \[
    BF_{S^1\times\Z_p}(W_\Gamma,\mathfrak{s}_{i+2})
    \sim
    BF_{S^1\times\Z_p}(W_\Gamma,\mathfrak{s}_{i+1}).
    \]

    \item If $\Delta_i>0$, then the $i$th step is a central step
    starting at
    \[
    \mathfrak{s}_{i+1}
    =
    \widetilde{\mathrm{sp}}_{\tilde{\s}}(x_{\s}(k_i))
    \]
    for some $k_i\ge0$, and
    \[
    BF_{S^1\times\Z_p}(W_\Gamma,\mathfrak{s}_{i+2})
    \circ U^{\Delta_i}_{[k_i]}
    \sim
    BF_{S^1\times\Z_p}(W_\Gamma,\mathfrak{s}_{i+1}).
    \]

    \item If $\Delta_i<0$, then the $i$th step is a central step
    starting at
    \[
    \mathfrak{s}_{i+1}
    =
    \widetilde{\mathrm{sp}}_{\tilde{\s}}(x_{\s}(k_i))
    \]
    for some $k_i\ge0$, and
    \[
    BF_{S^1\times\Z_p}(W_\Gamma,\mathfrak{s}_{i+2})
    \sim
    BF_{S^1\times\Z_p}(W_\Gamma,\mathfrak{s}_{i+1})
    \circ U^{-\Delta_i}_{[k_i]}.
    \]
\end{itemize}

Using these properties, we can now describe an $S^1 \times \Z_p$-equivariant lattice homotopy type model for $Y$ as follows. For each $i$, we define a sequence of virtual $S^1 \times \Z_p$-representations $V_i$ and $S^1 \times \Z_p$-equivariant virtual linear maps
\[
f_i\colon \begin{cases}
    V_i \longrightarrow V_{i+1} \quad&\text{if } \Delta_i \ge 0,\\
    V_{i+1} \longrightarrow V_i \quad &\text{if } \Delta_i < 0,
\end{cases}
\]
as follows. We start by defining $V_0 = 0$, the zero representation. Suppose that we have defined $V_0,\dots,V_i$. Then we define $V_{i+1}$ in the following manner:
\[
\begin{aligned}
V_{i+1} &= V_i \oplus \C^{\Delta_i}_{[k_i]}, 
&\qquad f_i &= U^{\Delta_i}_{[k_i]}, 
&\qquad \text{if } \Delta_i > 0, \\
V_{i+1} &= V_i, 
&\qquad f_i &= \mathrm{id}, 
&\qquad \text{if } \Delta_i = 0, \\
V_i &= V_{i+1} \oplus \C^{-\Delta_i}_{[k_i]}, 
&\qquad f_i &= U^{-\Delta_i}_{[k_i]}, 
&\qquad \text{if } \Delta_i < 0.
\end{aligned}
\]

By the choice of $N$, we have $\Delta_i\le 0$ for every $i\ge N$.
For each such step, the additional part of the homotopy-colimit diagram
has the form
\[
V_i^+ \longleftarrow V_{i+1}^+
\xrightarrow{\mathrm{id}} V_{i+1}^+,
\]
where the left-hand map is the corresponding representation-sphere
inclusion (and is the identity when $\Delta_i=0$).
Its mapping cylinder equivariantly deformation retracts onto $V_i^+$,
relative to the previously constructed part.  Consequently, enlarging
the truncation beyond $N$ does not change the equivariant stable
homotopy type.  For each finite truncation, all virtual
representations involved may be made actual by a common suspension.

This defines $V_0,\dots,V_N$ and $f_0,\dots,f_{N-1}$, which induce $(S^1 \times \Z_p)$-spectra $(V_i)^+$ and $(S^1 \times \Z_p)$-equivariant stable maps $f_i^+$ between them. For simplicity, denote the domain of $f_i^+$ by $W_i^+$. Using these data, we build an $S^1 \times \Z_p$--CW complex $\mathcal{H}_{S^1 \times \Z_p}(\Gamma,\tilde{\s})$ as follows:
\[
\mathcal{H}_{S^1 \times \Z_p}(\Gamma,\tilde{\s}) 
= 
\left( \bigvee_{i=0}^N V_i^+ \right) 
\;\vee\; 
\left( \bigvee_{i=0}^{N-1} (W_i^+ \wedge [0,1]) \right) \big/ \sim ,
\]
where we take quotients by identifying the ends of cylinders $W_i^+ \wedge [0,1]$ with spheres $V_i^+$ and $V_{i+1}^+$ as follows.
\begin{itemize}
    \item If $W_i^+ = V_i^+$, identify $W_i^+ \times \{0\}$ with $V_i^+$ via the identity map, and attach $W_i^+ \times \{1\}$ to $V_{i+1}^+$ via $f_i^+$.
    \item If $W_i^+ = V_{i+1}^+$, identify $W_i^+ \times \{0\}$ with $V_{i+1}^+$ via the identity map, and attach $W_i^+ \times \{1\}$ to $V_i^+$ via $f_i^+$.
\end{itemize}

\begin{defn}\label{defn: S1xZp lattice spectrum}
We denote the resulting $(S^1 \times \Z_p)$-spectrum by $\mathcal{H}_{S^1 \times \Z_p}(\Gamma,\tilde{\s})$.
\end{defn}

\begin{rem}\label{rem: lattice homotopy type as hocolim}
The construction of $\mathcal{H}_{S^1 \times \Z_p}(\Gamma,\tilde{\s})$ can also be described categorically as follows:
\[
\mathcal{H}_{S^1 \times \Z_p}(\Gamma,\tilde{\s}) 
= \operatorname{hocolim} \left[
\vcenter{
\xymatrix@C=4em@R=2.5em{
V_0^+ & V_1^+ & V_2^+ & \cdots & V_N^+ \\
& W_0^+ \ar[ul] \ar[u] 
& W_1^+ \ar[ul] \ar[u] 
& \cdots \ar@{}[ul]^(.25){}="a"^(.9){}="b" \ar "a";"b"
& W_{N-1}^+ \ar@{}[ul]^(.15){}="a"^(.75){}="b" \ar "a";"b" \ar[u]
}
}
\right]
\]
where $N$ is a sufficiently large integer.
\end{rem}
Then we have the following theorem.

\begin{thm}\label{thm: eqv lattice comparison map}
There exists an $(S^1 \times \Z_p)$-equivariant stable map
\[
\mathcal{T}_{S^1 \times \Z_p} \colon
(\mathrm{ind}_{\Z_p}^t \dirac_{W_\Gamma, \mathfrak{s}_1})^+
\wedge
\mathcal{H}_{S^1 \times \Z_p}(\Gamma, \tilde{\s})
\longrightarrow
SWF_{S^1 \times \Z_p}(\partial W_\Gamma, \tilde{\s})
\]
that is an $S^1$-equivariant homotopy equivalence. Here, $\mathfrak{s}_1$ is the first term of the sequence
\[
\mathfrak{s}_1, \mathfrak{s}_2, \dots
\]
appearing in the discussion above.
\end{thm}

\begin{proof}
From the construction above, for every $i\ge0$ we have
\[
\bigl(\mathrm{ind}_{\Z_p}^t
\dirac_{W_\Gamma,\mathfrak{s}_1}\bigr)^+
\wedge V_i^+
\simeq
\bigl(\mathrm{ind}_{\Z_p}^t
\dirac_{W_\Gamma,\mathfrak{s}_{i+1}}\bigr)^+.
\]
Hence we define $\mathcal{T}_{S^1\times\Z_p}$ as follows.
\begin{itemize}
    \item For each $i$, set
    \[
    \mathcal{T}_{S^1\times\Z_p}\vert_{V_i}
    =
    BF_{S^1\times\Z_p}(W_\Gamma,\mathfrak{s}_{i+1}).
    \]

    \item For each $i$, define
    $\mathcal{T}_{S^1\times\Z_p}\vert_{W_i\wedge[0,1]}$
    using the equivariant homotopy supplied above between the two
    endpoint maps associated with
    \[
    BF_{S^1\times\Z_p}(W_\Gamma,\mathfrak{s}_{i+1})
    \qquad\text{and}\qquad
    BF_{S^1\times\Z_p}(W_\Gamma,\mathfrak{s}_{i+2}),
    \]
    after precomposing the map on the larger representation sphere
    with the corresponding $f_i$ when necessary.
\end{itemize}
The endpoint conditions make these maps agree on the attaching spheres, so they define an $S^1\times\Z_p$-equivariant map. There are only vertex spheres and edge cylinders in this construction. 

After forgetting the $\Z_p$-action, the restriction to each vertex sphere is the ordinary relative Bauer--Furuta map. The vertex-sphere classes generate the co-Borel homology of the lattice model, and their relations are the edge relations. Consequently the induced map on co-Borel homology is the lattice map of \cite[Lemma~5.10]{dai2023lattice}, independently of the choices of edge homotopies. That map is an isomorphism. On $S^1$-fixed points, the interval-shaped diagram retracts to any vertex sphere, where the relative invariant is a homotopy equivalence. After a common suspension making both spaces simply connected, \cite[Lemma~5.11]{dai2023lattice} therefore applies and proves that $\mathcal T_{S^1\times\Z_p}$ is an $S^1$-equivariant homotopy equivalence.
\end{proof}

Furthermore, it can be seen that the $(S^1 \times \Z_p)$-equivariant stable homotopy equivalence class of $\mathcal{H}_{S^1 \times \Z_p}(\Gamma,\tilde{\s})$ can be read off from the $\Z_p$-labelled planar graded root $\mathcal{R}_{\Gamma,\tilde{\s}}=(R,\lambda_V,\lambda_A)$ defined in \Cref{def: labelled gr root for Y and s}; the process is given as follows.
\begin{itemize}
    \item For each leaf $v$ of $R$, define the $(S^1 \times \Z_p)$-representation 
    \[
    V_v = \C^{-\lambda_V(v)}.
    \]
    \item For each simple angle $(v,v')$ of $R$, define the $(S^1 \times \Z_p)$-representation 
    \[
    W_{(v,v')} = \C^{-\lambda_V(v)-\lambda_A(v,v')}.
    \]
    Note that $W_{(v,v')}$ is naturally a subrepresentation of both $V_v$ and $V_{v'}$.
    \item Define 
    \[
    \mathcal{H}(\mathcal{R}_{\Gamma,\tilde{\s}}) 
    = 
    \left( \bigvee_{\text{leaf } v} V_v^+ \right) 
    \;\vee\; 
    \left( \bigvee_{\text{simple angle } (v,v')} (W_{(v,v')}^+ \wedge [0,1]) \right) \Big/ \sim,
    \]
    where $\sim$ is defined by attaching $W_{(v,v')}^+ \times \{0\}$ to $V_v^+$ and $W_{(v,v')}^+ \times \{1\}$ to $V_{v'}^+$ via the natural inclusions 
    \[
    W_{(v,v')} \hookrightarrow V_v 
    \qquad  \text{ and }\qquad 
    W_{(v,v')} \hookrightarrow V_{v'}.
    \]
\end{itemize}

\begin{lem}\label{lem: label gr root gives eqv lattice sp}
Possibly after a cyclic permutation of elements of $\Z_p$ applied to all leaf and angle labels of $\mathcal{R}_{\Gamma,\tilde{\s}}$, there exists a rational virtual complex representation class $V\in R(\Z_p)\otimes\Q$ such that there is an $S^1\times\Z_p$-equivariant stable map
\[
V^+\wedge\mathcal{H}(\mathcal{R}_{\Gamma,\tilde{\s}})
\longrightarrow
SWF_{S^1\times\Z_p}(Y,\tilde{\s})
\]
which is an $S^1$-equivariant stable homotopy equivalence. Here $V^+$ is interpreted as a formal suspension.
\end{lem}

\begin{proof}
Let $v_0$ be the leftmost leaf of $\mathcal{R}_{\Gamma,\tilde{\s}}$, i.e., there is no leaf $v'$ of $\mathcal{R}_{\Gamma,\tilde{\s}}$ such that $(v',v_0)$ forms a simple angle. Then it is straightforward to see that there exists an “inclusion”
\[
f \colon \left(\C^{\lambda_V(v_0)}\right)^+ \wedge \mathcal{H}(\mathcal{R}_{\Gamma,\tilde{\s}}) \;\hookrightarrow\; \mathcal{H}_{S^1 \times \Z_p}(\Gamma,\tilde{\s}).
\]
We claim that $f$ is an $(S^1 \times \Z_p)$-equivariant homotopy equivalence.

To see this, choose leaves $v,v'$ of $\mathcal{R}_{\Gamma,\tilde{\s}}$ such that $(v,v')$ forms a simple angle at some vertex $w$. The part of $\mathcal{H}(\mathcal{R}_{\Gamma,\tilde{\s}})$ corresponding to $v,v',w$ is
\[
\mathcal{H}(v,v') = 
\bigl(V_v^+ \vee V_{v'}^+ \vee (W_{(v,v')}^+ \wedge [0,1])\bigr) \big/ \sim,
\]
where $W_{(v,v')}^+ \times \{0\}$ and $W_{(v,v')}^+ \times \{1\}$ are attached to $V_v^+$ and $V_{v'}^+$ via inclusions.

On the other hand, in the $\Z_p$-equivariant $\mathrm{Spin}^c$ computation sequence
\[
\widetilde{\mathrm{sp}}_{\tilde{\s}}(x_{\s}(0)),
\widetilde{\mathrm{sp}}_{\tilde{\s}}(x^{\s}_{0,0}),\dots,\;
\widetilde{\mathrm{sp}}_{\tilde{\s}}(x^{\s}_{0,n_0-1}),
\widetilde{\mathrm{sp}}_{\tilde{\s}}(x_{\s}(1)),\dots
\]
used to define $\mathcal{H}_{S^1 \times \Z_p}(\Gamma,\tilde{\s})$, there exist indices $0 \le i < j < k$ such that the vertices $v,w,v'$ correspond respectively to
\[
\widetilde{\mathrm{sp}}_{\tilde{\s}}(x_{\s}(i)),
\qquad
\widetilde{\mathrm{sp}}_{\tilde{\s}}(x_{\s}(j)),
\qquad
\widetilde{\mathrm{sp}}_{\tilde{\s}}(x_{\s}(k)).
\]
By construction of $\mathcal{R}_{\Gamma,\tilde{\s}}$, the sequence 
\[
\bigl\{\chi_{\tilde{\s}}(x_{\s}(\ell))\bigr\}_{\ell \ge 0}
\]
is increasing for $i \le \ell \le j$ and decreasing for $j \le \ell \le k$. Moreover, all the representations occurring here are relative index
representations.  Since the first leaf corresponds to
$x_{\s}(0)=0$, we have $\lambda_V(v_0)=0$.  Thus
\[
\begin{aligned}
\ind_{\Z_p}^t
\dirac_{W_\Gamma,\widetilde{\mathrm{sp}}_{\tilde{\s}}(x_{\s}(i))}
-
\ind_{\Z_p}^t
\dirac_{W_\Gamma,\widetilde{\mathrm{sp}}_{\tilde{\s}}(x_{\s}(0))}
&= -\lambda_V(v),\\
\ind_{\Z_p}^t
\dirac_{W_\Gamma,\widetilde{\mathrm{sp}}_{\tilde{\s}}(x_{\s}(j))}
-
\ind_{\Z_p}^t
\dirac_{W_\Gamma,\widetilde{\mathrm{sp}}_{\tilde{\s}}(x_{\s}(0))}
&= -\lambda_V(v)-\lambda_A(v,v'),\\
\ind_{\Z_p}^t
\dirac_{W_\Gamma,\widetilde{\mathrm{sp}}_{\tilde{\s}}(x_{\s}(k))}
-
\ind_{\Z_p}^t
\dirac_{W_\Gamma,\widetilde{\mathrm{sp}}_{\tilde{\s}}(x_{\s}(0))}
&= -\lambda_V(v').
\end{aligned}
\]
For simplicity, denote the $m$th term in the subsequence from $i$ to
$j$ by $\tilde{\s}_m$, and let its length be $N$.  Define
\[
V_m=
\ind_{\Z_p}^t\dirac_{W_\Gamma,\tilde{\s}_m}
-
\ind_{\Z_p}^t
\dirac_{W_\Gamma,\widetilde{\mathrm{sp}}_{\tilde{\s}}(x_{\s}(0))}.
\]
Then we obtain a chain of inclusions
\[
W_{(v,v')} = V_N \;\hookrightarrow\; V_{N-1} \;\hookrightarrow\; \cdots \;\hookrightarrow\; V_1 = V_v,
\]
whose composition is precisely the inclusion $W_{(v,v')} \hookrightarrow V_v$.

Now, in $(\C^{-\lambda_V(v_0)})^+ \wedge \mathcal{H}(\Gamma,\tilde{\s})$, the piece corresponding to the subsequence from $i$ to $j$ is
\[
\bigl(V_1^+ \vee (V_2^+ \wedge [0,1]) \vee \cdots \vee (V_N^+ \wedge [0,1])\bigr)\big/ \sim,
\]
where $\sim$ attaches $V_{m+1}^+ \times \{0\}$ to $V_m^+$ for $1 \le m < N$. Up to $(S^1 \times \Z_p)$-equivariant homotopy equivalence, this simplifies to
\[
\bigl(V_1^+ \vee (V_N^+ \wedge [0,1])\bigr)\big/ \sim,
\]
where $\sim$ attaches $V_N^+ \times \{0\}$ directly to $V_1^+$ via the composition of inclusions, i.e.\ the inclusion $W_{(v,v')} \hookrightarrow V_v$. A similar argument applies to the subsequence between $j$ and $k$. 

Thus, the contribution of the computation sequence between $i$ and $k$ in $(\C^{-\lambda_V(v_0)})^+ \wedge \mathcal{H}_{S^1 \times \Z_p}(\Gamma,\tilde{\s})$ is
\[
\bigl(V_v^+ \vee V_{v'}^+ \vee (W_{(v,v')}^+ \wedge [0,1])\bigr)\big/ \sim,
\]
where $\sim$ attaches $W_{(v,v')}^+ \times \{0\}$ to $V_v^+$ and $W_{(v,v')}^+ \times \{1\}$ to $V_{v'}^+$ via inclusions. But this is exactly the corresponding part of $\mathcal{H}(\mathcal{R}_{\Gamma,\tilde{\s}})$. Applying this argument to all simple angles of $\mathcal{R}_{\Gamma,\tilde{\s}}$, we obtain an $(S^1 \times \Z_p)$-equivariant homotopy equivalence
\[
\mathcal{H}(\mathcal{R}_{\Gamma,\tilde{\s}}) \simeq \left(\C^{-\lambda_V(v_0)}\right)^+ \wedge \mathcal{H}_{S^1 \times \Z_p}(\Gamma,\tilde{\s}).
\]
This proves the claim, and the lemma follows from \Cref{thm: eqv lattice comparison map}.
\end{proof}

\subsection{The $\Z_p$-equivariant lattice chain} \label{subsec: Zp lattice chain}

In this subsection, we construct a chain model (in fact, a finite-cell
$A_\infty$-module model) for the $\Z_p$-equivariant lattice homotopy type $\mathcal{H}_{S^1 \times \Z_p}(\Gamma,\tilde{\s})$, which computes its $(S^1 \times \Z_p)$-equivariant cohomology. Recall that, for any topological space $X$, the singular cochain complex $C^\ast(X)$ is an $E_\infty$-algebra. Moreover, if a topological group $G$ acts on a topological space $X$, then the reduced equivariant singular cochain complex $\widetilde{C}^\ast_G(X)$ is naturally an $E_\infty$-module over $C^\ast_G(\ast) \cong C^\ast(BG)$. 
The ``lattice chain model'' constructed here will be quasi-isomorphic to
\[
\widetilde{C}^\ast_{S^1\times\Z_p}\!\left(\mathcal{H}_{S^1\times\Z_p}(\Gamma,\tilde{\s})\right),
\]
and hence, up to an overall grading shift, also to $C^\ast_{S^1\times\Z_p}\!\left(SWF_{S^1\times\Z_p}(Y,\tilde{\s})\right)$, as a $C^\ast(B(S^1\times\Z_p))$-module.

\begin{lem}\label{lem: chain level Mayer Vietoris}
Let $G$ be a topological group acting on a topological space $X$.
Suppose that $U,V\subset X$ are $G$-invariant open subsets with
$U\cup V=X$.  Then there is a quasi-isomorphism of
$C^\ast(BG)$-modules
\[
\widetilde{C}^\ast_G(X)
\simeq
\operatorname{hofib}\!\left(
\widetilde{C}^\ast_G(U)\oplus\widetilde{C}^\ast_G(V)
\xrightarrow{\,(i_U^\ast,-i_V^\ast)\,}
\widetilde{C}^\ast_G(U\cap V)
\right).
\]
\end{lem}

\begin{proof}
Apply the Borel construction to the invariant open cover
$X=U\cup V$.  After replacing $U$ and $V$ by slightly smaller
invariant open subsets if necessary, equivariant excision identifies
the resulting square of Borel constructions with a homotopy pushout.
Since singular cochains are contravariant, applying
$C^\ast(-)$ gives a homotopy pullback square of
$C^\ast(BG)$-modules.

The homotopy pullback of
\[
\widetilde{C}^\ast_G(U)
\longrightarrow
\widetilde{C}^\ast_G(U\cap V)
\longleftarrow
\widetilde{C}^\ast_G(V)
\]
is the homotopy fiber of the restriction difference
\[
(i_U^\ast,-i_V^\ast)\colon
\widetilde{C}^\ast_G(U)\oplus\widetilde{C}^\ast_G(V)
\longrightarrow
\widetilde{C}^\ast_G(U\cap V).
\]
This gives the claimed quasi-isomorphism of
$C^\ast(BG)$-modules.
\end{proof}

We use the following explicit $A_\infty$ models for
\[
\mathcal R^\circ_p=C^\ast(B(S^1\times\Z_p);\Z_p).
\]
For $p=2$, take
\[
\mathcal R_2=\Z_2[U,\theta],
\qquad \deg U=2,\quad\deg\theta=1,
\]
with $m_2$ ordinary multiplication and $m_j=0$ for $j\ne2$.  We write
$S=\theta^2$ when convenient.

For $p>2$, take
\[
\mathcal R_p=\Z_p[U,R,S]/(R^2),
\qquad
\deg U=\deg S=2,\quad\deg R=1,
\]
with $m_2$ ordinary multiplication and the only nonzero higher
multiplication
\[
m_p(a_1R,\dots,a_pR)=a_1\cdots a_p\,S,
\qquad a_i\in\Z_p[U,S].
\]
It vanishes on the other polynomial/exterior basis types, and
$m_j=0$ for $j\ne2,p$.  We use a fixed $A_\infty$
quasi-isomorphism
\[
\mathcal R^\circ_p\simeq\mathcal R_p.
\]
This follows from the usual product decomposition
\[
C^*(B(S^1\times\Z_p);\Z_p)
\simeq C^*(BS^1;\Z_p)\otimes_{\Z_p}C^*(B\Z_p;\Z_p).
\]
The first factor is formal as an associative algebra: a cocycle representing $U$ defines a quasi-isomorphism from the polynomial algebra. The same argument applies to $C^*(B\Z_2;\Z_2)$. For odd $p$, the $A_\infty$ structure on $C^*(B\Z_p;\Z_p)$ is computed in \cite[Theorem~1.3]{benson2021massey}, with Sylow order $p$ and inertial index one. After rescaling the degree-one generator to absorb its sign convention, this is the $S$-linear operation $m_p(R,\ldots,R)=S$ displayed above. Tensoring with the polynomial factor gives its $\Z_p[U,S]$-linear extension.

For $[k]\in\Z_p$, the shear
\[
\psi_{[k]}(U)=U+kS,\qquad
\psi_{[k]}(R)=R,\qquad
\psi_{[k]}(S)=S
\]
is an $A_\infty$ automorphism; at $p=2$ the same notation is used with
$S=\theta^2$.

\begin{rem}
Whenever commutativity is needed we work with the original
$E_\infty$-algebra $\mathcal R^\circ_p$; otherwise we use the explicit
$A_\infty$ model $\mathcal R_p$ through the chosen quasi-isomorphism.
\end{rem}

We then compute $\widetilde{C}^\ast_G(V^+)$ for various $(S^1 \times \Z_p)$-representations $V$.

\begin{lem}\label{lem: inclusion map pullbacks}
For every virtual complex $\Z_p$-representation
$\mathbf n\in\Z[\Z_p]$,
\[
\widetilde C^\ast_{S^1\times\Z_p}
\bigl((\C^{\mathbf n})^+;\Z_p\bigr)
\simeq
\mathcal R_p[-2|\mathbf n|].
\]
If $\mathbf n\ge0$ coefficientwise, the zero-section pullback
\[
S^0\longrightarrow(\C^{\mathbf n})^+
\]
is multiplication by
\[
U^{\mathbf n}
=
\prod_{[k]\in\Z_p}(U+kS)^{n_{[k]}}.
\]
More generally, if $\mathbf m,\mathbf n$ are actual representations
and $\mathbf n-\mathbf m\ge0$, the corresponding representation-sphere
inclusion has pullback $U^{\mathbf n-\mathbf m}$ after common
suspension.
\end{lem}

\begin{proof}
For actual $\mathbf n\ge0$, the associated Borel vector bundle
\[
\xi_{\mathbf n}
=
ES^1\times E\Z_p
\times_{S^1\times\Z_p}\C^{\mathbf n}
\longrightarrow B(S^1\times\Z_p)
\]
has real rank $2|\mathbf n|$.  The mod-$p$ Thom isomorphism gives
\[
\widetilde C^\ast_{S^1\times\Z_p}
\bigl((\C^{\mathbf n})^+;\Z_p\bigr)
\simeq
\mathcal R_p[-2|\mathbf n|].
\]
For a virtual $\mathbf n$, add an actual representation $\mathbf q$
so that both $\mathbf q$ and $\mathbf q+\mathbf n$ are actual, apply
the preceding statement, and desuspend.

For the weight-$[k]$ line,
\[
c_1(\C_{[k]})=U+kS.
\]
The zero-section pullback of its Thom class is multiplication by
$U+kS$.  Multiplicativity of Euler classes gives the formula for
coefficientwise nonnegative $\mathbf n$, and the final statement
follows after a common suspension.
\end{proof}

\begin{lem}\label{lem: homology determines chain map}
Let $\mathcal R^\circ_p=C^\ast(B(S^1\times\Z_p);\Z_p)$.
Evaluation at the unit induces a natural isomorphism
\[
\operatorname{Hom}_{D(\mathcal R^\circ_p)}
\bigl(\mathcal R^\circ_p,M[d]\bigr)
\cong
H^d(M)
\]
for every $\mathcal R^\circ_p$-module $M$.

In particular, if
\[
f\colon
\mathcal R^\circ_p
\longrightarrow
\mathcal R^\circ_p[2]
\]
is a module map whose value on the unit represents
$U+kS\in H^2(\mathcal R^\circ_p)$, then $f$ is homotopic,
as a one-sided module map, to multiplication by $U+kS$.
The same conclusion holds after transporting the map to the chosen
$A_\infty$ model $\mathcal R_p$.
\end{lem}

\begin{proof}
Since $\mathcal R^\circ_p$ is the regular module over itself,
\[
R\operatorname{Hom}_{\mathcal R^\circ_p}
(\mathcal R^\circ_p,M)
\simeq M.
\]
Taking degree $d$ cohomology gives
\[
\operatorname{Hom}_{D(\mathcal R^\circ_p)}
\bigl(\mathcal R^\circ_p,M[d]\bigr)
\cong H^d(M),
\]
and the isomorphism is evaluation at the unit.

For $M=\mathcal R^\circ_p$ and $d=2$, the class of a module map is
therefore determined by the image of $1$.  Hence a map sending $1$
to the class $U+kS$ is homotopic to multiplication by $U+kS$.
Transporting this statement along the chosen $A_\infty$
quasi-isomorphism with $\mathcal R_p$ gives the final assertion.
\end{proof}

We can now construct the $\Z_p$-equivariant lattice chain $\mathcal{C}^\ast_{S^1 \times \Z_p}(\Gamma,\mathfrak{s})$ as follows. Recall that, in \Cref{subsec: Zp lattice homotopy}, we considered the $\Z_p$-equivariant $\mathrm{Spin}^c$ computation sequence
\[
\widetilde{\mathrm{sp}}_{\tilde{\s}}(x_{\s}(0)),\,
\widetilde{\mathrm{sp}}_{\tilde{\s}}(x^{\s}_{0,0}),\,
\dots,\,
\widetilde{\mathrm{sp}}_{\tilde{\s}}(x^{\s}_{0,n_0-1}),\,
\widetilde{\mathrm{sp}}_{\tilde{\s}}(x_{\s}(1)),\,
\dots
\]
and rewrote it as
\[
\mathfrak{s}_1,\mathfrak{s}_2,\mathfrak{s}_3,\dots.
\]
We then used their $\Z_p$-equivariant Dirac indices to define virtual $S^1 \times \Z_p$-representations $V_i$ and $S^1 \times \Z_p$-equivariant virtual linear maps $f_i$, each either from $V_i$ to $V_{i+1}$ or from $V_{i+1}$ to $V_i$. We also defined $W_i$ as the domain of $f_i$, as well as integers $\Delta_i$, and when $\Delta_i \neq 0$, the nonnegative integer $k_i$. Using these data, we define the $\mathcal{R}_p$-module
\[
\mathcal{C}^\ast_{S^1 \times \Z_p}(\Gamma,\mathfrak{s}) = (C_V \oplus C_W,m_\ast)
\]
as follows:
\begin{itemize}
   \item $C_V = \bigoplus_i \mathcal{V}_{V_i}$, where
\[
\mathcal{V}_{V_i}
=
\mathcal{R}_p[-2\dim_\C V_i];
\]

\item $C_W = \bigoplus_i \mathcal{V}_{W_i}$, where
\[
\mathcal{V}_{W_i}
=
\mathcal{R}_p[-2\dim_\C W_i-1];
\]
    \item $m_1\vert_{C_W}=0$;
    \item For each $i$, the $\mathcal{V}_{W_i}$-component of
\[
m_1\colon
\mathcal{V}_{V_i}\oplus\mathcal{V}_{V_{i+1}}
\longrightarrow
C_W
\]
is determined by the following row:
\[
\begin{cases}
(1,-1), & \text{if } \Delta_i=0,\\
\bigl(1,-(U+k_iS)^{\Delta_i}\bigr),
    & \text{if } \Delta_i>0,\\
\bigl((U+k_iS)^{-\Delta_i},-1\bigr),
    & \text{if } \Delta_i<0.
\end{cases}
\]
    \item Each summand is a shifted regular $\mathcal R_p$-module. The full module is the homotopy fiber of the displayed incidence map from the vertex summands to the unshifted edge summands, with the usual suspension signs.
    Its remaining operations are those of this mapping-cone construction. The entries of the incidence map lie in $\Z_p[U,S]$ (in $\Z_2[U,\theta]$ when $p=2$), over which the chosen operations are linear.
\end{itemize}
Then we have the following theorem.

\begin{thm}\label{thm: lattice chain model is correct}
For each $[k]\in\Z_p$, consider the $A_\infty$-algebra automorphism $\psi_{[k]}$ of $\mathcal R_p$ defined, for $p>2$, by
\[
\psi_{[k]}(U)=U+kS,\qquad
\psi_{[k]}(R)=R,\qquad
\psi_{[k]}(S)=S.
\]
For $p=2$, it is defined by
\[
\psi_{[k]}(U)=U+k\theta^2,\qquad
\psi_{[k]}(\theta)=\theta,
\]
where $S=\theta^2$.\footnote{This corresponds to the pullback map $(B\varphi_{[k]})^\ast$, where $\varphi_{[k]}(x,[n])=(e^{2\pi i kn/p}x,[n])$ is the group automorphism displayed in \Cref{lem: eqv SWF of different eqv lifts}.} For $p>2$, the polynomial-linearity of $m_p$ implies that $\psi_{[k]}$ preserves all the $A_\infty$ operations; for $p=2$, it preserves multiplication.

Then, under the identification $C^\ast(B(S^1\times\Z_p);\Z_p)\simeq\mathcal{R}_p$ of $A_\infty$-algebras up to quasi-isomorphism, the chain complex $\mathcal{C}^\ast_{S^1\times\Z_p}(\Gamma,\tilde{\s})$ is quasi-isomorphic to $C^\ast_{S^1\times\Z_p}(SWF_{S^1\times\Z_p}(Y,\tilde{\s}))$ as an $\mathcal{R}_p$-module, after a degree shift and a reparametrization of $\mathcal{R}_p$ via $\psi_{[k]}$ for some $[k]\in\Z_p$.
\end{thm}

\begin{proof}
    By \Cref{lem: inclusion map pullbacks}, the attaching maps induce the displayed Euler polynomials on cohomology. Since their sources and targets are shifted regular modules, \Cref{lem: homology determines chain map} identifies their module homotopy classes with multiplication by these polynomials. Applying \Cref{lem: chain level Mayer Vietoris} to the vertex spheres and edge cylinders therefore gives a quasi-isomorphism
    \[
    C^\ast_{S^1 \times \Z_p}(\mathcal{H}_{S^1 \times \Z_p}(\Gamma,\tilde{\s});\Z_p) \longrightarrow \mathcal{C}^\ast_{S^1 \times \Z_p}(\Gamma,\tilde{\s}),
    \]
    since $\mathcal{H}_{S^1 \times \Z_p}(\Gamma,\tilde{\s})$ is constructed by gluing cylinders to spheres. The theorem then follows from \Cref{thm: eqv lattice comparison map}.
\end{proof}

\begin{rem} \label{rem: S1xZp cochain for different eqv lifts}
    Let $\mathfrak{s}$ be a $\mathrm{Spin}^c$ structure on $Y$. As shown in \Cref{lem: eqv SWF of different eqv lifts}, for any two $\Z_p$-equivariant lifts $\tilde{\s},\tilde{\s}' \in \mathrm{Spin}^c_{\Z_p}(Y)$ of $\mathfrak{s}$, there exists some $[k]\in \Z_p$ such that their $S^1 \times \Z_p$-equivariant Seiberg--Witten Floer spectra are related by the automorphism $\varphi_{[k]}$ displayed in \Cref{lem: eqv SWF of different eqv lifts}. Consequently, their cochain complexes
    \[
    \widetilde{C}^\ast_{S^1 \times \Z_p}(SWF_{S^1 \times \Z_p}(Y,\tilde{\s});\Z_p) 
    \qquad \text{ and } \qquad 
    \widetilde{C}^\ast_{S^1 \times \Z_p}(SWF_{S^1 \times \Z_p}(Y,\tilde{\s}');\Z_p)
    \]
    are related, as $\mathcal{R}_p$-modules, by the automorphism $\psi_{[k]}$ of $\mathcal{R}_p$.

    Thus, once we compute $\widetilde{C}^\ast_{S^1 \times \Z_p}(SWF_{S^1 \times \Z_p}(Y,\tilde{\s});\Z_p)$ for one $\Z_p$-equivariant lift $\tilde{\s}$ of $\mathfrak{s}$, we can obtain $\widetilde{C}^\ast_{S^1 \times \Z_p}(SWF_{S^1 \times \Z_p}(Y,\tilde{\s}');\Z_p)$ for any other lift $\tilde{\s}'$ of $\mathfrak{s}$ simply by replacing every occurrence of $U$ with $U+kS$.
\end{rem}

For simplicity, from now on we use the following notation: given an element $\mathbf{n} = \sum_{[k]\in \Z_p} n_{[k]} \cdot [k]$, we define
\[
U^\mathbf{n} := \prod_{[k]\in \Z_p} (U+kS)^{n_{[k]}}.
\]

\begin{rem}
    In the proof of \Cref{thm: lattice chain model is correct}, we apply \Cref{lem: chain level Mayer Vietoris} to glue cylinders to spheres. For clarity, we present here a simplified case in a more explicit form. Suppose that we have three complex virtual $(S^1 \times \Z_p)$-representations
    \[
    V_1 = \C^{\mathbf{m}+\mathbf{n}_1},\qquad V_2 = \C^{\mathbf{m}+\mathbf{n}_2},\qquad W = \C^\mathbf{m},
    \]
    where $\mathbf{n}_1,\mathbf{n}_2 \ge 0$. Consider the inclusions
    \[
    i_1\colon W \hookrightarrow V_1, \qquad i_2\colon W \hookrightarrow V_2.
    \]
    Define
    \[
    X = \bigl(V_1^+ \vee V_2^+ \vee (W^+ \wedge [0,1])\bigr)\big/\sim,
    \]
    where $(w,0)\sim i_1^+(w)$ and $(w,1)\sim i_2^+(w)$ for $w\in W^+$.
    By \Cref{lem: inclusion map pullbacks}, under the identifications
    \[
    \tilde{H}^\ast_{S^1 \times \Z_p}(V_1^+;\Z_p) \cong \tilde{H}^\ast_{S^1 \times \Z_p}(V_2^+;\Z_p) \cong \tilde{H}^\ast_{S^1 \times \Z_p}(W^+;\Z_p) \cong \mathcal{R}_p
    \]
    (up to suitable degree shifts), the pullback maps along $i_1$ and $i_2$ are given by
    \[
    (i_1^+)^\ast = U^{\mathbf{n}_1}, \qquad (i_2^+)^\ast = U^{\mathbf{n}_2}.
    \]
    Hence, applying \Cref{lem: chain level Mayer Vietoris} shows that $C^\ast_{S^1 \times \Z_p}(X;\Z_p)$ is quasi-isomorphic to the following module:
    \[
    \xymatrix{
    & \mathcal{R}_p & & \mathcal{R}_p \\
    \mathcal{R}_p \ar[ur]^{U^{\mathbf{n}_1}} & & \mathcal{R}_p \ar[ul]_{\mathrm{id}}\ar[ur]^{\mathrm{id}} & & \mathcal{R}_p \ar[ul]_{U^{\mathbf{n}_2}}
    }
    \]
   Observe that whenever we encounter a sequence of differentials of the form
\[
x \xrightarrow{\;\;a\;\;} y
\xleftarrow{\;\;1\;\;} z
\xrightarrow{\;\;b\;\;} w,
\]
we may eliminate the contractible pair determined by the unit entry.
With the present row-sign convention, replacing $x$ by $x-az$ gives
\[
m_1(x-az)=-ab\,w.
\]
Thus the reduced differential is multiplication by $-ab$.
At $p=2$ the sign is invisible; at odd primes it can alternatively be
absorbed by compatible sign changes of the generators.
Applying this simplification yields the following module:
    \[
    \xymatrix{
    & \mathcal{R}_p \\
    \mathcal{R}_p \ar[ur]^{U^{\mathbf{n}_1}} & & \mathcal{R}_p \ar[ul]_{U^{\mathbf{n}_2}},
    }
    \]
    which is precisely the part of $\mathcal{C}^\ast_{S^1 \times \Z_p}(\Gamma,\tilde{\s})$ that we wanted $C^\ast_{S^1 \times \Z_p}(X;\Z_p)$ to correspond to.
\end{rem}

We now describe how to read off $\mathcal{C}^\ast_{S^1 \times \Z_p}(\Gamma,\tilde{\s})$, up to quasi-isomorphism, from $\Z_p$-labelled planar graded roots. Given a $\Z_p$-labelled planar graded root $\mathcal{R} = (R,\lambda_V,\lambda_A)$, we define the $\mathcal{R}_p$-module 
\[
\mathcal{C}^\ast_{S^1 \times \Z_p}(\mathcal{R}) = (C_V \oplus C_A,m_\ast),
\]
where $C_V, C_A$, and $m_\ast$ are defined as follows:
\begin{itemize}
   \item
\[
C_V=
\bigoplus_{\text{leaf }v}
\mathcal{R}_p[2|\lambda_V(v)|].
\]

\item
\[
C_A=
\bigoplus_{\text{simple angle }(v,v')}
\mathcal{R}_p
\bigl[
2|\lambda_V(v)+\lambda_A(v,v')|-1
\bigr].
\]

\item $m_1\vert_{C_A}=0$.

\item For each simple angle $(v,v')$, the
$\mathcal{V}_{(v,v')}$-component of
\[
m_1\colon
\mathcal{V}_v\oplus\mathcal{V}_{v'}
\longrightarrow C_A
\]
is given by the row
\[
\left(
U^{\lambda_A(v,v')},
-
U^{\lambda_A(v,v')+\lambda_V(v)-\lambda_V(v')}
\right).
\]
    \item The remaining operations are given by the same shifted regular-module and mapping-cone convention as above.
\end{itemize}
Then the following lemma is immediate.

\begin{lem} \label{lem: lattice chain from label graded root}
    Under the chosen $A_\infty$ quasi-isomorphism with $\mathcal R_p$, we have
    \[
    \mathcal{C}^\ast_{S^1 \times \Z_p}(\mathcal{R}_{\Gamma,\tilde{\s}})
    \simeq C^\ast_{S^1 \times \Z_p}(SWF_{S^1 \times \Z_p}(Y,\tilde{\s})),
    \]
    up to a degree shift and a reparametrization of $\mathcal{R}_p$ via $\psi_{[k]}$ for some $[k]\in \Z_p$.
\end{lem}

\begin{proof}
    This follows directly from \Cref{lem: label gr root gives eqv lattice sp} and \Cref{thm: lattice chain model is correct}.
\end{proof}

\subsection{A sanity check: an explicit computation for $\Sigma(2,3,19)$} \label{subsec: sanity check computation}

Consider the Seifert fibered homology sphere $Y=\Sigma(2,3,19)$. Since it has only one $\mathrm{Spin}^c$ structure, denoted by $\mathfrak{s}_0$, we have $\mathfrak{s} = \mathfrak{s}^{\mathrm{can}}_Y$, i.e., the canonical $\mathrm{Spin}^c$ structure on $Y$ is $\mathfrak{s}_0$. Moreover, because $|H_1(Y;\Z)|=1$ is not divisible by any prime, $\mathfrak{s}_0$ has exactly $p$ equivariant lifts; we denote one of them by $\tilde{\mathfrak{s}}_0$. Using our method, we can compute its $S^1 \times \Z_p$-equivariant lattice model $\mathcal{H}_{S^1 \times \Z_p}(\Gamma,\tilde{\mathfrak{s}}_0)$. The star-shaped negative-definite almost rational plumbing graph $\Gamma$ satisfying $Y\cong \partial W_\Gamma$ is given as follows:

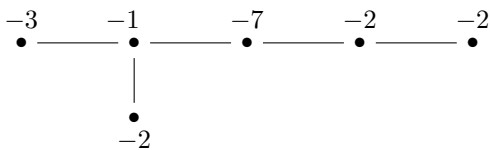
\[
\begin{tikzpicture}[xscale=1.5, yscale=1, baseline={(0,-0.1)}]
    \node at (-0.1, 0.3) {$-1$};
    \node at (-1, 0.3) {$-3$};
    
    \node at (0, -1.3) {$-2$};

    \node at (1, 0.3) {$-7$};
    \node at (2, 0.3) {$-2$};
    \node at (3, 0.3) {$-2$};
    
    \node at (0, 0) (A0) {$\bullet$};
    \node at (0, -1) (A1) {$\bullet$};

    \node at (-1, 0) (B1) {$\bullet$};
    
    \node at (1, 0) (C1) {$\bullet$};
    \node at (2, 0) (C2) {$\bullet$};
    \node at (3, 0) (C3) {$\bullet$};
    
    \draw (B1) -- (A0) -- (A1);
    \draw (A0) -- (C1) -- (C2) -- (C3);
\end{tikzpicture}
\]

In order to compute the $\Z_p$-labelled planar graded root by following the procedure described in \Cref{def: labelled gr root for Y and s}, it suffices to compute the delta-sequence for $\mathfrak{s}_0$. Since $\mathfrak{s}_0$ is the canonical $\mathrm{Spin}^c$ structure of $Y$, its delta-sequence can be computed very easily, as described in \Cref{rem: Delta seq for canonical spin c}. Since $N_Y = 13$, we have $\Delta_{Y,\mathfrak{s}_0}(i)\ge 0$ for all $i>13$, so we only need the values of $\Delta_{Y,\mathfrak{s}_0}(i)$ when they are nonzero. Thus it suffices to list their nonzero values up to $i=13$; these values are given below.

\begin{center}
\begin{tabular}{ |p{0.5cm}|p{1.25cm}||p{0.5cm}|p{1.25cm}||p{0.5cm}|p{1.25cm}||p{0.5cm}|p{1.25cm}||p{0.5cm}|p{1.25cm}||p{0.5cm}|p{1.25cm}|  }
 \hline
 $i$& $\Delta_{Y,\mathfrak{s}_0}(i)$ &$i$&$\Delta_{Y,\mathfrak{s}_0}$ &$i$&$\Delta_{Y,\mathfrak{s}_0}$ &$i$&$\Delta_{Y,\mathfrak{s}_0}$ &$i$&$\Delta_{Y,\mathfrak{s}_0}$ &$i$&$\Delta_{Y,\mathfrak{s}_0}$\\
 \hline
0 & 1 & 1 & $-1$ & 6 & 1 & 7 & $-1$ & 12 & 1 & 13 & $-1$ \\
 \hline
\end{tabular}
\end{center}

Now, using these values, we can construct the $\Z_p$-labelled planar graded root $\mathcal{R}_{\Gamma,\mathfrak{s}}$. First, the unlabeled planar graded root is as follows; its four leaves are denoted by $v_{-2},v_{-1},v_1,v_2$. Note that we draw planar graded roots upside down.

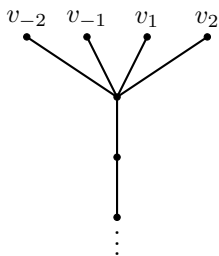
\begin{center}
\begin{tikzpicture}[scale=.8]
    \node at (-1.5,0.3) {$v_{-2}$};
    \draw [fill=black] (-1.5,0) circle (1.5pt);
    \node at (-0.5,0.3) {$v_{-1}$};
    \draw [fill=black] (-0.5,0) circle (1.5pt);
    \node at (0.5,0.3) {$v_1$};
    \draw [fill=black] (0.5,0) circle (1.5pt);
    \node at (1.5,0.3) {$v_2$};
    \draw [fill=black] (1.5,0) circle (1.5pt);
    \draw [fill=black] (0,-1) circle (1.5pt);
    \draw [fill=black] (0,-2) circle (1.5pt);
    \draw [fill=black] (0,-3) circle (1.5pt);
    \draw [fill=black] (0,-3.3) node {$\vdots$};
    \draw [thick] (0,-1)--(0,-2)--(0,-3);
    \draw [thick] (-0.5,0)--(0,-1)--(0.5,0);
    \draw [thick] (-1.5,0)--(0,-1)--(1.5,0);
\end{tikzpicture}
\end{center}Then the leaf labels and angle labels are given as follows. Note that $[0]$ and $0$ are distinct.

\begin{center}
\begin{tabular}{ |p{1cm}|p{0.5cm}|p{5cm}||p{2.2cm}|p{0.5cm}|p{0.7cm}|  }
 \hline
 leaves & $i$ & $\lambda_V$ & simple angles & $i$ & $\lambda_A$ \\
 \hline
$v_{-2}$ & 0 & 0 & $(v_{-2},v_{-1})$ & 1 & $[0]$ \\
$v_{-1}$ & 2 & $[0]-[1]$ & $(v_{-1},v_{1})$ & 7 & $[6]$ \\
$v_{1}$ & 8 & $[0]-[1]+[6]-[7]$ & $(v_{1},v_{2})$ & 13 & $[12]$ \\
$v_{2}$ & 14 & $[0]-[1]+[6]-[7]+[12]-[13]$ &  & & \\
 \hline
\end{tabular}
\end{center}
From this data, one can explicitly construct $\mathcal{H}(\Gamma,\tilde{\mathfrak{s}}_0)$ as follows. Consider the elements of $\Z[\Z_p]$:
\[
\begin{split}
    \mathbf{n}_{-2} &= [0]+[6]+[12],\quad \mathbf{n}_{-1} = [1]+[6]+[12],\quad \mathbf{n}_1 = [1]+[7]+[12],\quad \mathbf{n}_2 =[1]+[7]+[13], \\
    \mathbf{m}_{-1} &= [6]+[12],\quad \mathbf{m}_0 = [1]+[12],\quad \mathbf{m}_1 = [1]+[7].
\end{split}
\]
Then, up to suspension by a virtual representation sphere and reparametrization of $S^1 \times \Z_p$, we have
\[
\mathcal{H}_{S^1 \times \Z_p}(\Gamma,\tilde{\mathfrak{s}}_0) \simeq \mathrm{hocolim}\left[ \vcenter{ \xymatrix{
(\C^{\mathbf{n}_{-2}})^+ & (\C^{\mathbf{n}_{-1}})^+ &(\C^{\mathbf{n}_1})^+ & (\C^{\mathbf{n}_2})^+ \\ 
& (\C^{\mathbf{m}_{-1}})^+ \ar[lu]\ar[u] & (\C^{\mathbf{m}_0})^+ \ar[lu]\ar[u] & (\C^{\mathbf{m}_1})^+ \ar[lu]\ar[u]
} } \right],
\]where the arrows are the pointed maps induced by the canonical inclusions of $(S^1 \times \Z_p)$-representations. Observe that when $p=2$, the group ring elements $\mathbf{n}_i$ and $\mathbf{m}_j$ become
\[
\mathbf{n}_{-2} = 3[0],\quad \mathbf{n}_{-1} = 2[0]+[1],\quad \mathbf{n}_1 = [0]+2[1],\quad \mathbf{n}_2 = 3[1],\quad \mathbf{m}_{-1} = 2[0],\quad \mathbf{m}_0 = [0]+[1],\quad \mathbf{m}_1 = 2[1].
\]
Hence, taking the fixed point locus with respect to the action of the generator $\tau$ of $\Z_2=\{1\}\times\Z_2\subset S^1\times\Z_2$, we obtain
\[
\mathcal{H}_{S^1 \times \Z_2}(\Gamma,\tilde{\mathfrak{s}}_0)^\tau \simeq \mathrm{hocolim}\left[ \vcenter{ \xymatrix{
(\C^3)^+ & (\C^2)^+ &(\C^1)^+ & (\C^0)^+ \\ 
& (\C^2)^+ \ar[lu]\ar[u] & (\C^1)^+ \ar[lu]\ar[u] & (\C^0)^+ \ar[lu]\ar[u]
} } \right] \simeq (\C^3)^+.
\]
Similarly, taking the fixed point locus with respect to the action of $-\tau$ (that is, $(-1)\circ \tau$, where $-1$ denotes the unique element of order two in $S^1$) yields
\[
\mathcal{H}_{S^1 \times \Z_2}(\Gamma,\tilde{\mathfrak{s}}_0)^{-\tau} \simeq \mathrm{hocolim}\left[ \vcenter{ \xymatrix{
(\C^0)^+ & (\C^1)^+ &(\C^2)^+ & (\C^3)^+ \\ 
& (\C^0)^+ \ar[lu]\ar[u] & (\C^1)^+ \ar[lu]\ar[u] & (\C^2)^+ \ar[lu]\ar[u]
} } \right] \simeq (\C^3)^+.
\]In both cases, the geometric fixed-point calculation of the finite
lattice model gives a complex sphere.  It gives
the corresponding equivalence on $\Z_2$-Borel cohomology.  Applying
the ordinary localization theorem to the finite lattice model gives
\[
H^\ast_{\Z_2}
\bigl(SWF_{S^1\times\Z_2}(Y,\tilde{\mathfrak{s}}_0);\Z_2\bigr)
\otimes_{\Z_2[\theta]}\Z_2[\theta,\theta^{-1}]
\cong
\Z_2[\theta,\theta^{-1}].
\]
Thus the Borel cohomology has rank one over $\Z_2[\theta]$.  No
equivalence of geometric fixed-point spectra is asserted here.

\begin{rem} \label{rem: this computation is for all eqv spin c}
    Note that this computation is carried out for some $\Z_p$-equivariant $\mathrm{Spin}^c$ structure on $Y$. However, since $Y$ is a homology sphere (which admits a unique nonequivariant $\mathrm{Spin}^c$ structure) and $|H_1(Y;\Z)| = 1$ is not divisible by any prime, it follows from \Cref{cor: twisting on seifert QHS} that any two $\Z_p$-equivariant $\mathrm{Spin}^c$ structures are related by twisting. Therefore, the computations in this subsection apply to any $\Z_p$-equivariant $\mathrm{Spin}^c$ structure on $Y$.
\end{rem}

\subsection{The chain-level $(S^1 \times \Z_p)$-local equivalence group and Fr\o yshov invariants} \label{subsec: S1xZp local equivalence group}
Recall that for any $E_\infty$-algebra $A$ over a field, the derived category $\mathcal{D}(A)$ of $A$-modules is well defined. The \emph{perfect derived category} $\mathcal{D}^{\mathrm{perf}}(A)$ is defined as the closure of $A$ itself, viewed as an $A$-module, inside $\mathcal{D}(A)$ under degree shifts, extensions, and passing to direct summands. The objects of $\mathcal{D}^{\mathrm{perf}}(A)$ are called \emph{perfect $A$-modules}.

\begin{rem}\label{rem: finite G-CW-complexes are perfect}
For the group and coefficients used in this paper, namely
\[
G=S^1\times\Z_p,
\qquad
\mathbb F=\Z_p,
\]
the reduced Borel cochain complex of every finite $G$-CW complex is
perfect over
$
\mathcal R^\circ_p=C^\ast(BG;\Z_p).$

Indeed, consider the local system of reduced cochain complexes of the
fibers of
$
X_{hG}\longrightarrow BG.$
Since $X$ is a finite $G$-CW complex and $G$ is compact, its underlying
CW type is finite. Hence the cohomology of the fiber is
finite-dimensional and supported in a finite range.

Moreover,
$
\pi_1(BG)=\Z_p.$
Every finite-dimensional $\Z_p$-representation over $\Z_p$ admits a
finite filtration whose successive quotients are trivial
representations, since for a generator $\tau$ one has
$
(\tau-1)^p=0.$
Applying such filtrations to the cohomology local systems and then
using a finite Postnikov tower expresses the fiber-cochain local
system using finitely many shifts, cones, and retracts of the constant
local system.

Derived global sections preserve these finite constructions and send
the constant local system to $\mathcal R^\circ_p$. Therefore $
\widetilde C^\ast_G(X;\Z_p)$
is a perfect $\mathcal R^\circ_p$-module.
\end{rem}

\begin{lem}\label{lem: equivariant cochain products and restriction}
Put $G=S^1\times\Z_p$ and $A=C^*(BG;\Z_p)$. For finite based $G$-CW complexes $X$ and $Z$, the natural maps give quasi-isomorphisms
\[
\begin{split}
\widetilde C_G^*(X)\otimes_A^{\mathbb L}\widetilde C_G^*(Z)
&\simeq\widetilde C_G^*(X\wedge Z),\\
\widetilde C_G^*(X)\otimes_A^{\mathbb L}C^*(BS^1;\Z_p)
&\simeq\widetilde C_{S^1}^*(X).
\end{split}
\]
The second map is induced by the inclusion $S^1\subset G$.
\end{lem}
\begin{proof}
Use the local systems of fiber cochains considered in \Cref{rem: finite G-CW-complexes are perfect}. Their tensor product computes the reduced cochains of the fiberwise smash product, by the K\"unneth theorem over $\Z_p$. The product comparison is an equivalence when one local system is constant. Both sides preserve finite cones, shifts, and retracts in that variable, so the finite construction in the cited remark proves the first assertion. For the second assertion, pull back the local system along $BS^1\to BG$. The base-change comparison is again an equivalence for the constant local system and is preserved by the same finite operations. This proves the second assertion.
\end{proof}

We will use two distinct scalar extensions. First, forgetting the
$\Z_p$-factor is implemented by the derived scalar extension associated
to the augmentation map
\[
\mathcal{R}^\circ_p
\simeq
\Z_p[U]\otimes_{\Z_p} C^\ast(B\Z_p;\Z_p)
\xrightarrow{\ \mathrm{id}\otimes\epsilon\ }
\Z_p[U],
\]
where
\[
\epsilon\colon C^\ast(B\Z_p;\Z_p)\longrightarrow \Z_p
\]
is the augmentation. We then localize at $U$ via
\[
\Z_p[U]\longrightarrow \Z_p[U,U^{-1}].
\]
Thus the scalar extension used in the definition of local maps is the
composite
\[
\mathcal{R}^\circ_p
\longrightarrow
\Z_p[U]
\longrightarrow
\Z_p[U,U^{-1}].
\]
All scalar extensions from $\mathcal{R}^\circ_p$ are understood in the
derived sense. In these scalar extensions, $\Z_p[U]$ denotes the associative cochain model for $C^*(BS^1;\Z_p)$; whenever an $E_\infty$ structure is needed, we retain the one coming from singular cochains.

\begin{defn}
    An $\mathcal{R}^\circ_p$-module $M$ is said to be of \emph{weak SWF-type} if 
    \[
    M\otimes_{\mathcal{R}^\circ_p} \Z_p[U,U^{-1}] \simeq \Z_p[U,U^{-1}][n]
    \]
    as an ($A_\infty$) $\Z_p[U,U^{-1}]$-module for some $n\in\Z$.  

    Given two $\mathcal{R}^\circ_p$-modules $M,N$ of weak SWF-type, an $\mathcal{R}^\circ_p$-module map $f\colon M \to N$ is called \emph{local} if
    \[
    f\otimes \mathrm{id}\colon 
    M\otimes_{\mathcal{R}^\circ_p} \Z_p[U,U^{-1}] \longrightarrow 
    N\otimes_{\mathcal{R}^\circ_p} \Z_p[U,U^{-1}]
    \]
    is a quasi-isomorphism of (dg) $\Z_p[U,U^{-1}]$-modules.  

    Finally, two $\mathcal{R}^\circ_p$-modules $M,N$ of weak SWF-type are said to be \emph{weakly locally equivalent} if there exist local maps $M \to N[n]$ and $N \to M[m]$ for some integers $m,n$.
\end{defn}

Note that $\mathcal{R}_p$, regarded as an $\mathcal{R}^\circ_p$-module, is of weak SWF-type.

\begin{defn}
    An $\mathcal{R}^\circ_p$-module $M$ of weak SWF-type is said to be of \emph{SWF-type} if it is perfect and weakly locally equivalent to $\mathcal{R}^\circ_p$. Two $\mathcal{R}^\circ_p$-modules $M,N$ of SWF-type are \emph{locally equivalent} if there exist local maps $M \to N$ and $N \to M$. 
\end{defn}

\begin{lem} \label{lem: S1xZp chain Z-graded loc eqv gp is abelian}
    The following statements hold.
    \begin{enumerate}
        \item For any $\mathcal{R}^\circ_p$-modules $M,N$ of SWF-type, the tensor product $M\otimes_{\mathcal{R}^\circ_p} N$ is also of SWF-type and is quasi-isomorphic to $N\otimes_{\mathcal{R}^\circ_p} M$.
        \item Given an $\mathcal{R}^\circ_p$-module $M$ of SWF-type, its dual module $M^\vee$ (that is, the module such that the functor $M^\vee \otimes -$ is corepresented by $M$), which exists and is perfect by the perfectness of $M$ \cite[Proposition~7.2.4.4]{Lurie_HA}, is also of SWF-type. Moreover, $M\otimes_{\mathcal{R}^\circ_p} M^\vee$ is locally equivalent to $\mathcal{R}^\circ_p$.
    \end{enumerate}
\end{lem}

\begin{proof}
All tensor products and scalar extensions in this proof are derived.
Put
\[
L=\Z_p[U,U^{-1}].
\]
Perfect modules are dualizable, and tensor products of perfect
$\mathcal R^\circ_p$-modules are perfect.

If $M$ and $N$ are of SWF-type, then for some integers $a,b$,
\[
M\otimes_{\mathcal R^\circ_p}L\simeq L[a],
\qquad
N\otimes_{\mathcal R^\circ_p}L\simeq L[b].
\]
Since scalar extension is symmetric monoidal,
\[
(M\otimes_{\mathcal R^\circ_p}N)
\otimes_{\mathcal R^\circ_p}L
\simeq
L[a+b].
\]
Tensoring the local maps that exhibit $M$ and $N$ as weakly locally
equivalent to shifted regular modules gives local maps in both
directions between $M\otimes N$ and a shifted copy of
$\mathcal R^\circ_p$.  Hence $M\otimes N$ is of SWF-type, and symmetry
gives
\[
M\otimes_{\mathcal R^\circ_p}N
\simeq
N\otimes_{\mathcal R^\circ_p}M.
\]

If $M^\vee$ is the dual of $M$, then perfectness of $M$ implies
perfectness of $M^\vee$, and scalar extension preserves duals:
\[
M^\vee\otimes_{\mathcal R^\circ_p}L
\simeq
\left(M\otimes_{\mathcal R^\circ_p}L\right)^\vee.
\]
Thus
$M\otimes_{\mathcal R^\circ_p}L\simeq L[a]$ implies
\[
M^\vee\otimes_{\mathcal R^\circ_p}L\simeq L[-a].
\]
Dualizing the local maps for $M$ shows that $M^\vee$ is of SWF-type.

Finally, the evaluation and coevaluation maps
\[
M^\vee\otimes_{\mathcal R^\circ_p}M\to\mathcal R^\circ_p,
\qquad
\mathcal R^\circ_p\to
M\otimes_{\mathcal R^\circ_p}M^\vee
\]
become equivalences after scalar extension to $L$, because a shifted
rank-one Laurent module is invertible.  They are therefore local maps.
Hence $M^\vee\otimes_{\mathcal R^\circ_p}M$ is locally equivalent to
$\mathcal R^\circ_p$.
\end{proof}

We now define the chain-level $S^1 \times \Z_p$-local equivalence group $\mathfrak{C}^{ch}_{S^1 \times \Z_p}$. First, set
\[
\mathfrak{C}^{ch,\Z}_{S^1 \times \Z_p}  
= \frac{\left\{ \mathcal{R}^\circ_p\text{-modules of SWF-type}\right\}}
       {\text{local equivalence}},
\]
and endow it with the group operation given by tensor product. By \Cref{lem: S1xZp chain Z-graded loc eqv gp is abelian}, $\mathfrak{C}^{ch,\Z}_{S^1 \times \Z_p}$ is an abelian group. Moreover, $\Z$ acts on $\mathfrak{C}^{ch,\Z}_{S^1 \times \Z_p}$ by translation.

\begin{defn}
The \emph{chain-level $S^1\times\Z_p$-local equivalence group}
$\mathfrak C^{ch}_{S^1\times\Z_p}$ is the balanced quotient
\[
\mathfrak C^{ch}_{S^1\times\Z_p}
=
\bigl(
\mathfrak C^{ch,\Z}_{S^1\times\Z_p}\times\Q
\bigr)\big/\sim,
\]
where
\[
(M[n],r)\sim(M,r-n)
\qquad
\text{for every }n\in\Z.
\]

\end{defn}
It is clear that $\mathfrak{C}^{ch}_{S^1 \times \Z_p}$ is also an abelian group. With these formalisms, we can now show that the equivariant cochain functor induces a group homomorphism from the space-level local equivalence group to the chain-level local equivalence group.

\begin{lem} \label{lem: S1xZp eqv cochain is a group hom}
   The monoidal functor
   \[
   C^\ast_{S^1 \times \Z_p}(-;\Z_p)\colon 
   \mathcal{C}^{sp}_{S^1 \times \Z_p} \longrightarrow 
   \mathrm{Mod}^{op}_{C^\ast(B(S^1 \times \Z_p);\Z_p)}
   \]
   induces a group homomorphism
   \[
   C^\ast_{S^1 \times \Z_p}(-;\Z_p)\colon 
   \mathfrak{C}^{sp}_{S^1 \times \Z_p} \longrightarrow 
   \mathfrak{C}^{ch}_{S^1 \times \Z_p}.
   \]
   More explicitly, on rationally graded local-equivalence classes the
induced map is
\[
[(X,a,b)]
\longmapsto
\left[
\left(
\widetilde{C}^\ast_{S^1\times\Z_p}(X;\Z_p),
-\alpha_{\R}(a)-2\alpha_{\C}(b)
\right)
\right].
\]
\end{lem}

\begin{proof}
   By \Cref{lem: equivariant cochain products and restriction}, forgetting the $\Z_p$-action is the derived scalar extension used in the definition of local maps, and smash products give tensor products. Ordinary $S^1$-localization therefore shows that a space-level local map induces a cochain-level local map. Hence it suffices to prove that $\widetilde{C}^\ast_{S^1 \times \Z_p}(X;\Z_p)$ is an $\mathcal{R}^\circ_p$-module of SWF-type. 

Since $X$ is $(S^1\times\Z_p)$-equivariantly homotopy equivalent to a finite $(S^1\times\Z_p)$-CW complex, it follows from \Cref{rem: finite G-CW-complexes are perfect} that $\widetilde{C}^\ast_{S^1\times\Z_p}(X;\Z_p)$ is a perfect $C^\ast(B(S^1\times\Z_p);\Z_p)$-module (and thus an $\mathcal{R}^\circ_p$-module). Furthermore, the inclusion 
\[
i\colon X^{S^1} \hookrightarrow X
\]
induces an $\mathcal{R}^\circ_p$-module map
\[
i^\ast\colon \widetilde{C}^\ast_{S^1 \times \Z_p}(X;\Z_p) \longrightarrow \widetilde{C}^\ast_{S^1 \times \Z_p}(X^{S^1};\Z_p).
\]
Since $i$ itself is a local map, it follows that $i^\ast$ is also a local map. Moreover, because $X^{S^1}$ is homotopy equivalent to a sphere of some dimension $m\in \Z$, we have a Thom quasi-isomorphism 
\[
\widetilde{C}^\ast_{S^1 \times \Z_p}(X^{S^1};\Z_p) \simeq \mathcal{R}_p[-m].
\]
Hence the composition
\[
\widetilde{C}^\ast_{S^1 \times \Z_p}(X;\Z_p) 
\xrightarrow{i^\ast} 
\widetilde{C}^\ast_{S^1 \times \Z_p}(X^{S^1};\Z_p) 
\xrightarrow{\mathrm{Thom}} 
\mathcal{R}^\circ_p[-m]
\]
is a local map.

    Now let $(X^\vee,-r)$ be the additive inverse of $(X,r)$ in $\mathfrak{C}^{sp}_{S^1 \times \Z_p}$. By definition, there exists a (space-level) local map
\[
T_X\colon X \wedge X^\vee \longrightarrow S^0.
\]
Passing to equivariant cochains yields a local map
\[
T^\ast_X\colon \mathcal{R}_p \longrightarrow 
\widetilde{C}^\ast_{S^1 \times \Z_p}(X;\Z_p)\otimes_{\mathcal{R}^\circ_p} 
\widetilde{C}^\ast_{S^1 \times \Z_p}(X^\vee;\Z_p).
\]
Since $X^\vee$ is also a space of type $(S^1 \times \Z_p)$-SWF, we have already established the existence of a local map
\[
f\colon \widetilde{C}^\ast_{S^1 \times \Z_p}(X^\vee;\Z_p) \longrightarrow \mathcal{R}^\circ_p[n]
\]
for some $n\in \Z$. Hence the composition
\[
\mathcal{R}_p[-n] 
\xrightarrow{T^\ast_X[-n]} 
\widetilde{C}^\ast_{S^1 \times \Z_p}(X;\Z_p) \otimes_{\mathcal{R}^\circ_p} 
\widetilde{C}^\ast_{S^1 \times \Z_p}(X^\vee;\Z_p)[-n] 
\xrightarrow{\mathrm{id}\otimes f[-n]} 
\widetilde{C}^\ast_{S^1 \times \Z_p}(X;\Z_p)
\]
is a local map, as desired.
\end{proof}

While \Cref{lem: S1xZp eqv cochain is a group hom} appears quite natural and elementary, it has the following consequence, which may be of independent interest.

\begin{cor} \label{cor: dualities agree}
    For any space $X$ of type $(S^1 \times \Z_p)$-SWF and its additive inverse $X^\vee$ in $\mathfrak{C}^{sp}_{S^1 \times \Z_p}$, the $\mathcal{R}^\circ_p$-modules
    \[
    \widetilde{C}^\ast_{S^1 \times \Z_p}(X;\Z_p)^\vee 
    \qquad \text{ and } \qquad 
    \widetilde{C}^\ast_{S^1 \times \Z_p}(X^\vee;\Z_p)
    \]
    are locally equivalent.
\end{cor}

\begin{proof}
The space-level inverse $X^\vee$ represents the additive inverse of
$X$.  By \Cref{lem: S1xZp eqv cochain is a group hom}, equivariant
cochains carry this inverse to the additive inverse of
$\widetilde C^\ast_{S^1\times\Z_p}(X;\Z_p)$ in the chain-level local
equivalence group.  By
\Cref{lem: S1xZp chain Z-graded loc eqv gp is abelian}, that inverse
is represented by the dual perfect module.  Hence
\[
\widetilde C^\ast_{S^1\times\Z_p}(X^\vee;\Z_p)
\]
and
\[
\widetilde C^\ast_{S^1\times\Z_p}(X;\Z_p)^\vee
\]
are locally equivalent.
\end{proof}

We now define chain-level equivariant Fr\o yshov invariants. 

\begin{defn}
Given an element
\[
\mathcal X=(M,r)\in
\mathfrak C^{ch}_{S^1\times\Z_p},
\]
we define
\[
\delta^{(p)}_0(\mathcal X)
=
\frac12
\left(
r+
\min
\left\{
n\in\Z
\;\middle|\;
\text{there exists a local map }
\mathcal R^\circ_p\longrightarrow M[n]
\right\}
\right)
\]
and
\[
\delta(\mathcal X)
=
\frac12
\left(
r+
\min
\left\{
n\in\Z
\;\middle|\;
\begin{array}{c}
x\in
H^n\!\left(
M\otimes^{\mathbb L}_{\mathcal R^\circ_p}\Z_p[U]
\right),\\
U^k x\neq0\text{ for every }k>0
\end{array}
\right\}
\right).
\]
These define functions
\[
\delta,\delta^{(p)}_0\colon
\mathfrak C^{ch}_{S^1\times\Z_p}\longrightarrow\Q.
\]
\end{defn}

Both minima are finite.  The first set is nonempty by the definition
of SWF-type.  For the second, after derived augmentation to
$\Z_p[U]$, perfectness gives finitely generated graded cohomology,
while the weak SWF condition gives exactly one Laurent tower after
inverting $U$.  Hence a non-$U$-torsion class exists and its possible
degrees are bounded below.

\begin{lem} \label{lem: Froyshov can be defined on chain level}
    For any $X \in \mathfrak{C}^{sp}_{S^1 \times \Z_p}$ we have
    \[
    \delta(X) = \delta\bigl(\widetilde{C}^\ast_{S^1 \times \Z_p}(X;\Z_p)\bigr) \qquad \text{ and }
    \qquad  
    \delta^{(p)}_0(X) = \delta^{(p)}_0\bigl(\widetilde{C}^\ast_{S^1 \times \Z_p}(X;\Z_p)\bigr),
    \]
    where the invariant $\delta^{(p)}_0(X)$ is the equivariant Fr\o yshov invariant introduced in \cite{Baraglia-Hekmati:2024-1}; see also \Cref{section:relation_to_BH} for our formulation. 
\end{lem}

\begin{proof}
Set $
M=
\widetilde{C}^\ast_{S^1\times\Z_p}(X;\Z_p).$
By \Cref{lem: homology determines chain map}, a module map
\[
\mathcal R^\circ_p\longrightarrow M[n]
\]
is represented by a class $
x\in H^n(M)$.
After derived augmentation and localization, the target is a single
shifted copy of $\Z_p[U,U^{-1}]$.  Hence this map is local if and only
if the image of $x$ is a nonzero homogeneous Laurent monomial, since
these are precisely the homogeneous units of $\Z_p[U,U^{-1}]$.

On the other hand, by the base-change identification obtained by
forgetting the $\Z_p$-factor and by ordinary $S^1$-localization, this
condition is equivalent to requiring that the restriction of $x$ to
$X^{S^1}$ have nonzero leading ordinary $U$-term.  This is exactly the
$j=0$ condition in the definition of the equivariant Fr\o yshov
invariant.  At $p=2$ this means reduction modulo $(\theta)$, while at
odd $p$ it means setting $R=S=0$.

Taking the minimum degree therefore gives
$
\delta^{(p)}_0(X)
=
\delta^{(p)}_0(M)$.
The same base-change identification before inverting $U$ identifies
the ordinary non-$U$-torsion tower and proves $
\delta(X)=\delta(M).$
\end{proof}

Hence we obtain the following commutative diagram of functions; note, however, that $\delta^{(p)}_0$ is not a group homomorphism.
\[
\xymatrix{
\mathfrak{C}^{sp}_{S^1 \times \Z_p} \ar[rr]^{\widetilde{C}^\ast_{S^1 \times \Z_p}(-;\Z_p)} \ar[rrd]^{\delta,\delta^{(p)}_0} && \mathfrak{C}^{ch}_{S^1 \times \Z_p} \ar[d]^{\delta,\delta^{(p)}_0} \\ && \Q
}
\]

\subsection{Example: an explicit computation for $\Sigma(3,5,19)$} \label{subsec: explicit computation for 3 5 19}

Consider the Seifert fibered homology sphere $Y=\Sigma(3,5,19)$. It has only one $\mathrm{Spin}^c$ structure, which we denote by $\mathfrak{s}_0$. In particular, $\mathfrak{s}_0 = \mathfrak{s}^{can}_Y$, the canonical $\mathrm{Spin}^c$ structure on $Y$. The star-shaped negative-definite almost rational plumbing graph $\Gamma$ satisfying $Y \cong \partial W_\Gamma$ is given as follows:
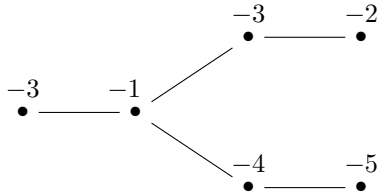
\[
\begin{tikzpicture}[xscale=1.5, yscale=1, baseline={(0,-0.1)}]
    \node at (-0.1, 0.3) {$-1$};
    \node at (-1, 0.3) {$-3$};
    
    \node at (1, 1.3) {$-3$};
    \node at (2, 1.3) {$-2$};

    \node at (1, -0.7) {$-4$};
    \node at (2, -0.7) {$-5$};
    
    \node at (0, 0) (A0) {$\bullet$};
    \node at (1, 1) (A1) {$\bullet$};
    \node at (2, 1) (A2) {$\bullet$};

    \node at (-1, 0) (B1) {$\bullet$};
    
    \node at (1, -1) (C1) {$\bullet$};
    \node at (2, -1) (C2) {$\bullet$};
    
    \draw (B1) -- (A0) -- (A1) -- (A2);
    \draw (A0) -- (C1) -- (C2);
\end{tikzpicture}
\]

We proceed as in \Cref{subsec: sanity check computation}. The nonzero values of the delta sequence $\Delta_{Y,\mathfrak{s}_0}(i)$ are listed below.
\begin{center}
\begin{tabular}{ |p{0.5cm}|p{1.25cm}||p{0.5cm}|p{1.25cm}||p{0.5cm}|p{1.25cm}||p{0.5cm}|p{1.25cm}|  }
 \hline
 $i$& $\Delta_{Y,\mathfrak{s}_0}(i)$ &$i$&$\Delta_{Y,\mathfrak{s}_0}(i)$&$i$&$\Delta_{Y,\mathfrak{s}_0}(i)$&$i$&$\Delta_{Y,\mathfrak{s}_0}(i)$\\
 \hline
0 & 1 & 30 & 1 & 72 & 1 & 105 & 1 \\
1 & $-1$ & 31 & $-1$ & 73 & $-1$ & 110 & 1 \\
4 & $-1$ & 43 & $-1$ & 75 & 1 & 114 & 1 \\
8 & $-1$ & 45 & 1 & 87 & 1 & 117 & 1 \\
13 & $-1$ & 46 & $-1$ & 88 & $-1$ & 118 & $-1$ \\
15 & 1 & 57 & 1 & 90 & 1 &  &  \\
16 & $-1$ & 58 & $-1$ & 95 & 1 & & \\
23 & $-1$ & 60 & 1 & 102 & 1 &  & \\
28 & $-1$ & 61 & $-1$ & 103 & $-1$ &  & \\
 \hline
\end{tabular}
\end{center}

The $\Z_p$-labelled planar graded root $\mathcal{R}_{\Gamma,\mathfrak{s}}$ is then given as follows (drawn upside down).

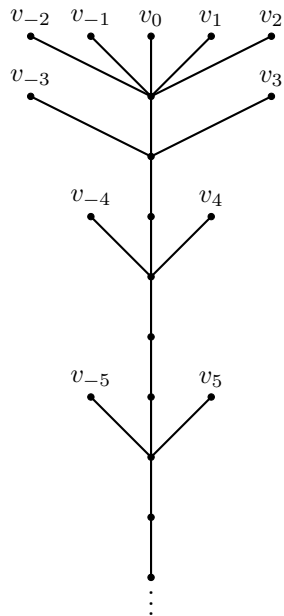
\begin{center}
\begin{tikzpicture}[scale=.8]
    \node at (0,0.3) {$v_0$};
    \draw [fill=black] (0,0) circle (1.5pt);
    \node at (-2,0.3) {$v_{-2}$};
    \draw [fill=black] (-2,0) circle (1.5pt);
    \node at (-1,0.3) {$v_{-1}$};
    \draw [fill=black] (-1,0) circle (1.5pt);
    \node at (1,0.3) {$v_1$};
    \draw [fill=black] (1,0) circle (1.5pt);
    \node at (2,0.3) {$v_2$};
    \draw [fill=black] (2,0) circle (1.5pt);
    \node at (-2,-0.7) {$v_{-3}$};
    \draw [fill=black] (-2,-1) circle (1.5pt);
    \draw [fill=black] (0,-1) circle (1.5pt);
    \node at (2,-0.7) {$v_3$};
    \draw [fill=black] (2,-1) circle (1.5pt);
    \draw [fill=black] (0,-2) circle (1.5pt);
    \node at (-1,-2.7) {$v_{-4}$};
    \draw [fill=black] (-1,-3) circle (1.5pt);
    \draw [fill=black] (0,-3) circle (1.5pt);
    \node at (1,-2.7) {$v_4$};
    \draw [fill=black] (1,-3) circle (1.5pt);
    \draw [fill=black] (0,-4) circle (1.5pt);
    \draw [fill=black] (0,-5) circle (1.5pt);
    \node at (-1,-5.7) {$v_{-5}$};
    \draw [fill=black] (-1,-6) circle (1.5pt);
    \draw [fill=black] (0,-6) circle (1.5pt);
    \node at (1,-5.7) {$v_5$};
    \draw [fill=black] (1,-6) circle (1.5pt);
    \draw [fill=black] (0,-7) circle (1.5pt);
    \draw [fill=black] (0,-8) circle (1.5pt);
    \draw [fill=black] (0,-9) circle (1.5pt);
    \draw [fill=black] (0,-9.3) node {$\vdots$};
    \draw [thick] (0,0)--(0,-1)--(0,-2)--(0,-3)--(0,-4)--(0,-5)--(0,-6)--(0,-7)--(0,-8)--(0,-9);
    \draw [thick] (-2,0)--(0,-1)--(2,0);
    \draw [thick] (-1,0)--(0,-1)--(1,0);
    \draw [thick] (-2,-1)--(0,-2)--(2,-1);
    \draw [thick] (-1,-3)--(0,-4)--(1,-3);
    \draw [thick] (-1,-6)--(0,-7)--(1,-6);
\end{tikzpicture}
\end{center}

The leaf labels and angle labels are given in \Cref{tab:labels3519}.

\begin{table}[!htb]
\centering
\begin{tabular}{ |p{1cm}|p{0.5cm}|p{8.8cm}||>{\raggedright\arraybackslash}p{1.5cm}|p{0.5cm}|>{\raggedright\arraybackslash}p{2cm}|  }
 \hline
 leaves & $i$ & $\lambda_V$ & simple angles & $i$ & $\lambda_A$ \\
 \hline
$v_{-5}$ & 0 & 0 & $(v_{-5},v_{-4})$ & 1 & $[0]$ \\
$v_{-4}$ & 14 & $[0]-[1]-[4]-[8]-[13]$ & $(v_{-4},v_{-3})$ & 16 & $[15]$ \\
$v_{-3}$ & 29 & $[0]-[1]-[4]-[8]-[13]+[15]-[16]-[23]-[28]$ & $(v_{-3},v_{-2})$ & 31 & $[30]$ \\
$v_{-2}$ & 44 & $[0]-[1]-[4]-[8]-[13]+[15]-[16]-[23]-[28]+[30]-[31]-[43]$ & $(v_{-2},v_{-1})$ & 46 & $[45]$ \\
$v_{-1}$ & 47 & $[0]-[1]-[4]-[8]-[13]+[15]-[16]-[23]-[28]+[30]-[31]-[43]+[45]-[46]$ & $(v_{-1},v_0)$ & 58 & $[57]$ \\
$v_0$ & 59 & $[0]-[1]-[4]-[8]-[13]+[15]-[16]-[23]-[28]+[30]-[31]-[43]+[45]-[46]+[57]-[58]$ & $(v_0,v_1)$ & 61 & $[60]$ \\
$v_1$ & 62 & $[0]-[1]-[4]-[8]-[13]+[15]-[16]-[23]-[28]+[30]-[31]-[43]+[45]-[46]+[57]-[58]+[60]-[61]$ & $(v_1,v_2)$ & 73 & $[72]$ \\
$v_2$ & 74 & $[0]-[1]-[4]-[8]-[13]+[15]-[16]-[23]-[28]+[30]-[31]-[43]+[45]-[46]+[57]-[58]+[60]-[61] + [72] - [73]$ & $(v_2,v_3)$ & 88 & $[75]+[87]$ \\
$v_3$ & 89 & $[0]-[1]-[4]-[8]-[13]+[15]-[16]-[23]-[28]+[30]-[31]-[43]+[45]-[46]+[57]-[58]+[60]-[61] + [72] - [73] + [75] + [87] - [88]$ & $(v_3,v_4)$ & 103 & $[90]+[95]+[102]$ \\
$v_4$ & 104 & $[0]-[1]-[4]-[8]-[13]+[15]-[16]-[23]-[28]+[30]-[31]-[43]+[45]-[46]+[57]-[58]+[60]-[61] + [72] - [73] + [75] + [87] - [88] + [90] + [95] + [102] - [103]$ & $(v_4,v_5)$ & 118 & $[105] + [110] + [114] + [117]$ \\
$v_5$ & 119 & $[0]-[1]-[4]-[8]-[13]+[15]-[16]-[23]-[28]+[30]-[31]-[43]+[45]-[46]+[57]-[58]+[60]-[61] + [72] - [73] + [75] + [87] - [88] + [90] + [95] + [102] - [103] + [105] + [110] + [114] + [117] - [118]$ & & & \\
 \hline
\end{tabular}
\caption{Leaf and angle labels for $\Sigma(3,5,19)$.}\label{tab:labels3519}
\end{table}

Using \Cref{thm: lattice chain model is correct}, we see that 
$C^\ast_{S^1 \times \Z_p}(\mathcal{H}_{S^1 \times \Z_p}(\Gamma,\mathfrak{s});\Z_p)$ 
is quasi-isomorphic to $\mathcal{C}^\ast_{S^1 \times \Z_p}(\Gamma,\mathfrak{s})$, 
which is generated over $\mathcal{R}_p$ by elements $x_i$ and $y_j$ with 
$-5 \le i,j \le 5$ and $j \neq 0$. The $A_\infty$-operations are inherited from 
$\mathcal{R}_p$, together with the following differentials (i.e., the $m_1$ operations):

\[
\begin{split}
    \partial x_0 &= (U+58S)y_{-1} + (U+60S)y_1, \\
    \partial x_1 &= (U+61S)y_1 + (U+72S)y_2, \\
    \partial x_{-1} &= (U+57S)y_{-1} + (U+46S)y_{-2}, \\
    \partial x_2 &= (U+73S)y_2 + (U+75S)(U+87S)y_3, \\
    \partial x_{-2} &= (U+45S)y_{-2} + (U+43S)(U+31S)y_{-3}, \\
    \partial x_3 &= (U+88S)y_3 + (U+90S)(U+95S)(U+102S)y_4, \\
    \partial x_{-3} &= (U+30S)y_{-3} + (U+28S)(U+23S)(U+16S)y_{-4}, \\
    \partial x_4 &= (U+103S)y_4 + (U+105S)(U+110S)(U+114S)(U+117S)y_5, \\
    \partial x_{-4} &= (U+15S)y_{-4} + (U+13S)(U+8S)(U+4S)(U+1S)y_{-5}, \\
    \partial x_5 &= (U+118S)y_5, \\
    \partial x_{-5} &= Uy_{-5}.
\end{split}
\]

We now consider local cohomology classes of 
$H^\ast_{S^1 \times \Z_p}(\mathcal{H}_{S^1 \times \Z_p}(\Gamma,\mathfrak{s});\Z_p)$ 
for various primes $p$. Before proceeding with computations, we define the notion of 
local homology classes; note that $\mathcal{H}_{S^1 \times \Z_p}(\Gamma,\mathfrak{s})$ 
satisfies the assumptions below.

\begin{defn}
    Let $M$ be an $\mathcal{R}_p$-module such that 
    $M\otimes_{\mathcal{R}_p} \Z_p[U,U^{-1}]$ is quasi-isomorphic to 
    $\Z_p[U,U^{-1}][r]$ for some $r\in \Q$. A homology class 
    $\alpha \in H^\ast(M)$ is called \emph{local} if its image under the map
    \[
    H^\ast(M)\longrightarrow 
    H^\ast(M\otimes_{\mathcal{R}_p} \Z_p[U,U^{-1}]) 
    \cong \Z_p[U,U^{-1}][r]
    \]
    generates $\Z_p[U,U^{-1}][r]$ as a module over $\Z_p[U,U^{-1}]$.
\end{defn}

For various primes $p$, we compute the minimal degree of a local class in the cohomology of this lattice model. This is twice the value of $\delta_0^{(p)}(\mathcal{H}_{S^1\times\Z_p}(\Gamma,\mathfrak{s}))$. Note, however, that $\mathcal{H}_{S^1 \times \Z_p}(\Gamma,\mathfrak{s})$ is defined only up to equivariant stable homotopy equivalence, so our computations are determined only up to an overall degree shift. To fix this ambiguity, we declare 
\[
\delta(\mathcal{H}_{S^1 \times \Z_p}(\Gamma,\mathfrak{s})) = 0,
\]
which amounts to setting $\deg(x_0)=0$.

The chain model
$\mathcal{C}^\ast(\Gamma,\mathfrak{s})$ carries nontrivial higher
$A_\infty$ operations inherited from $\mathcal R_p$.
For the homology calculations in this subsection, however, only the
unary operation $m_1$ enters the definition of cycles and boundaries.
Thus these homology computations may be carried out using the
underlying cochain complex $(\mathcal{C}^\ast,m_1)$; no formality
assumption on $\mathcal R_p$ is being made.

\begin{lem}\label{lem: diff of froyshov in lattice floer}
Choose any $\tilde{\mathfrak{s}}\in\mathrm{Spin}^c_{\Z_p}(Y)$ such that $\mathcal{N}(\tilde{\mathfrak{s}})=\mathfrak{s}$. Then
\[
\delta^{(p)}_0\bigl(\mathcal{C}^\ast_{S^1\times\Z_p}(\mathcal{R}_{\Gamma,\tilde{\mathfrak{s}}})\bigr)
-\delta\bigl(\mathcal{C}^\ast_{S^1\times\Z_p}(\mathcal{R}_{\Gamma,\tilde{\mathfrak{s}}})\bigr)
=\delta^{(p)}_0(Y,\mathfrak{s})-\delta(Y,\mathfrak{s}),
\]
where $\delta^{(p)}_0(Y,\mathfrak{s})$ denotes the $\Z_p$-equivariant Fr{\o}yshov invariant introduced in \cite{Baraglia-Hekmati:2024-1}, and $\delta(Y,\mathfrak{s})$ denotes the monopole Fr{\o}yshov invariant with $\Z_p$ coefficients.
\end{lem}

\begin{proof}
By \Cref{lem: lattice chain from label graded root}, the lattice cochain module is quasi-isomorphic to the equivariant Floer cochains of $(Y,\tilde{\mathfrak{s}})$, up to an overall grading shift and a reparametrization $U\mapsto U+kS$. The grading shift changes both invariants by the same amount, while the reparametrization preserves the grading and the augmentation to $\Z_p[U]$, and hence preserves both invariants. The statement therefore follows from \Cref{lem: Froyshov can be defined on chain level}.
\end{proof}

\begin{ex} \label{ex: p=2 case}
    Suppose $p=2$ (so that $S=\theta^2$). A minimal degree local homology class $\alpha$ is given by
    \[
    \alpha = \left[\begin{array}{l}(U+S)^2 x_0 + U(U+S)(x_1+x_{-1}) + U^2 (x_2+x_{-2}) + U(U+S)^2(x_3+x_{-3}) \\
    \quad + U^3(U+S)^2 (x_4+x_{-4}) + U^4(U+S)^4 (x_5+x_{-5})\end{array} \right].
    \]
    The displayed expression is a cycle by the equations at the $y_j$. It is local because after setting $\theta=0$ and inverting $U$, its $x_0$-coefficient is the unit $U^2$. To check minimality, write a homogeneous local cycle as $\sum a_i x_i+\sum b_j y_j$. The equations at $y_1$ and $y_2$ give
    \[
    U a_0+(U+S)a_1=0,
    U a_1+(U+S)a_2=0.
    \]
    Since $U$ and $U+\theta^2$ are relatively prime, $(U+S)^2$ divides $a_0$. The coefficient $a_0$ is nonzero for a local class, since it determines all vertex coefficients after localization. Thus every local class has degree at least four. With $\deg x_0=0$, the displayed class attains this bound. Therefore, by \Cref{lem: diff of froyshov in lattice floer}, we obtain
    \[
    \delta_0^{(2)}(Y)-\delta(Y)= \tfrac{1}{2}\deg \alpha = 2.
    \]
\end{ex}

\begin{ex} \label{ex: big p case}
    Suppose $p>118$, so that the elements $[0],[1],\dots,[118]\in \Z_p$ are pairwise distinct. A minimal degree local homology class $\alpha$ is given by
    \[
    \alpha = \left[ P(U,S)\cdot x_0 \;+\; \text{other terms involving }x_i\text{ for }i=-5,\dots,-1,1,\dots,5 \right],
    \]
    where the homogeneous polynomial $P \in \Z_p[U,S]$ is
    \[
    \begin{split}
    P(U,S) &= U(U+15S)(U+30S)(U+45S)(U+57S) \\
           &\quad (U+61S)(U+73S)(U+88S)(U+103S)(U+118S).
    \end{split}
    \]
    The remaining coefficients are determined successively by the cycle equation at each $y_j$: if $A_jx_i$ and $B_jx_{i+1}$ are the two contributions to that row, set $a_{i+1}=-A_ja_i/B_j$, starting with $a_0=P(U,S)$ and proceeding in both directions. For $p>118$, the distinct linear factors displayed above are relatively prime, and all these quotients are polynomials. Conversely, the same equations force the $x_0$-coefficient of every local cycle to be divisible by $P$. After setting $S=R=0$, its localization is nonzero. Thus the minimum degree is $\deg P=20$ when $\deg x_0=0$. Therefore, by \Cref{lem: diff of froyshov in lattice floer},
    \[
    \delta_0^{(p)}(Y)-\delta(Y)=\tfrac{1}{2}\deg \alpha = 10.
    \]
    Note that $10$ is also the dimension of $HF_{red}(Y,\mathfrak{s})$; in fact, this equality holds in a much more general sense, as we will see in \Cref{thm: Froyshov strict inequality}.
\end{ex}

\subsection{Behavior of $\delta^{(p)}_0$ for large primes $p$} \label{subsec: large p Froyshov}

In this subsection we study the behavior of the $\Z_p$-equivariant Fr\o yshov invariants $\delta^{(p)}_0(Y,\mathfrak{s})$, where $Y$ is a Seifert fibered rational homology sphere equipped with the Seifert $\Z_p$-action and $\mathfrak{s}$ is a $\mathrm{Spin}^c$ structure on $Y$. Since the value of $\delta^{(p)}_0$ is clearly invariant under twisting operations, and any two $\Z_p$-equivariant lifts of a given $\mathrm{Spin}^c$ structure on $Y$ are related by twisting by \Cref{cor: twisting on seifert QHS}, we will sometimes not care about the difference between $\Z_p$-equivariant $\mathrm{Spin}^c$ structures and nonequivariant $\mathrm{Spin}^c$ structures throughout this section.

\begin{lem}\label{lem: homology of graded root}
Let $k\ge0$ and let
\[
n_1^\pm,\dots,n_k^\pm\ge0.
\]
Consider the $\Z$-graded cochain complex $C$ freely generated over
$\Z_p[U]$, with $\deg U=2$, by
\[
x_0,\dots,x_k,\qquad y_1,\dots,y_k,
\]
with $
\partial y_j=0$,
and, when $k>0$,
\[
\partial x_0=U^{n_1^-}y_1,\qquad
\partial x_k=U^{n_k^+}y_k,
\]
\[
\partial x_i
=
U^{n_i^+}y_i+
U^{n_{i+1}^-}y_{i+1},
\qquad 1\le i<k.
\]
Suppose
\[
\deg x_\ell
=
\max_{0\le i\le k}\deg x_i.
\]
Then
\[
\dim_{\Z_p}H^\ast(C)_{\mathrm{tor}}
=
\sum_{i=1}^{\ell}n_i^-
+
\sum_{i=\ell+1}^{k}n_i^+.
\]
For $k=0$, both sums are empty and the torsion dimension is zero.
\end{lem}

\begin{proof}
We follow an induction on
\[
K_C=k+\sum_{i=1}^{\ell}n_i^-+\sum_{i=\ell+1}^k n_i^+.
\]
For $k=0$ the complex consists of a free tower, so the statement holds. Suppose $k>0$. Reversing the order of the generators interchanges the two ends of the complex and the signs in the superscripts, so we may use either endpoint in the following argument.

\smallskip
\noindent{\bf \underline{Case 1: one of $n_1^-$ and $n_k^+$ is zero.}}
Assume $n_1^-=0$. Then $\partial x_0=y_1$, and cancelling this pair leaves a complex $C'$ on $x_1,\ldots,x_k$ and $y_2,\ldots,y_k$, with the same remaining exponents. More explicitly, replace $x_1$ by $x_1-U^{n_1^+}x_0$ before cancelling. The grading relation gives
\[
\deg x_1-\deg x_0=2n_1^+\ge0.
\]
Thus a maximal-degree generator can be chosen among $x_1,\ldots,x_k$. If the chosen maximum was $x_0$, then $n_1^+=0$ and we can replace it by $x_1$ without changing the proposed formula. After reindexing, $K_{C'}=K_C-1$, and the exponent sum in the statement is unchanged. The induction hypothesis proves the assertion.

\smallskip
\noindent{\bf \underline{Case 2: $n_1^-$ and $n_k^+$ are positive.}}
After reversing the order if necessary, assume $\ell\ge1$. Let $C'$ be obtained by replacing $n_1^-$ by $n_1^--1$. Give its generators the same degrees except for $\deg x'_0=\deg x_0-2$. Then $x'_\ell$ still has maximal degree and $K_{C'}=K_C-1$.

There is an injective degree-preserving cochain map $f\colon C\to C'$ given by
\[
f(x_0)=Ux'_0,\qquad f(x_i)=x'_i\ (i>0),
\qquad f(y_j)=y'_j.
\]
The quotient is a one-dimensional vector space generated by the image of $x'_0$. No cocycle of $C'$ maps nontrivially to this quotient. Indeed, such a cocycle could be taken homogeneous of degree $\deg x'_0$, whereas $\deg x'_\ell>\deg x'_0$. Its $x'_\ell$-coefficient would therefore vanish, and the equations at the successive $y'_j$ force every $x'_i$-coefficient to vanish.

The connecting homomorphism in the cohomology sequence is consequently injective. All torsion is in the cokernel of the differential from the $x$-summands to the $y$-summands: this differential has full row rank over $\Z_p(U)$, and its kernel is free over $\Z_p[U]$. We therefore obtain
\[
0\longrightarrow\Z_p
\longrightarrow H^*(C)_{\mathrm{tor}}
\longrightarrow H^*(C')_{\mathrm{tor}}
\longrightarrow0.
\]
By induction, its middle term has dimension
\[
1+\left(\sum_{i=1}^{\ell}n_i^--1
+\sum_{i=\ell+1}^k n_i^+\right)
=\sum_{i=1}^{\ell}n_i^-+\sum_{i=\ell+1}^k n_i^+,
\]
as claimed.
\end{proof}

\begin{defn}
    For an element $\mathbf{n}\in \Z[\Z_p]$, written as $\mathbf{n}=\sum_{\alpha\in \Z_p} n_\alpha \cdot \alpha$, its \emph{support} is
    \[
    \mathrm{supp}(\mathbf{n})=\{\,\alpha\in \Z_p \mid n_\alpha\ne 0\,\}\subset \Z_p.
    \]
    Two elements $\mathbf{m},\mathbf{n}\in \Z[\Z_p]$ are said to be \emph{disjointly supported} if
    \[
    \mathrm{supp}(\mathbf{m})\cap \mathrm{supp}(\mathbf{n})=\emptyset.
    \]
\end{defn}

\begin{thm} \label{thm: Froyshov strict inequality}
    Let $Y$ be a Seifert fibered rational homology sphere, and let $\mathfrak{s}$ be a $\mathrm{Spin}^c$ structure on $Y$. Then, for all sufficiently large primes $p$, we have
    \[
    \delta^{(p)}_0(Y,\mathfrak{s}) = \delta(Y,\mathfrak{s}) + \dim_{\Z_p} HF_{\mathrm{red}}(Y,\mathfrak{s};\Z_p).
    \]
    In particular, if $\dim_{\Z_p}\widehat{HF}(Y,\mathfrak{s};\Z_p)>1$, then
    \[
    \delta^{(p)}_0(Y,\mathfrak{s}) > \delta(Y,\mathfrak{s})
    \]
    for all sufficiently large $p$.
\end{thm}

\begin{proof}
Consider the finite retained part of the original computation sequence
through its last negative increment.  Choose $p$ sufficiently large so
that
\[
p\nmid |H_1(Y;\Z)|
\]
and reduction modulo $p$ is injective on all integer labels occurring
in this finite portion.  Since distinct increasing and decreasing
blocks use disjoint sets of original labels, it follows that the
nonzero elements $\mathbf n_i^\pm\in\Z[\Z_p]$ appearing in the reduced
path diagram are pairwise disjointly supported.

If the reduced path consists of a single vertex, then the lattice
cochain module is a shift of the regular module.  In this case
$HF_{\mathrm{red}}(Y,\mathfrak{s})=0$ and the desired equality is
immediate.  Hence we assume from now on that the reduced path has at
least one edge.

The complex
$\mathcal{C}^\ast_{S^1\times\Z_p}(\Gamma,\mathfrak{s})$
then has the following differential diagram, where
\[
f_i^\pm
=
U^{\mathbf n_i^\pm}\cdot\mathrm{id}_{\mathcal R_p}
\]
for nonzero pairwise disjointly supported
$\mathbf n_i^\pm\ge0$:
\[
\vcenter{ \xymatrix{
\mathcal{R}_p y_1 & \mathcal{R}_p y_2 &  & \hspace{-.4cm}\cdots \hspace{.2cm} & \mathcal{R}_p y_{n-1} & \mathcal{R}_p y_n \\
\mathcal{R}_p x_0 \ar[u]^{f_1^-} & \mathcal{R}_p x_1 \ar[ul]_{f_1^+}\ar[u]_{f_2^-} & \mathcal{R}_p x_2 \ar[ul]_{f_2^+}\ar@{}[u]^(.1){}="a"^(.75){}="b" \ar "a";"b" & \hspace{-.4cm}\cdots \hspace{.2cm} & \mathcal{R}_p x_{n-2} \ar@{}[ul]^(.1){}="a"^(.75){}="b" \ar "a";"b" \ar[u]_{f_{n-1}^-} & \mathcal{R}_p x_{n-1} \ar[ul]_{f_{n-1}^+}\ar[u]_{f_n^-} & \mathcal{R}_p x_n \ar[ul]_{f_n^+}
} }.
\]
Choose $k\in \{0,\dots,n\}$ such that $\deg(x_k)=\max_{0\le j\le n}\deg(x_j)$; without loss of generality we perform a degree shift so that $\deg(x_k)=0$. Define
\[
N := \dim\, HF_{red}(Y,\mathfrak{s}).
\]
As in \Cref{ex: p=2 case,ex: big p case}, it suffices to prove:
\begin{itemize}
    \item there exists a local homology class of degree $2N$ in $H^\ast(\mathcal{C}^\ast_{S^1 \times \Z_p}(\Gamma,\mathfrak{s}))$;
    \item no local homology class of degree less than $2N$ exists.
\end{itemize}

For the homology calculation below, only the unary operation $m_1$
enters the cycle and boundary equations.  We therefore work with the
underlying cochain complex of
$\mathcal{C}^\ast_{S^1\times\Z_p}(\Gamma,\mathfrak{s})$.
Its coefficients may be written, on the level relevant to $m_1$, as
\[
H^\ast(\mathcal R_p)
=
\Z_p[U,R,S]/(R^2).
\]
No formality assumption on $\mathcal R_p$ is being made.

Suppose first that there exists a local homology class $[\alpha] \in H^\ast(\mathcal{C}^\ast_{S^1 \times \Z_p}(\Gamma,\s))$ with $\deg \alpha < 2N$. Let $\alpha$ be a cycle in $\mathcal{C}^\ast_{S^1 \times \Z_p}(\Gamma,\s)$ representing $[\alpha]$. We may write
\[
\alpha = \sum_{i=0}^n r_i x_i + \sum_{j=1}^n s_j y_j
\]
for some homogeneous elements $r_i,s_j \in \Z_p[U,R,S]/(R^2)$. After derived augmentation
\[
\mathcal R^\circ_p\longrightarrow\Z_p[U]
\]
and localization at $U$, every Euler factor occurring in the
differential becomes a unit.  Hence the localized differential onto
the edge summands generated by $y_1,\dots,y_n$ is surjective.
Consequently the $y_j$-components represent zero in the localized
homology.  Thus
\[
\alpha_0
=
\sum_{i=0}^n r_i x_i
\]
is again a local cycle of the same degree:
\[
\deg\alpha_0=\deg\alpha<2N.
\]

Now
\[
0 = \partial \alpha_0 = \sum_{j=1}^n \big( U^{\mathbf{n}_j^+} r_j + U^{\mathbf{n}_j^-} r_{j-1} \big) y_j,
\]
so $U^{\mathbf{n}_j^+} r_j + U^{\mathbf{n}_j^-} r_{j-1} = 0$ for all $j=1,\dots,n$. Consider the projection
\[
\mathrm{pr} \colon \Z_p[U,R,S]/(R^2) \longrightarrow \Z_p[U,S]; \qquad R \longmapsto 0.
\]
Then
\[
U^{\mathbf{n}_j^+} \mathrm{pr}(r_j) + U^{\mathbf{n}_j^-} \mathrm{pr}(r_{j-1}) = 0 \quad \text{for all } j=1,\dots,n,
\]
and hence
\[
U^{\sum_{i=1}^j \mathbf{n}_{k+i}^-} \cdot \mathrm{pr}(r_k) \;=\; (-1)^j U^{\sum_{i=1}^j \mathbf{n}_{k+i}^+} \cdot \mathrm{pr}(r_{k+j}) \quad \text{for } j=1,\dots,n-k.
\]

After setting $S=0$ and inverting $U$, all the coefficients in the
recurrence are units.  Its localized kernel is therefore free of rank
one, and any one of its $x_i$-coefficients determines all the others.
Since $\alpha_0$ is local, we must have $
\mathrm{pr}(r_k)\big|_{S=0}\neq0.$
In particular, $
\mathrm{pr}(r_k)\neq0$.
Since the $\mathbf{n}_i^\pm$ are pairwise disjointly supported, the monomials $U^{\sum_{i=1}^j \mathbf{n}_{k+i}^-}$ and $U^{\sum_{i=1}^j \mathbf{n}_{k+i}^+}$ are relatively prime in $\Z_p[U,S]$. Thus $U^{\sum_{i=1}^j \mathbf{n}_{k+i}^+}$ divides $\mathrm{pr}(r_k)$, in particular $U^{\mathbf{n}_{k+1}^+ + \cdots + \mathbf{n}_n^+}$ divides $\mathrm{pr}(r_k)$. A similar argument shows that $U^{\mathbf{n}_1^- + \cdots + \mathbf{n}_k^-}$ also divides $\mathrm{pr}(r_k)$. Since these factors are relatively prime, their product
\[
U^{\mathbf{n}_1^- + \cdots + \mathbf{n}_k^- + \mathbf{n}_{k+1}^+ + \cdots + \mathbf{n}_n^+}
\]
divides $\mathrm{pr}(r_k)$. Therefore
\[
\begin{split}
2N
>
\deg\alpha_0
=
\deg r_k
=
\deg\mathrm{pr}(r_k)
&\ge
\deg
U^{
\mathbf n_1^-+\cdots+\mathbf n_k^-
+\mathbf n_{k+1}^++\cdots+\mathbf n_n^+
}\\
&=
2\bigl(
|\mathbf n_1^-|+\cdots+|\mathbf n_k^-|
+|\mathbf n_{k+1}^+|+\cdots+|\mathbf n_n^+|
\bigr).
\end{split}
\]
which implies
\[
|\mathbf{n}_1^-| + \cdots + |\mathbf{n}_k^-| + |\mathbf{n}_{k+1}^+| + \cdots + |\mathbf{n}_n^+| < N.
\]
However, by \cite[Theorem~8.3]{nemethi2005ozsvath} and \Cref{lem: homology of graded root}, we have
\[
N = \dim HF_{red}(Y,\s) \;=\; |\mathbf{n}_1^-| + \cdots + |\mathbf{n}_k^-| + |\mathbf{n}_{k+1}^+| + \cdots + |\mathbf{n}_n^+|,
\]
a contradiction. Hence no local homology class of degree less than $2N$ can exist.

Finally, define
\[
\beta
=
\sum_{i=0}^n
(-1)^i
U^{
\left(\sum_{j=1}^i\mathbf n_j^-\right)
+
\left(\sum_{j=i+1}^n\mathbf n_j^+\right)
}
x_i.
\]
The coefficient of each $y_j$ in $\partial\beta$ is the sum of two
Euler products with the same total exponent and opposite signs.
Hence
\[
\partial\beta=0.
\]
By the grading relation, all summands of $\beta$ have the same degree.
The coefficient of $x_k$ has total complex degree
\[
|\mathbf n_1^-|+\cdots+|\mathbf n_k^-|
+
|\mathbf n_{k+1}^+|+\cdots+|\mathbf n_n^+|
=
N,
\]
and since $\deg x_k=0$, we obtain
\[
\deg\beta=2N.
\]

It remains to check that $\beta$ is local.  After augmentation and
localization, its $x_k$-coefficient is
\[
(-1)^k U^N,
\]
which is a unit in $\Z_p[U,U^{-1}]$.  The localized recurrence
determines all the other $x_i$-coefficients from this one, and there
are no boundaries with a nonzero $x$-component.  Hence $\beta$
generates the unique Laurent tower.  Therefore $[\beta]$ is local.

Thus the minimum degree of a local class is exactly $2N$, and hence
\[
\delta^{(p)}_0(Y,\mathfrak{s})
=
\delta(Y,\mathfrak{s})
+
\dim_{\Z_p}HF_{\mathrm{red}}(Y,\mathfrak{s};\Z_p),
\]
as desired.
\end{proof}

\begin{rem}
If $\s=\mathfrak{s}^{\mathrm{can}}_Y$, the bound in
\Cref{rem: Delta seq for canonical spin c} gives an explicit retained
range for the large-prime argument.  In particular, it is sufficient
to require
\[
p>N_Y
\qquad\text{and}\qquad
p\nmid |H_1(Y;\Z)|.
\]
Under these assumptions, reduction modulo $p$ is injective on all
integer labels occurring before the last negative increment, so the
support-disjointness condition used in the proof of
\Cref{thm: Froyshov strict inequality} holds.

For $Y=\Sigma(3,5,19)$ we have $|H_1(Y;\Z)|=1$ and $N_Y=118$, so this
reduces simply to $p>118$, in agreement with
\Cref{ex: big p case}.
\end{rem}

\begin{rem}
When $Y$ is a Seifert fibered homology sphere,
\cite[Proposition~3.6]{baraglia2024brieskorn} identifies
$\delta^{(p)}_0(Y)$ and $\delta^{(p)}_\infty(Y)$.
The argument above proves the displayed formula for all sufficiently
large primes $p$. 
\end{rem}

    \begin{figure}[h!]
\begin{center}
\begin{tikzpicture}[scale=.8]
    \draw [fill=black] (0,-2) circle (1.5pt);
    \draw [fill=black] (-1,-3) circle (1.5pt);
    \draw [fill=black] (0,-3) circle (1.5pt);
    \draw [fill=black] (1,-3) circle (1.5pt);
    \draw [fill=black] (0,-4) circle (1.5pt);
    \draw [fill=black] (0,-5) circle (1.5pt);
    \draw [fill=black] (-1,-6) circle (1.5pt);
    \draw [fill=black] (0,-6) circle (1.5pt);
    \draw [fill=black] (1,-6) circle (1.5pt);
    \draw [fill=black] (0,-7) circle (1.5pt);
    \draw [fill=black] (0,-8) circle (1.5pt);
    \draw [fill=black] (0,-9) circle (1.5pt);
    \draw [fill=black] (0,-9.3) node {$\vdots$};
    \draw [thick] (0,-2)--(0,-3)--(0,-4)--(0,-5)--(0,-6)--(0,-7)--(0,-8)--(0,-9);
    \draw [thick] (-1,-3)--(0,-4)--(1,-3);
    \draw [thick] (-1,-6)--(0,-7)--(1,-6);

    \draw [fill=black] (3,-3) circle (1.5pt);
    \draw [fill=black] (7,-3) circle (1.5pt);
    \draw [fill=black] (4,-4) circle (1.5pt);
    \draw [fill=black] (5,-4) circle (1.5pt);
    \draw [fill=black] (6,-4) circle (1.5pt);
    \draw [fill=black] (4,-5) circle (1.5pt);
    \draw [fill=black] (5,-5) circle (1.5pt);
    \draw [fill=black] (6,-5) circle (1.5pt);
    \draw [fill=black] (5,-6) circle (1.5pt);
    \draw [fill=black] (5,-7) circle (1.5pt);
    \draw [fill=black] (5,-8) circle (1.5pt);
    \draw [fill=black] (5,-8.3) node {$\vdots$};
    \draw [thick] (3,-3)--(4,-4)--(5,-5)--(6,-4)--(7,-3);
    \draw [thick] (4,-5)--(5,-6)--(6,-5);
    \draw [thick] (5,-4)--(5,-5)--(5,-6)--(5,-7)--(5,-8);
\end{tikzpicture}
\end{center}
\caption{The graded roots of $\Sigma(3,5,19)/\Z_2$ for the Seifert $\Z_2$-action: the canonical $\mathrm{Spin}^c$ structure on the left and the other $\mathrm{Spin}^c$ structure on the right.}
\label{fig: two graded roots}
    \end{figure}

\appendix

\section{Atiyah--Segal--Singer's equivariant index theorem for manifolds with boundary}
\label{appendix A}

\subsection{Equivariant index theorem}

We recall the equivariant Atiyah--Patodi--Singer index theorem for
$\mathrm{Spin}^c$ Dirac operators, following
\cite{donnelly1978eta}.  We also fix the local-coordinate conventions
used throughout the paper.

Let $X$ be a compact oriented smooth $4$--manifold, possibly with
boundary, equipped with a smooth $\Z_p$--action.  Suppose that
$\partial X=Y$
is a disjoint union of rational homology $3$--spheres and that the
$\Z_p$--action preserves each boundary component.  Let
$\widetilde{\s}$ be a $\Z_p$--equivariant $\mathrm{Spin}^c$ structure
on $X$.

Choose a $\Z_p$--invariant Riemannian metric which is a product near
the boundary, and a $\Z_p$--invariant $\mathrm{Spin}^c$ connection
$A_0$ which is flat near the boundary.  We write
\[
\dirac_{A_0}\colon
\Gamma(S^+)\longrightarrow\Gamma(S^-)
\]
for the corresponding Dirac operator.  Near the boundary it has the
product form
\[
\dirac_{A_0}
=
\frac{d}{dt}
+
\diracpartial_{B_0},
\]
where $B_0=A_0|_Y$.

With the convention used in this paper, the APS boundary condition is
defined using the nonpositive spectral subspace.  Thus
\[
\dirac_{A_0}+p_{(-\infty,0]}
\colon
\Gamma(S^+)
\longrightarrow
\Gamma(S^-)
\oplus
\bigoplus_{\lambda\leq0}V(\lambda)
\]
is Fredholm, and we put
\[
\ind_{\Z_p}^{\mathrm{APS}}\dirac_{A_0}
=
\Ker(\dirac_{A_0}+p_{(-\infty,0]})
-
\Cok(\dirac_{A_0}+p_{(-\infty,0]})
\in R(\Z_p).
\]
For $\gamma\in\Z_p$, its character value is denoted by
\[
\ind_{\gamma}^{\mathrm{APS}}\dirac_{A_0}
=
\operatorname{Tr}_{\gamma}
\left(
\ind_{\Z_p}^{\mathrm{APS}}\dirac_{A_0}
\right).
\]

The reduced equivariant eta correction is
\[
\overline{\eta}_{\gamma}
(\diracpartial_{B_0})
=
\frac{
\eta_{\gamma}(\diracpartial_{B_0})
-
c_{\gamma}(\diracpartial_{B_0})
}{2},
\]
where $
c_{\gamma}(\diracpartial_{B_0})
=
\operatorname{Tr}
\left(
\gamma|
_{\ker\diracpartial_{B_0}}
\right)$. For a nonidentity element $\gamma$, Donnelly's equivariant APS theorem
gives
\[
\ind_{\gamma}^{\mathrm{APS}}\dirac_{A_0}
=
\int_{X^\gamma}
(-1)^{\frac{\dim X^\gamma}{2}}
\frac{
\operatorname{ch}_{\gamma}
\bigl(j^\ast(S^+-S^-)\bigr)
\operatorname{td}(TX^\gamma\otimes\C)
}{
e(TX^\gamma)
\operatorname{ch}_{\gamma}
(\Lambda^{-1}N\otimes\C)
}
+
\overline{\eta}_{\gamma}
(\diracpartial_{B_0}),
\]
where $j\colon X^\gamma\hookrightarrow X$ is the inclusion and $N$ is
the normal bundle of the fixed-point set.

\subsubsection*{Local $\mathrm{Spin}^c$ coordinates}

We now fix the local coordinates used below.  Set
\[
\widetilde T
=
\frac{U(1)\times U(1)\times U(1)}
{\{\pm(1,1,1)\}}.
\]
There is a natural homomorphism
\[
\pi\colon
\widetilde T
\longrightarrow
SO(2)\times SO(2),
\qquad
[(x,y,z)]
\longmapsto
(xy,xy^{-1}).
\]
Its kernel is the central $U(1)$ corresponding to the $\mathrm{Spin}^c$ gauge factor.

The positive and negative spinor representations restricted to
$\widetilde T$ are
\[
\rho_+([(x,y,z)])
=
\begin{pmatrix}
xz&0\\
0&x^{-1}z
\end{pmatrix},
\qquad
\rho_-([(x,y,z)])
=
\begin{pmatrix}
yz&0\\
0&y^{-1}z
\end{pmatrix}.
\]

Suppose first that $q$ is an isolated fixed point of a nontrivial
element $\gamma$, and choose oriented complex coordinates such that
the tangent action is
\[
(z,w)
\longmapsto
(\zeta_p^{u}z,\zeta_p^{v}w),
\qquad
0<u,v<p.
\]
Every equivariant lift is obtained, after choosing these coordinates,
in the form
\begin{equation}
\label{eqn: eqv number in codim 4A}
[(x,y,z)]
\longmapsto
\left[
\left(
\zeta_{2p}^{u+v}x,\,
\zeta_{2p}^{u-v}y,\,
\zeta_{2p}^{u+v+2m}z
\right)
\right],
\qquad
m\in\Z_p.
\end{equation}
We call $m$ the \emph{equivariance number} at $q$ and write $
n_{\mathrm{eqv}}^q(\s)=m$.

The induced characters on the two spinor fibers are
\[
S^+_q
\cong
\C_{[u+v+m]}
\oplus
\C_{[m]},
\qquad
S^-_q
\cong
\C_{[u+m]}
\oplus
\C_{[v+m]}.
\]

Now suppose that $\Sigma$ is an oriented fixed surface and that the
action on its oriented normal bundle has weight $u$, so that locally
\[
(z,w)\longmapsto(z,\zeta_p^u w).
\]
The corresponding local lift is
\begin{equation}
\label{eqn: eqv number in codim 2A}
[(x,y,z)]
\longmapsto
\left[
\left(
\zeta_{2p}^{u}x,\,
\zeta_{2p}^{-u}y,\,
\zeta_{2p}^{u+2m}z
\right)
\right],
\qquad
m\in\Z_p.
\end{equation}
Again we write
\[
n_{\mathrm{eqv}}^\Sigma(\s)=m.
\]

In both \eqref{eqn: eqv number in codim 4A} and
\eqref{eqn: eqv number in codim 2A}, the $p$th powers of the three
multipliers represent the same central sign.  Hence the induced map
has order $p$ in $\widetilde T$.  The projections under $\pi$ have
exactly the prescribed tangent weights.

\subsubsection*{The local coefficients}

For $u,v\in\Z_p\smallsetminus\{0\}$, define
\[
R(u,v)
=
\frac{1}{
(1-\zeta_p^{-u})(1-\zeta_p^{-v})
}, \qquad
S(u)
=
\frac{1}{
2(1-\zeta_p^{-u})
},
\qquad \text{ and } \qquad 
T(u)
=
-
\frac{
1+\zeta_p^{-u}
}{
2(1-\zeta_p^{-u})^2
}.
\]

Equivalently,
\[
R(u,v)
=
-\frac14
\zeta_{2p}^{u+v}
\csc\frac{u\pi}{p}
\csc\frac{v\pi}{p}, \qquad 
S(u)
=
-\frac{i}{4}
\zeta_{2p}^{u}
\csc\frac{u\pi}{p},
\qquad \text{ and }\qquad 
T(u)
=
\frac14
\zeta_{2p}^{u}
\csc\frac{u\pi}{p}
\cot\frac{u\pi}{p}.
\]
We now state the fixed-point formula in these conventions.

\begin{thm}
\label{even_equivariant_index}
Let $X$ be as above and let $\gamma\in\Z_p$ be nontrivial.
Suppose $
X^\gamma
=
X^\gamma_0
\sqcup
X^\gamma_2$,
where $
X^\gamma_0
=
\{q_1,\dots,q_r\}$
and
$
X^\gamma_2
=
\Sigma_1\sqcup\cdots\sqcup\Sigma_s$
is a union of closed oriented surfaces.

At $q_i$, let $(u_i,v_i)$ be the two tangent weights and let
$
m_i
=
n_{\mathrm{eqv}}^{q_i}(\s)$.
On $\Sigma_j$, let $u_j$ be the oriented normal weight and let
$
m_j
=
n_{\mathrm{eqv}}^{\Sigma_j}(\s).$
Then
\[
\begin{split}
\ind_{\gamma}^{\mathrm{APS}}\dirac_{A_0}
={}&
\overline{\eta}_{\gamma}
(\diracpartial_{B_0})
+
\sum_{i=1}^{r}
\zeta_p^{m_i}
R(u_i,v_i)
+
\sum_{j=1}^{s}
\zeta_p^{m_j}
\left(
S(u_j)
\left\langle
c_1(\s),[\Sigma_j]
\right\rangle
+
T(u_j)
[\Sigma_j]^2
\right).
\end{split}
\]
\end{thm}

\begin{proof}
At an isolated fixed point, the spinor-character numerator obtained
from the local lift above is
\[
\zeta_p^m
(1-\zeta_p^u)
(1-\zeta_p^v).
\]
The normal denominator is
\[
(1-\zeta_p^u)
(1-\zeta_p^{-u})
(1-\zeta_p^v)
(1-\zeta_p^{-v}).
\]
Cancelling the common factors gives $
\zeta_p^m R(u,v)$.

For a fixed surface $\Sigma$, put
\[
c=\widetilde F_{A_0^t},
\qquad
d=\widetilde F_N,
\qquad
t=c_1(T\Sigma).
\]
Here $
\widetilde F_{A_0^t}
=
\frac{1}{2\pi i}F_{A_0^t}$
and similarly for $\widetilde F_N$.
The determinant condition shows that the twisting-line curvature is
$
\frac{c-t-d}{2}$.
The degree-two part of the fixed-point density is therefore
\[
\zeta_p^m
\left[
\frac{
e^{(c-t-d)/2}
\operatorname{td}(T\Sigma)
}{
1-\zeta_p^{-u}e^{-d}
}
\right]_{[2]}.
\]
Since $
\operatorname{td}(T\Sigma)
=
1+\frac{t}{2}$ 
in degrees at most two, we have
\[
e^{(c-t-d)/2}
\operatorname{td}(T\Sigma)
=
1+\frac{c-d}{2}.
\]
Moreover,
\[
\frac{1}{
1-\zeta_p^{-u}e^{-d}
}
=
\frac{1}{1-\zeta_p^{-u}}
-
\frac{
\zeta_p^{-u}
}{
(1-\zeta_p^{-u})^2
}
d
\]
up to degree two.  Hence the degree-two contribution is
\[
\zeta_p^m
\left(
\frac{c}{
2(1-\zeta_p^{-u})
}
-
\frac{
1+\zeta_p^{-u}
}{
2(1-\zeta_p^{-u})^2
}
d
\right),
\]
which is
\[
\zeta_p^m
\bigl(
S(u)c+T(u)d
\bigr).
\]
Integrating over $\Sigma$ proves the surface term.
\end{proof}

\begin{rem}
The theorem is a statement about the character value
$
\ind_{\gamma}^{\mathrm{APS}}\dirac_{A_0}\in\C$,
not directly about the virtual representation
$
\ind_{\Z_p}^{\mathrm{APS}}\dirac_{A_0}\in R(\Z_p).$
The latter is recovered from all character values by Fourier
inversion.
\end{rem}

\begin{rem}
The formula is independent of the auxiliary choices of orientations
used in describing the local coordinates.
For example, reversing the orientation of a fixed surface reverses
the normal orientation.  Thus
\[
u\longmapsto -u,
\qquad
m\longmapsto m+u.
\]
Directly from the definitions,
\[
\zeta_p^{m+u}S(-u)
=
-\zeta_p^m S(u),
\]
whereas
\[
\zeta_p^{m+u}T(-u)
=
\zeta_p^m T(u).
\]
The $c_1$--pairing changes sign and the self-intersection does not,
so the total contribution is unchanged.

Similarly, reversing both oriented planes at an isolated fixed point
changes
\[
(u,v,m)
\longmapsto
(-u,-v,m+u+v),
\]
and
\[
\zeta_p^{m+u+v}R(-u,-v)
=
\zeta_p^m R(u,v).
\]
\end{rem}

\begin{rem}
\label{rem: index thm for nonclosed fixed points}
Suppose now that a fixed surface $\Sigma$ is properly embedded and
may have nonempty boundary in $Y$. Since $Y$ is a disjoint union of rational homology spheres, the map
\[
H^2(X,Y;\Q)
\longrightarrow
H^2(X;\Q)
\]
is an isomorphism.  Thus $c_1(\s)$ has a unique relative rational
lift.  If $A_0^t$ is flat near the boundary, then
$\widetilde F_{A_0^t}$ represents this relative class, and we define
\[
\left\langle
c_1(\s),[\Sigma]
\right\rangle
=
\int_\Sigma
\widetilde F_{A_0^t}.
\]
Flatness near the boundary is all that is needed here; the boundary
holonomy need not be trivial.

Likewise, let $[\Sigma]^2$ denote the relative normal Euler number
with respect to the chosen rational boundary framing.  Then
\[
\int_\Sigma
\widetilde F_N
=
[\Sigma]^2
-
\frac{1}{2\pi}
\sum_{K\subset\partial\Sigma}
t(Y,K,g),
\]
where $t(Y,K,g)$ is the torsion term used in the main text.

Consequently, for properly embedded oriented fixed surfaces the
equivariant APS formula is
\[
\begin{split}
\ind_{\gamma}^{\mathrm{APS}}\dirac_{A_0}
={}&
\overline{\eta}_{\gamma}
(\diracpartial_{B_0})
+
\sum_i
\zeta_p^{m_i}R(u_i,v_i)
+
\sum_j
\zeta_p^{m_j}
\left(
S(u_j)
\left\langle
c_1(\s),[\Sigma_j]
\right\rangle
+
T(u_j)
[\Sigma_j]^2
\right)
\\
&-
\frac{1}{2\pi}
\sum_j
\sum_{K\subset\partial\Sigma_j}
\zeta_p^{m_j}T(u_j)\,
t(Y,K,g).
\end{split}
\]
\end{rem}

With the framing convention used in the main text, reversing the
orientation of a fixed circle while keeping the chosen normal framing
gives
\[
t(-K,g,\alpha)
=
t(K,g,\alpha).
\]
Indeed, both the connection component $\theta_{23}$ and the
orientation of integration change sign.

\subsection{Equivariant spectral flow}

In this subsection, we review the two equivariant spectral flows
\[
\operatorname{Sf}^k(\diracpartial_{B_0}(g_s))\in\C
\qquad\text{and}\qquad
\operatorname{Sf}(\diracpartial_{B_0}(g_s))\in R(\Z_p)
\]
for a $\Z_p$-equivariant $\mathrm{Spin}^c$ rational homology $3$-sphere $(Y,\mathfrak{s})$ equipped with a smooth path of $\Z_p$-invariant Riemannian metrics $\{g_s\}_{s\in[0,1]}$. Here $B_0$ is a fixed invariant flat connection. See \cite{lim2024equivariant} for details.

We regard $D_s=\diracpartial_{B_0}(g_s)$ as a path of self-adjoint equivariant Fredholm operators with discrete spectrum. Choose a subdivision $0=s_0<s_1<\cdots<s_N=1$ and positive numbers $a_1,\ldots,a_N$ such that $a_i$ is not an eigenvalue of $D_s$ for any $s\in[s_{i-1},s_i]$. Such choices exist after refining the subdivision. Put
\[
E_i(s)=\bigoplus_{0<\lambda<a_i}V_\lambda(D_s),
\qquad s\in[s_{i-1},s_i].
\]
Thus zero eigenvalues are placed on the nonpositive side.

\begin{defn}
The $\Z_p$-equivariant spectral flows are defined by
\[
\operatorname{Sf}(\diracpartial_{B_0}(g_s))
:=\sum_{i=1}^N\bigl([E_i(s_i)]-[E_i(s_{i-1})]\bigr)
\in R(\Z_p)
\]
and
\[
\operatorname{Sf}^k(\diracpartial_{B_0}(g_s))
:=\operatorname{Tr}_{[k]}\bigl(\operatorname{Sf}(\diracpartial_{B_0}(g_s))\bigr)
\in\C.
\]
Equivalently,
\[
\operatorname{Sf}(\diracpartial_{B_0}(g_s))
=\frac1p\sum_{l=0}^{p-1}
\left(\sum_{k=0}^{p-1}\operatorname{Sf}^k(\diracpartial_{B_0}(g_s))\zeta_p^{-kl}\right)
[\C_{[l]}].
\]
\end{defn}

These definitions are independent of the subdivision and spectral cuts. Refining a time interval makes the terms telescope. Changing $a_i$ adds the same finite-dimensional equivariant spectral subspace at both endpoints, so the two contributions cancel. A crossing from negative to positive carried by a representation $E$ contributes $+[E]$. A constant path has spectral flow zero, including when the operators have nontrivial kernel. The value at $k=0$ is the ordinary spectral flow with this endpoint convention.

As in the non-equivariant case, spectral flow can also be described by the Dirac index on the cylinder $([0,1]\times Y,\pi^*\mathfrak{s},dt^2+g_s)$. With the complementary endpoint APS conditions corresponding to the nonpositive convention above, we have
\[
\ind_{\Z_p}^{\mathrm{APS}}
\dirac_{[0,1]\times Y,\pi^*\mathfrak{s},\pi^*B_0}
=\operatorname{Sf}(\diracpartial_{B_0}(g_s)).
\]
Taking character values gives
\[
\ind_{[k]}^{\mathrm{APS}}
\dirac_{[0,1]\times Y,\pi^*\mathfrak{s},\pi^*B_0}
=\operatorname{Sf}^k(\diracpartial_{B_0}(g_s))
\qquad([k]\in\Z_p).
\]

\bibliographystyle{alpha}
\bibliography{tex}

\end{document}